\documentclass[10pt,reqno]{amsart}
\usepackage{graphicx, amsmath, amsthm, amsfonts, enumerate, amssymb, mathrsfs, color}
\usepackage{fullpage, cases}
\usepackage{bbm}

\def\dntau{\partial^{(\tau)}_n}

\def\soft{\text{\rm soft}}
\def\stiff{\text{\rm stiff}}

\def\eff{\text{\rm eff}}
\def\hom{\text{\rm hom}}

\def\Port{\mathcal P_{\bot}^{(\tau)}}
\def\P{\mathcal P^{(\tau)}}
\def\Porto{\mathcal P_{\bot}^{(0)}}
\def\Po{\mathcal P^{(0)}}

\def\e{\varepsilon}

\DeclareMathOperator*{\dom}{\mathrm{dom}}

\DeclareMathOperator*{\ran}{\mathrm{ran}}

\newtheorem{theorem}{Theorem}[section]
\newtheorem{proposition}[theorem]{Proposition}

\newtheorem{corollary}[theorem]{Corollary}

\newtheorem{definition}[theorem]{Definition}
\newtheorem{lemma}[theorem]{Lemma}

\newtheorem{remark}[theorem]{Remark}

\usepackage[pdftex]{hyperref}

\begin{document}
\title[Critical-contrast homogenisation of PDEs]{Effective behaviour of critical-contrast PDE{\small s}: micro-resonances, frequency conversion, and time dispersive properties. I.}
\author{Kirill D. Cherednichenko}
\address{Department of Mathematical Sciences, University of Bath, Claverton Down, Bath, BA2 7AY, United Kingdom}
\email{cherednichenkokd@gmail.com}
\author{Yulia Yu. Ershova}
\address{Department of Mathematical Sciences, University of Bath, Claverton Down, Bath, BA2 7AY, United Kingdom {\sc and} Department of Mathematics, St.\,Petersburg State University of Architecture and Civil Engineering, 2-ya Krasnoarmeiskaya St. 4, 190005 St.\,Petersburg, Russia} 
\email{julija.ershova@gmail.com}
\author{Alexander V. Kiselev}
\address{Departamento de F\'{i}sica Matem\'{a}tica, Instituto de Investigaciones en Matem\'aticas Aplicadas y en Sistemas, Universidad Nacional Aut\'onoma de M\'exico, C.P. 04510, M\'exico D.F. {\sc and} International Research Laboratory ``Multiscale Model Reduction'', Ammosov North-Eastern Federal University, Yakutsk, Russia}
\email{alexander.v.kiselev@gmail.com}


\begin{abstract}
A novel approach to critical-contrast homogenisation for periodic PDEs is proposed, via an explicit asymptotic analysis of Dirichlet-to-Neumann operators. Norm-resolvent asymptotics for non-uniformly elliptic problems with highly oscillating coefficients are explicitly constructed. An essential feature of the new technique is that it relates homogenisation limits to a class of time-dispersive media.
\end{abstract}

\maketitle

\par{\raggedleft\slshape To the memory of Professor Vasily Zhikov\par}

\section{Introduction}
The research aimed at modelling and engineering metamaterials has been recently brought to the forefront of materials science (see, {\it e.g.}, \cite{Phys_book} and references therein). It is widely acknowledged that these novel materials acquire non-classical properties as a result of a careful design of the microstructure, 
which can be assumed periodic with a small enough periodicity cell. The mathematical machinery involved in their modelling must therefore include as its backbone the theory of homogenisation (see 
{\it e.g.} \cite{Lions, Bakhvalov_Panasenko, Jikov_book}), which aims at characterising limiting, or ``effective'', properties of small-period composites. A typical problem here is to study the asymptotic behaviour of solutions to equations of the type
\begin{equation}
-{\rm div}\bigl(A^\varepsilon(\cdot/\varepsilon)\nabla u_\e\bigr)-\omega^2u_\e=f,\ \ \ \ f\in L^2({\mathbb R}^d),\quad d\ge2,\qquad \omega^2\notin{\mathbb R}_+,
\label{eq:generic_hom}
\end{equation}
where for all $\varepsilon>0$ the matrix $A^\varepsilon$ is $Q$-periodic, $Q:=[0,1)^d,$ non-negative symmetric, and may additionally be required to satisfy the condition of uniform ellipticity.

On the other hand, the result sought ({\it i.e.}, the ``metamaterial" behaviour in the limit of vanishing $\e$) belongs to the domain of the so-called time-dispersive media (see, {\it e.g.}, \cite{Tip_1998,Figotin_Schenker_2005,Tip_2006,Figotin_Schenker_2007b}). For such media, in the frequency domain one faces a setup of the type
\begin{equation*}
-{\rm div}\bigl(A\nabla u\bigr)+\mathfrak B(\omega^2)u=f,\ \ \ \ f\in L^2({\mathbb R}^d),\
\end{equation*}
where $A$ is a constant matrix and $\mathfrak B(\omega^2)$ is a frequency-dependent operator in $L^2({\mathbb R}^d)$ taking the place of $-\omega^2$ in (\ref{eq:generic_hom}), if, for the sake of argument, in the time domain we started with an equation of second order in time.
If, in addition, the matrix function $\mathfrak B$ is scalar, {\it i.e.}, $\mathfrak B(\omega^2)=\beta(\omega^2)I$ with a scalar function $\beta$, the problem of the type
\begin{equation}
-{\rm div}\bigl(A(\omega^2)\nabla u\bigr)=\omega^2 u
\label{eq:generic_spectral_td}
\end{equation}
appears in place of the spectral problem after a formal division by $-\beta(\omega^2)/\omega^2$, with frequency-dependent (but independent of the spatial variable) matrix $A(\omega^2).$

Thus, the matrix elements of $A(\omega^2)$, interpreted as material parameters of the medium, acquire a non-trivial dependence on the frequency, which may lead to their taking negative values in certain  frequency intervals. The latter property is, in turn, characteristic of metamaterials \cite{Veselago}. It is therefore of paramount interest to understand how inhomogeneity in the spatial variable (see \eqref{eq:generic_hom}) can lead, in the limit $\e\to0,$ to frequency dispersion, and, in particular, to uncover the conditions on $A^\e$ sufficient for this. A result, which from the above perspective can be seen as negative, is provided by the homogenisation theory in the uniformly strongly elliptic setting ({\it i.e.}, both $A^\e$ and $(A^\e)^{-1}$ are bounded uniformly in $\varepsilon$). Here one proves (see \cite{Zhikov_1989, BirmanSuslina} and references therein) the existence of a constant matrix $A^\hom$ such that solutions $u_\e$ to \eqref{eq:generic_hom} converge to $u_\hom$ satisfying
\begin{equation*}
-{\rm div}\bigl(A^\hom \nabla u_\hom\bigr)-\omega^2u_\hom=f,
\end{equation*}
which leaves no room for time dispersion. This negative result also carries over to vector models, including the Maxwell system.

If the uniform ellipticity assumption is dropped, the asymptotic analysis of the problems (\ref{eq:generic_hom}) becomes more complicated. By employing the technique of two-scale convergence, first Zhikov \cite{Zhikov2000, Zhikov2005} then Bouchitt\'{e}, Bourel, and Felbacq \cite{BouchitteFelbacq, BouchitteBourelFelbacq2009, BouchitteBourelFelbacq2017} obtained a related effective problem of the form \eqref{eq:generic_spectral_td}. The former works treat the critical-contrast model of the type \eqref{eq:generic_hom}, while the latter are devoted to an associated scattering problem.  Here, by ``critical contrast" one means that the components of a composite dielectric medium have material properties in a proper contrast to each other, governed by the size of periodicity cell (see Section 2 for further details). More recently, Kamotski and Smyshlyaev \cite{KamSm2018} developed a general two-scale compactness argument for the analysis of ``degenerate'' homogenisation problems, of which critical-contrast problem (\ref{eq:generic_hom}) is a particular case.  In a parallel development, prompted by \cite{Pendry_et_al}, the two-scale convergence approach has been used \cite{KohnShipman, BouchitteSchweizer, LamaczSchweizer2013, LamaczSchweizer2016} to mathematically justify the emergence of artificial magnetism in the case of ``resonant" metallic inclusions, {\it i.e.} in the case when the conductivity is high and scaled appropriately with the microstructure size, with the ``split-ring'' geometry of the inclusions. In yet another development, operator-theoretic techniques were used by Hempel and Lienau \cite{HempelLienau_2000} to prove that the limit spectrum in the case of critical-contrast periodic media has a band-gap structure, see also \cite{Friedlander} for a result concerning the asymptotics of the integrated density of states as well as \cite{CooperKamotskiSmyshlyaev} for the asymptotic spectral analysis of periodic Maxwell problems with the wavenumber-frequency pair situated in the vicinity of the lowest of the light-lines of the material components.

Although well received, the above results fall short of establishing a rigorous one-to-one correspondence between homogenisation limits for critical-contrast media and time dispersion in the effective medium. This is due to the following: (i) the additional assumptions imposed only permit to treat a limited set of models (curiously, excluding even the one-dimensional version of the problem, {\it cf.} \cite{CherednichenkoCooperGuenneau, CCC});
(ii) the control of the convergence error, for the solution sequences as well as the spectra of the problems, is lacking, due to the rather weak convergence of solutions claimed. A more general theory, akin to that of Birman and Suslina \cite{BirmanSuslina,BirmanSuslina_corr} in the moderate-contrast case, is therefore required. The present paper attempts to suggest such a theory.

The benefits of the novel unified approach as developed henceforth are these, in a nutshell:
\begin{enumerate}
  \item Being free from additional assumptions on the geometry and PDE type, it can be successfully applied in a consistent way to diverse problems motivated by applications;
  \item It can be viewed as a natural (albeit non-trivial) generalisation of the approach of Birman and Suslina in the uniformly elliptic case;
  \item The analysis is shown to be reducible by purely analytical means to an auxiliary uniformly elliptic problem; the latter, unlike the original problem, is within the reach of robust numerical techniques;
  \item The error bounds are controlled uniformly via norm-resolvent estimates (yielding the spectral convergence as a by-product);
  \item Not only is the relation of the composite to the corresponding effective time-dispersive medium made transparent (showing the artificial introduction of second (``fast'') variable via the two-scale asymptotics to be unnecessary from the technical point of view), but the approach can be also seen to offer a recipe for the construction of such media with prescribed dispersive properties from periodic composites whose individual components are non-dispersive.
\end{enumerate}

The analytical toolbox we propose also allows us\footnote{This argument will appear in a separate publication.}: (i) to explicitly construct spectral representations and functional models for both homogenisation limits of critical-contrast composites and the related time-dispersive models; (ii) on this basis, to solve direct and inverse scattering problems in both setups. It thus paves the way to treating the inverse problem of constructing a metamaterial ``on demand'', based on its desired properties. We shall therefore reiterate that the present paper can be seen as an
example of how surprisingly far one can reach by a consistent application of the existing vast toolbox of abstract spectral theory.

In \cite{Physics,GrandePreuve} (see also an earlier paper \cite{CherKis} dealing with critical-contrast homogenisation on ${\mathbb R}$) we considered a rather simple model of a high-contrast graph periodic along one axis. A unified treatment of critical-contrast homogenisation was proposed and carried out in three distinct cases: (i) where neither the soft nor the stiff component of the medium is connected; (ii) where the stiff component of the medium is connected; (iii) where the soft component of the medium is connected.

In the present paper we turn our attention to the PDE setup, and we focus on the scalar case, leaving the treatment of vector problems, in particular the full Maxwell system of electromagnetism as well as the systems of two- and three- dimensional elasticity, to a future publication. In view of keeping technicalities to a bare minimum, and at the same time making the substance of the argument as transparent as possible, we consider two  classical models, see Section 2 for details. The main ingredients of the theory, which is formulated in abstract, yet easily applicable, terms, remain virtually unchanged in more general models. One such generalisation is briefly discussed in Section 6.

The analytical backbone of our approach is the so-called generalised resolvent, or in other words the resolvent of the original operator family sandwiched by orthogonal projections to one of components of the medium (``soft'' one, see Section 2 for details). In its analysis, we draw our motivation from the celebrated general theory due to Neumark \cite{Naimark1940,Naimark1943} and the follow-up work by Strauss \cite{Strauss}. An explicit analysis of Dirichlet-to-Neumann (DN) maps (separately for the components comprising the medium) becomes necessary to facilitate the use of the well-known Kre\u\i n resolvent formula. The corresponding analysis is based on a version of Birman-Kre\u\i n-Vi\v sik theory \cite{Birman, Krein1, Krein2, Vishik} as proposed by Ryzhov \cite{Ryzh_spec}. In adopting the mentioned approach, we deviate from the boundary triples theory utilised previously in our ODE analysis \cite{Physics, GrandePreuve}. Still, the theory developed in \cite{Ryzh_spec} can be seen as a ``version'' of the latter, in that it attempts to write the second Green's identity in operator-theoretic terms as
\begin{equation}
\langle Au,v \rangle_H - \langle u,Av \rangle_H = \langle \Gamma_1 u, \Gamma_0 v \rangle_{\mathcal H} - \langle \Gamma_0 u,\Gamma_1 v  \rangle_{\mathcal H},
\label{Green_id}
\end{equation}
where $A$ is an operator in a Hilbert space $H$ (e.g., a Laplacian would lead to the classical Green's identity with $\Gamma_0$ and $\Gamma_1$ defined as traces of the function and of its normal derivative on the boundary, respectively) and $(\mathcal H, \Gamma_0,\Gamma_1)$ is the \emph{boundary triple}, consisting of an auxiliary Hilbert space ({\it e.g.}, the $L^2$ space over the boundary) and a pair of operators \emph{onto} $\mathcal H$. Unfortunately, it is well known that the direct approach via (\ref{Green_id}), although possible, encounters problems in the PDE setting (we refer the reader to a review \cite{Derkach} of the state of the art in the theory, see also Section 2 below), thus necessitating an alternative. For this we have selected the boundary triple theory of \cite{Ryzh_spec}, as a natural fit for the analysis we have in mind.

In order to put the results of our work into the correct context, we mention that already in the extensively studied setup of double-porosity models (see Model I, Section 2) we are not only able to develop the operator-theoretical approach, but also to extend the existing results on spectral convergence in at least two ways: firstly, by ascertaining the rate convergence as $O(\e^{2/3})$, and secondly, by disposing of the assumption on the eigenvectors of the Dirichlet Laplacian on the soft component, which was previously considered essential. Moreover, the end product of our analysis in this case, Theorem \ref{thm:ModelI_final_result}, puts the frequency-dispersive (and thus time-dispersive) properties of the homogenised medium under the spotlight, allowing to write the corresponding ``Zhikov function" $\mathfrak B(z)$ explicitly and thus putting on a firm ground the possibility to view media with double porosity as {\it frequency-converting} devices. Furthermore, an asymptotic development of our approach allows us to upgrade the order of convergence estimates from $O(\e^{2/3})$ to $O(\e^\alpha),$ $\alpha>1,$ which we shall carry out in the forthcoming second part \cite{ChEK_future} of our work. 

Finally, we compare our results to earlier approaches to the spectral analysis of scalar PDE problems with critical contrast, which thus far have been applied in the particular geometric setup of ``disperse'' soft inclusions in a connected stiff medium, see \cite{Zhikov2000}, \cite{CherCoop}. Similarly to \cite{CherCoop}, our machinery goes well beyond the two-scale convergence techniques of \cite{Zhikov2000}, in that the convergence of spectra is now a simple consequence of the convergence of solutions to the original PDE in the operator-norm topology. In contrast to the norm-resolvent asymptotics of \cite{CherCoop}, where order $O(\varepsilon)$ estimates were obtained for the difference between the resolvents of the original problems and the approximating operators, the estimate of the present paper, see Theorem \ref{thm:general_homo_result}, provides an explicit expression for a modified approximating homogenised operator, so that the new convergence estimates are of the higher order $O(\varepsilon^2).$ In addition, in contrast to \cite{Zhikov2000} and \cite{CherCoop}, we consider arbitrary geometries, where the soft inclusions need not be disconnected and the stiff component need not be connected. Yet another novel feature of the present work is an explicit description of the related effective macroscopic time-dispersive media, see the discussion above Section \ref{time_disp_section}. This establishes an explicit link between critical-contrast composites and metamaterials, which the earlier works were somewhat lacking.

In conclusion, we would like to mention that although our techniques necessarily rely upon what might be seen as a purely abstract framework, the exposition of the paper  is effectively complemented for the benefit of more applied-minded audience by that of our earlier paper \cite{Physics}, which attempts to make our ideas more transparent, by resorting to the asymptotic analysis of the eigenvectors, as opposed to the resolvents, of the corresponding self-adjoint operators.

\section{Problem setup and some preliminaries}
\label{setup_section}
In the present work, we consider the problem \eqref{eq:generic_hom} under the following assumptions:
\begin{equation}\label{weight}
A^\e(y)=
\begin{cases}
a^2(y)I,& y\in Q_{\text{stiff}},\\[0.3em]
\varepsilon^2I, & y\in Q_{\text{soft}},
\end{cases}
\end{equation}
where $Q_{\text{soft}}$ ($Q_{\text{stiff}}$) is the soft (respectively, stiff) component of the unit cube $Q=[0,1)^d\subset\mathbb R^d$, so that $\overline{Q}=\overline{Q}_{\text{soft}}\cup \overline{Q}_{\text{stiff}}$, and $a^2>c_0>0$ is a $C^\infty$ function on $\overline{Q}_\stiff$.

Two distinct models will be of a particular interest to us, see Fig.\,\ref{fig:media}. For simplicity but without loss of generality, in both of them the coefficient $a^2$ will be assumed identically equal to unity. In Model I, which is unitary equivalent to the model of \cite{HempelLienau_2000}, \cite{Friedlander}, we will assume further that $Q_{\text{soft}}$ is a simply connected domain with a smooth boundary $\Gamma$, such that the distance from its boundary to the boundary of the unit cube $Q$ is positive, {\it cf.} \cite{Zhikov2000}, \cite{CherCoop}. In Model II, in contrast, we will assume that $Q_{\text{stiff}}$ is a simply connected domain with a smooth boundary $\Gamma$, separated from the boundary of the unit cube $Q$. The assumption that $\Gamma$ is smooth in both cases cannot in general be dropped, as we will require DN maps pertaining to this boundary to be well-defined as pseudo-differential operators of order one \cite{Hoermander}, \cite{Friedlander_old}; also see {\it e.g.} \cite{Arendt} and references therein for the treatment of DN maps in the case of Lipschitz boundary, which can be used to generalise our approach to this case. It has to be noted, however, that recent developments in the application of pseudo-differential calculus to the extension theory on non-smooth domains, see \cite{AGW} and references therein, indicate that non-smooth boundaries may also be treated using our strategy. 
\begin{figure}[h!]
\begin{center}
\includegraphics[scale=0.4]{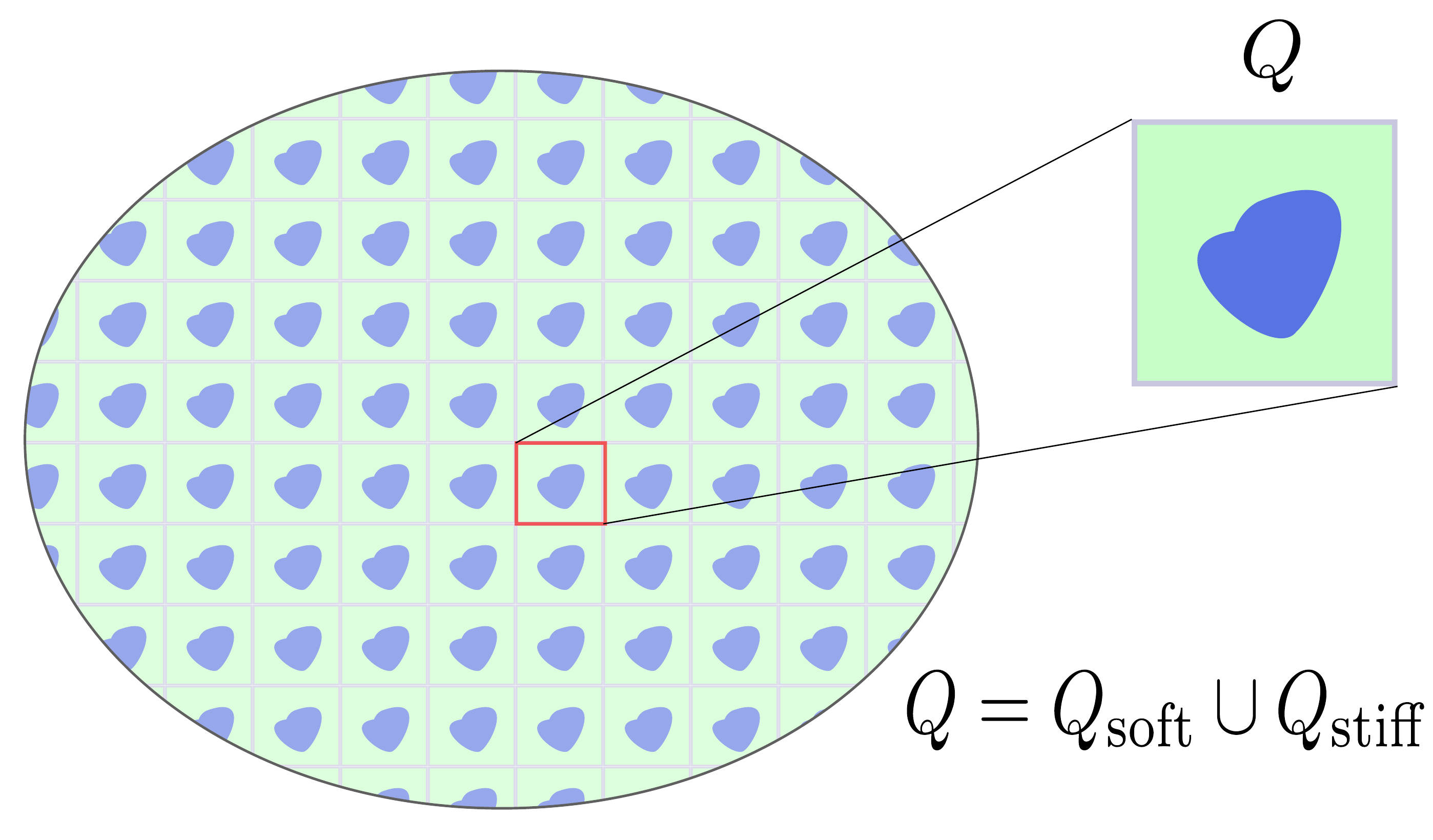}
\end{center}
\caption{{\scshape Model setups.} {\small Model I: soft component $Q_\soft$ in blue, stiff component $Q_\stiff$ in green. Model II: soft component $Q_\soft$ in green, stiff component $Q_\stiff$ in blue.}\label{fig:media}}
\end{figure}
To summarise, the setup of Model I is therefore that of soft inclusions in a stiff connected ``matrix'', Model II represents the ``dual" case of stiff inclusions in a soft connected matrix.

In what follows, we present a systematic treatment of Models I and II, followed by the list of amendments (see Section \ref{concluding_remarks}) allowing us to obtain the same results in the general case \eqref{weight}, with non-piecewise constant coefficient $a^2$ in (\ref{weight}). This choice of exposition is justified, since the treatment of the general case (which does not seem to be well-motivated from the point of view of real-world applications) is in no way different to the treatment of Models I and II.

In both cases we shall deal with the resolvent $({A}_\e-z)^{-1}$ of a self-adjoint operator in $L^2(\mathbb R^d)$ corresponding to the problem \eqref{eq:generic_hom}, so that solutions of the latter are expressed as $u_\e=({A}_\e-z)^{-1}f$ with $z=\omega^2.$ Here the operator family ${A}_\e$ can be thought of as defined by the coercive, but not uniformly coercive, forms
$$
\int_{{\mathbb R}^d}A^\e(\cdot/\e)\nabla u\cdot \overline{\nabla u},\qquad u, v\in H^1({\mathbb R}^d).
$$

In what follows, we will always assume that $z\in \mathbb C$ is separated from the spectrum of the original operator family, more precisely, we assume that $z\in K_\sigma$, where
$$
K_\sigma:=\bigl\{z\in \mathbb C |\ z\in K \text{ a compact set in } \mathbb C,\ \text{dist}(z, \mathbb R)\geq \sigma\bigr\},\qquad \sigma>0.
$$
After we have established the operator-norm asymptotics of $({A}_\e-z)^{-1}$ for $z\in K_\sigma,$ the result is extended by analyticity to a compact set 
 the distance of which to the spectrum of the leading order of the asymptotics is bounded below by $\sigma.$

In view of dealing with operators having compact resolvents, it is customary to apply either a Floquet (see \cite{Friedlander}) or Gelfand (see \cite{Gelfand}) transform, of which we use the latter, in line with \cite{BirmanSuslina}.

\subsection{Gelfand transform and direct integral}
\label{Gelfand_section}
Consider the Gelfand transform of a function $U\in L^2(\mathbb R^d)$, which is the element  $\widehat{U}=\widehat{U}(x,\tau)$ of $L^2(Q\times Q')$, $Q'=[-\pi,\pi)^d$, defined by the formula\footnote{The formula (\ref{Gelfand_formula}) is first applied to continuous functions $U$ with compact support, and then extended to the whole of $L^2({\mathbb R}^d)$ by continuity.}
\begin{equation}
\widehat{U}(x,\tau)=(2\pi)^{-d/2}\sum_{n\in{\mathbb Z}^d}U(x+n)\exp\bigl(-{\rm i}\tau\cdot(x+n)\big),\ \ \ x\in Q,
\ \tau\in Q'.
\label{Gelfand_formula}
\end{equation}
The Gelfand transform thus defined is a unitary operator between $L^2({\mathbb R}^d)$ and $L^2(
Q\times Q'),$ and the inverse
transform is expressed by the formula
\begin{equation}
\label{inverse_period_one}
U(x)=(2\pi)^{-d/2}\int_{Q'}\widehat{U}(x, \tau)\exp({\rm i}\tau\cdot x)d\tau,\ \ \ \ x\in{\mathbb R}^d,
\end{equation}
where $\widehat{U}$ is extended to ${\mathbb R}^d\times Q'$ by $Q$-periodicity in the spatial variable.
For the scaled version of the above transform, for $u\in L^2(\mathbb R^d)$ we set
\begin{equation}
Gu(x,t):=
\biggl(\frac{\varepsilon}{2\pi}\biggr)^{d/2}\sum_{n\in{\mathbb Z}^d}u(x+\varepsilon n)\exp\bigl(-{\rm i}t\cdot(x+\varepsilon n)\big),\ \ \ x\in \varepsilon Q,\ t\in\varepsilon^{-1}Q',
\label{scaled_Gelfand}
\end{equation}
and the inverse of the transform (\ref{scaled_Gelfand}) is given by
\begin{equation}
u(x)=\biggl(\frac{\varepsilon}{2\pi}\biggr)^{d/2}\int_{\varepsilon^{-1}Q'}Gu(x,t)
\exp({\rm i}t\cdot x)dt,\ \ \ \ x\in\mathbb R^d.
\label{inverse_scaled_Gelfand}
\end{equation}

As in \cite{BirmanSuslina}, an application of the Gelfand transform $G$ to the operator family ${A}_\e$ corresponding to the problem \eqref{eq:generic_hom} yields the two-parametric operator family $A_\e^{(t)}$ in $L^2(\e Q)$ given by the differential expression
$$
\left(\frac 1{\rm i}\nabla + t\right) A^\e(x/\e) \left(\frac 1{\rm i}\nabla + t\right), \qquad \varepsilon>0,\quad t\in\varepsilon^{-1}Q',
$$
subject to \emph{periodic} boundary conditions on the boundary of $\e Q$ and defined by the corresponding closed coercive sesquilinear form.
For each $\varepsilon>0,$ the operator ${A}_\e$ is then unitary equivalent to the von Neumann integral (see {\it e.g.} \cite[Chapter 7]{Birman_Solomjak})
$$
{A}_\e=G^*\biggl(\oplus \int_{\e^{-1}Q'} {A}_\e^{(t)}dt\biggr)G.
$$

Next, we introduce a unitary rescaling for the operator family ${A}_\e^{(t)}$. Set $\tau:=\e t\in Q'$ and consider the rescaled operator family ${A}_\e^{(\tau)}$ defined by
${A}_\e^{(\tau)}=\Phi_\e {A}_\e^{(t)}\Phi_\e^*$, where the unitary $\Phi_\e$ acts on $u\in L^2(\e Q)$ by the following rule:
$$
\Phi_\e u= {\e}^{d/2} u(\e\cdot).
$$
To simplify notation, we choose to keep the same symbol ${A}_\e^{(\tau)}$ for the unitary image of ${A}_\e^{(t)},$ where $t=\tau/\varepsilon\in\e^{-1}Q',$ hoping that this does not lead to a confusion. The operator family ${A}_\e^{(\tau)}$ in $L^2(Q)=:H$ ({\it cf.} \eqref{Green_id}) corresponds to the differential expressions
\[
\left(\frac 1{\rm i}\nabla_y + \tau\right) \e^{-2}A^\e(y) \left(\frac 1{\rm i} \nabla_y + \tau\right),\qquad \varepsilon>0,\quad \tau\in Q',
\]
subject to periodic boundary conditions on $\partial Q,$ and it is uniquely defined by its form.

Clearly, one can first apply a unitary rescaling to the operator family ${A}_\e$, followed by an application of the Gelfand transform \eqref{Gelfand_formula}, leading to the same operator family ${A}_\e^{(\tau)}$, so that the fibre representation
$$
{A}_\e\cong\oplus \int_{Q'} {A}_\e^{(\tau)} d\tau
$$
holds, where $\cong$ stands for unitary equivalence. Although the main body of the paper makes no explicit reference to the operators ${A}_\e^{(t)},$ the order in which we apply the (scaled) Gelfand transform and the unitary rescaling, and thus the very appearance of the operator ${A}_\e^{(t)},$ is required for the transparency of the analysis of Section \ref{prep_mat}. Therein, we find it instructive to utilise the scaled inverse Gelfand transform, in view of the natural appearance of an integral over an expanding dual cell $\e^{-1}Q'$, which is, in its turn, required for the main result of the section, Theorem \ref{thm:pseudodifferential}.


Similar to \cite{Friedlander} and facilitated by the abstract framework of \cite{Ryzh_spec}, instead of the form-based definition of the operators ${A}_\e^{(\tau)}$ we will consider them as operators of transmission problems, see \cite{Schechter} and references therein. Henceforth we treat the cube $Q$ as a torus by identifying its opposite sides, and $Q_\soft$ and $Q_\stiff$ are viewed as subsets of this torus so that, in particular, $\partial Q_\soft=\partial Q_\stiff=\Gamma.$ For each $\varepsilon,$ $\tau,$ the mentioned transmission problem is formulated as finding a function 
\[
u\in\bigl\{u\in L^2(Q): u\vert_{Q_\soft}\in H^1(Q_\soft), u\vert_{Q_\stiff}\in H^1(Q_\stiff)\bigr\} 
\]
that solves (in the variational, or weak, sense) the boundary-value problem (BVP) with $f\in H:$ 
\begin{equation}
\label{eq:transmissionBVP}
\begin{cases}
&\e^{-2}\left(\dfrac 1 {\rm i} \nabla + \tau\right)^2 u(x)-zu(x)=f(x),\quad  x\in Q_\stiff,\\[0.45em]
&\left(\dfrac 1 {\rm i} \nabla + \tau\right)^2 u(x)-zu(x)=f(x),\quad x\in Q_\soft,\\[0.8em]
&u_+(x)=u_-(x),\quad x\in \Gamma,\\[0.5em]
&\dfrac{\partial u}{\partial n_+}(x)+{\rm i}\bigl(\tau\cdot n_+(x)\bigr)u_+(x)+\e^{-2}\left(\dfrac{\partial u}{\partial n_-}(x)+{\rm i}\bigl(\tau\cdot n_-(x)\bigr)u_-(x)\right)=0,\quad x\in \Gamma.
\end{cases}
\end{equation}
where $u_+,$ $\partial u/\partial n_+,$ $u_-,$ $\partial u/\partial n_-$  are the limiting values of the function $u$ and its normal derivatives on $\Gamma$ from inside $Q_\soft$ and $Q_\stiff$, respectively; the vectors $n_+$ and $n_-=-n_+$ are outward normals to $\Gamma$ with respect to $Q_\soft$ and $Q_\stiff.$ 
By a classical argument the weak solution of the above problem is shown to be equal to $({A}_\e^{(\tau)}-z)^{-1}f.$ 

\subsection{Boundary operator framework}

Following \cite{Ryzh_spec} ({\it cf.} \cite{BehrndtLanger2007}, \cite{BMNW2008} and references therein for alternative approaches), which is based on the ideas of the classical Birman-Kre\u\i n-Vi\v sik theory (see \cite{Birman, Krein1, Krein2, Vishik}), the linear operator of the transmission boundary-value problem is introduced as follows. Let $\mathcal H:=L^2(\Gamma)$, and consider the ``$\tau$-harmonic" lift operators $\Pi_\soft$ and $\Pi_\stiff$ defined on $\phi\in\mathcal H$ via\footnote{In what follows we will often drop the superscript $(\tau),$ hopefully at no expense to the clarity, since most of our objects do depend on $\tau.$ Whenever this is not the case, we specifically indicate this. We will often apply the same convention to $\varepsilon$ and $z$ appearing in super- and subscripts. We hope that this will not lead to ambiguity.}
\begin{equation}
\label{eq:harmonic_lift1}
\Pi_{\stiff(\soft)} \phi:=u_\phi,\quad\begin{cases}
(\nabla+{\rm i}\tau)^2 u_\phi=0,\ \ u_\phi\in L^2(Q_{\stiff(\soft)}),&\\[0.4em]
u_\phi|_\Gamma = \phi,
\end{cases}
\end{equation}
with periodic conditions on $\partial Q$ added for each component as required. 
These operators are first defined on $\phi\in C^2(\Gamma)$, in which case the corresponding solutions $u_\phi$ can be seen as classical. The standard ellipticity estimate allows one to extend both $\tau$-harmonic lifts to bounded operators from $\mathcal H$ to 
$L^2(Q_{\stiff(\soft)}),$ so that the functions $u_\phi$ are treated as distributional solutions\footnote{``Very weak solutions" in the terminology of \cite{Necas}.} of the respective BVPs, see {\it e.g.}, \cite[Theorem 3.2, Chapter 5]{Necas}, \cite[Theorem 4.25]{McLean}.
The solution operator $\Pi:\mathcal H\mapsto H=L^{2}(Q_\soft)\oplus L^{2}(Q_\stiff)$ is defined as follows:
$$
\Pi \phi:= \Pi_\soft \phi \oplus \Pi_\stiff \phi,\quad\phi\in{\mathcal H}.
$$

Consider the self-adjoint operator family $A_0$ (where we drop the superscript $(\tau)$ and the subscript $\e$ for brevity) to be the Dirichlet decoupling of the operator family ${A}_\e^{(\tau)}$, {\it i.e.} the operator of the Dirichlet BVP on both $Q_\stiff$ and $Q_\soft,$ with periodic boundary conditions on $\partial Q$ where appropriate, defined by the same differential expression as ${A}_\e^{(\tau)}$. Clearly, one has $A_0=A_0^\stiff\oplus A_0^\soft$ relative to the orthogonal decomposition $H=L^2(Q_\stiff)\oplus L^2(Q_\soft),$ where 
$A_0^\stiff$ and $A_0^\soft$  are the operators of the Dirichlet BVP on $Q_\stiff$ and $Q_\soft,$ for the differential expressions $-\varepsilon^{-2}(\nabla+{\rm i}\tau)^2$ and $-(\nabla+{\rm i}\tau)^2,$ respectively. All three operators $A_0, A_0^\soft$ and $A_0^\stiff$ are self-adjoint and positive-definite. Moreover, there exists a bounded inverse $A_0^{-1}$ for all $\tau\in Q',$ and $\dom A_0\cap\ran \Pi=\{0\}.$ 

Denoting by $\widetilde \Gamma_0^{\stiff(\soft)}$ the left inverse\footnote{This left-inverse operator is well-defined on $\ran \Pi_{\stiff (\soft)}.$ Note that neither its closedness nor closability is assumed.} of $\Pi_{\stiff(\soft)},$ we introduce the trace operator $\Gamma_0^{\stiff(\soft)}$ as the null extension of $\widetilde \Gamma_0^{\stiff(\soft)}$ to 
$\dom A_0^{\stiff (\soft)}\dotplus \ran \Pi_{\stiff (\soft)}$. In the same way we introduce the operator $\widetilde \Gamma_0$ and its null extension $\Gamma_0$ to the domain $\dom A_0 \dotplus \ran \Pi$.

The solution operators $S_z^\stiff$, $S_z^\soft$ of the BVPs
\begin{numcases}{}
-\e^{-2}(\nabla+{\rm i}\tau)^2 u_\phi-z u_\phi=0,\quad u_\phi\in \dom A_0^\stiff\dotplus \ran\Pi_\stiff,\label{stiff_eq}&\\[0.2em]
\Gamma_0^\stiff u_\phi= \phi,\nonumber
\end{numcases}
and 
\begin{numcases}{}
-(\nabla+{\rm i}\tau)^2 u_\phi-z u_\phi=0,\quad u_\phi\in \dom A_0^\soft\dotplus \ran\Pi_\soft,\label{soft_eq}\\[0.2em]
\Gamma_0^\soft u_\phi= \phi,\nonumber
\end{numcases}
are defined as linear mappings from $\phi$ to $u_\phi$, respectively. These operators are bounded from ${\mathcal H}$ to $L^2(Q_\stiff)$ and $L^2(Q_\soft)$, respectively, and admit the representations
\begin{equation}
S_z^{\stiff(\soft)}=\bigl(1-z(A_0^{\stiff(\soft)})^{-1}\bigr)^{-1}\Pi_{\stiff(\soft)},\qquad{z\in\rho(A_0)}.
\label{eq:Sz}
\end{equation}
Clearly, the operator $S_z^\soft$ is $\e$-independent (which is a result of applying the rescaling $\Phi_\e$), and the operator $S_z^\stiff$ admits the obvious estimate
\begin{equation}
S_z^\stiff - \Pi_\stiff = O(\e^2),
\label{SPi_stiff}
\end{equation}
uniformly in $\tau\in Q'$, $z\in K_\sigma,$ in the uniform operator-norm topology, owing to the fact that 
\begin{equation}
\bigl\|(A_0^\stiff)^{-1}\bigr\|_{L^2(Q_\stiff)\to L^2(Q_\stiff)}\leq C \e^2.
\label{A0_bound}
\end{equation}
Finally, the solution operator $S_z:{\mathcal H}\mapsto L^2(Q_\soft)\oplus L^2(Q_\stiff)$ is defined by $S_z\phi=S_z^\soft\phi\oplus S_z^\stiff\phi,$ $\phi\in{\mathcal H}.$ It admits the representation 
$S_z=(1-z A_0^{-1})^{-1}\Pi$ and is bounded. Introducing the orthogonal projections $P_\soft$ and $P_\stiff$ from $H$ onto $L^2(Q_\soft)$ and $L^2(Q_\stiff)$, respectively, one has the obvious identities
\begin{equation}\label{eq:Szstiffsoft}
S_z^{\stiff(\soft)}=P_{\stiff(\soft)} S_z,
\qquad \Pi_{\stiff(\soft)}=P_{\stiff(\soft)}\Pi,
\end{equation}

Next, fix self-adjoint in $L^2(\Gamma)$ (and, in general, unbounded) operators $\Lambda^\soft$, $\Lambda^\stiff$ defined on domains $\dom \Lambda^{\stiff(\soft)},$ 
(for problems considered in the present article one has $\dom \Lambda^\stiff=\dom \Lambda^\soft  = H^1(\Gamma)$, where $H^1(\Gamma)$ is the standard Sobolev space pertaining to the smooth boundary $\Gamma$). Still following \cite{Ryzh_spec}, we define the ``second boundary operators'' $\Gamma_1^\stiff$ and $\Gamma_1^\soft$ with domains
\begin{equation}\label{eq:Gamma1softstiff_dom}
\dom \Gamma_1^{\stiff(\soft)}:=\dom A_0^{\stiff(\soft)}\dotplus \Pi_{\stiff(\soft)} \dom \Lambda^{\stiff(\soft)}, 
\end{equation}
and the action of $\Gamma_1^{\stiff (\soft)}$ is set by
\begin{equation}\label{eq:Gamma1softstiff_action}
\Gamma_1^{\stiff(\soft)}:\ \bigl(A_0^{\stiff(\soft)}\bigr)^{-1}f\dotplus \Pi_{\stiff(\soft)}\phi \mapsto \Pi_{\stiff(\soft)}^* f+\Lambda^{\stiff(\soft)}\phi,
\end{equation}
for all $f\in L^2(Q_{\stiff(\soft)})$, $\phi\in \dom \Lambda^{\stiff(\soft)}.$ 

Under the additional assumption that $\dom \Lambda^\soft =\dom \Lambda^\stiff=:\mathcal D,$ we also introduce the operator
$
\Gamma_1
$ as the formal\footnote{Note that $\Gamma_1^\stiff$ and $\Gamma_1^\soft$ have different domains.} sum of $\Gamma_1^\soft$ and $\Gamma_1^\stiff$. The rigorous definition of $\Gamma_1$ is as follows:
\begin{align}
&\dom \Gamma_1:= \dom A_0\dotplus \Pi \mathcal D,\label{three_star}\\[0.4em]
&\Gamma_1:\ A_0^{-1}f\dotplus \Pi \phi \mapsto \Pi^* f +\bigl(\Lambda^\soft+\Lambda^\stiff\bigr)\phi\qquad \forall\ f\in L^2(Q),\ \phi\in\mathcal D.\nonumber
\end{align}
(The PDE interpretation of this definition, which is, in fact the jump of the co-normal derivative on the boundary, is given below on p.\,\pageref{Marco}.)
We remark that the operators $\Gamma_1, \Gamma_1^{\soft (\stiff)}$ thus defined are not assumed to be either closed or closable.

For purposes of our analysis, we set $\Lambda^{\soft(\stiff)}$ to be the DN maps pertaining to the components $Q_\soft$ and $Q_\stiff$, respectively. More precisely, for the problem\footnote{Recall that $Q_\soft$ and $Q_\stiff$ are viewed as subsets of the torus obtained by identifying opposite sides of the cube $Q.$}
\begin{equation*}
\begin{cases}
\left(\dfrac 1 {\rm i} \nabla+\tau\right)^2 u_\phi=0, \quad u_\phi\in L^2(Q_\soft),\\[0.8em]
u_\phi|_\Gamma = \phi,
\end{cases}
\end{equation*}
we define $\Lambda^\soft$ as the operator mapping the boundary values $\phi$ of $u_\phi$ to the trace of its co-normal derivative\footnote{This definition is inspired by \cite{Ryzh_spec}. Note that the operator thus defined is the negative of the classical Dirichlet-to-Neumann map of {\it e.g.} \cite{Friedlander}.} 
\begin{equation}
-\left(\frac{\partial u_\phi}{\partial n_+}+{\rm i}(\tau\cdot n_+)u_\phi\right)\biggr|_\Gamma=:\dntau u_\phi,
\label{conormal_soft_later}
\end{equation}
where, as before, $n_+$ is the outward unit normal to $\Gamma$ relative to $Q_\soft$. This operator is first defined on sufficiently smooth $\phi\in{\mathcal H}.$ It is a classical result \cite{Friedlander_old,Friedlander, Taylor, Taylor_tools} that it is then extended to a self-adjoint operator with domain ${\mathcal D}=H^1(\Gamma)$ (moreover, it is a pseudo-differential operator of order one).

Similarly, on the stiff component $Q_\stiff$ we consider the problem
\begin{equation}\label{eq:DN_stiff}
\begin{cases}
\left(\dfrac 1 {\rm i} \nabla+\tau\right)^2 u_\phi=0, \quad u_\phi\in L^2(Q_\stiff),\\[0.8em]
u_\phi|_\Gamma = \phi,
\end{cases}
\end{equation}
and define $\Lambda^\stiff$ as the operator mapping the boundary values $\phi$ of $u_\phi$ to the $\varepsilon^{-2}$-\emph{weighted} trace of its co-normal derivative 
\begin{equation}
-\e^{-2}\left(\frac{\partial u_\phi}{\partial n_-}+{\rm i}(\tau\cdot n_-)u_\phi\right)\biggr|_\Gamma=:\e^{-2}\dntau u_\phi,
\label{conormal_stiff_later}
\end{equation}
where, as before, $n_-$ is the outward unit normal to $\Gamma$ relative to $Q_\stiff$. In the same way as $\Lambda^\soft,$ this is extended to a self-adjoint operator with domain ${\mathcal D}=H^1(\Gamma)$. Note that the operator $\Lambda^\stiff$ thus defined explicitly depends on $\e$. The only dependence on $\e$ is in the definition of the weighted co-normal derivative and is multiplicative.

In view of what follows, it is important to remark that, in general, the operator $\Lambda:=\Lambda^\soft+\Lambda^\stiff$ cannot be guaranteed to be self-adjoint on $H^1(\Gamma)$, since it is a sum of unbounded self-adjoint operators, albeit defined on the same domain. In our setup, however, the situation is better, as described in the following statement.

\begin{lemma}\label{lemma:selfadjointness_of_Lambda}
For $\e<\e_0,$ where $\e_0$ is an independent constant, the operator $\Lambda=\Lambda^\soft+\Lambda^\stiff$ is self-adjoint in $L^2(\Gamma)$ with domain $\dom \Lambda = H^1(\Gamma).$ 
\end{lemma}
\begin{proof}
Using the argument of \cite[Lemma 2]{Friedlander}, one has that for each $\tau$ the operators $\Lambda^{\stiff(\soft)}$  are perturbations of the corresponding Dirichlet-to-Neumann maps for $\tau=0$ by uniformly bounded operators. Therefore, it is sufficient to establish the claim of the lemma at an arbitrary single value of $\tau.$  From the definition of the operator $\Lambda^\stiff$ one has $\Lambda^\stiff = \e^{-2} \widetilde \Lambda^\stiff$, where $\widetilde \Lambda^\stiff$ is the ``unweighted'' DN map of the problem \eqref{eq:DN_stiff}, {\it i.e.}, the operator mapping the boundary values $\phi$ of $u_\phi$ to the traces of its co-normal derivative 
$\dntau u_\phi,$ which is self-adjoint with domain $H^1(\Gamma)$.  Now having fixed some value of $\tau,$ it follows by {\it e.g.} \cite[Chapter IV, Remark 1.5]{Kato} that $\Lambda^\soft$ is $\widetilde \Lambda^\stiff$-bounded, {\it i.e.,} there exist $\alpha, \beta>0$ such that
\begin{equation}
\bigl\|\Lambda^\soft u\bigr\|_{\mathcal H}\leq \alpha\bigl\|\widetilde\Lambda^\stiff u\bigr\|_{\mathcal H}+\beta \|u\|_{\mathcal H}\quad \forall u\in \dom \widetilde \Lambda^\stiff.
\label{relative_bound}
\end{equation}
Finally, the estimate (\ref{relative_bound}) implies that 
$$
\bigl\|\Lambda^\soft u\bigr\|_{\mathcal H}\leq \alpha\e^2 \bigl\| \Lambda^\stiff u\bigr\|_{\mathcal H}+\beta \|u\|_{\mathcal H}\quad \forall u\in \dom \Lambda^\stiff,
$$
so that for $\e< \e_0:=1/\sqrt{\alpha}$ the operator $\Lambda^\soft$ is $\Lambda^\stiff$-bounded with the relative bound strictly less than 1. By a classical result of Kato (see \cite[Chapter V, Theorem 4.3]{Kato}) this yields self-adjointness of $\Lambda^\stiff+\Lambda^\soft$ as an operator with domain $\dom \Lambda^\stiff=H^1(\Gamma)$.
\end{proof}

For our choice of $\Lambda^{\soft (\stiff)}$ we have ${\mathcal D}=H^1(\Gamma),$ which allows us to consider $\Gamma_1$  on $\dom A_0 \dotplus \Pi H^1(\Gamma),$ see (\ref{three_star}). One then writes \cite{Ryzh_spec} the second Green identity
\begin{equation}
\label{eq:quasi_triple_property}
\begin{aligned}
\langle Au,v \rangle_{L^2(Q)} - \langle u, Av \rangle_{L^2(Q)}&=\langle \Gamma_1 u, \Gamma_0 v \rangle_{{\mathcal H}} - \langle \Gamma_0 u, \Gamma_1 v \rangle_{{\mathcal H}}\\[0.3em]
\forall u,v\in \dom \Gamma_1&=\dom A_0\dotplus \Pi H^1(\Gamma),
\end{aligned}
\end{equation}
where the $A$ is the null extension (see \cite{Ryzhov_later}) of the operator $A_0$ to $\dom A_0\dotplus\ran\Pi.$ 
Thus, the triple $(\mathcal H, \Gamma_0, \Gamma_1)$ is closely related to a boundary quasi-triple \cite{BehrndtLanger2007} (see also \cite{BehrndtRohleder2015}) for the transmission problem considered; {\it cf.} \cite{BMNW2008,BMNW2018} and references therein for the way to introduce a ``proper'' boundary triple of \cite{Gor,DM}. Unlike the analysis of \cite{Physics}, \cite{GrandePreuve} and \cite{CherKis}, this latter version of the theory does not suit our needs, owing to the PDE setup we are dealing with here.

The calculation of $\Pi^*$ in \cite{Ryzh_spec} shows that $\Pi^*=\Gamma_1 A_0^{-1}$ and also that $\Gamma_1$ introduced above acts as follows: \label{Marco}
$$
\Gamma_1:\ u=P_\soft u + P_\stiff u \mapsto \dntau P_\soft u + \e^{-2} \dntau P_\stiff u,
$$
where the operators $\dntau$ on the right-hand side are defined above as (weighted) co-normal derivatives; we reiterate that the normal vector is always assumed to be external, both for $Q_\stiff$ and $Q_\soft$. The transmission problem at hand therefore (at least, formally) corresponds to the ``matching" (or ``interface", ``transmission") condition $\Gamma_1 u=0$.

\subsection{$M$-function and generalised resolvents}
On the basis of the triple $({\mathcal H}, \Gamma_0, \Gamma_1),$ we next introduce the corresponding version of the Dirichlet-to-Neumann map, namely the so-called $M$-function, which maps the boundary data $\Gamma_0u$ to $\Gamma_1u$ for all $u\in\ker(A-z),$ {\it cf.} (\ref{stiff_eq}), (\ref{soft_eq}), and note the its additivity property (Lemma \ref{lemma:M_additivity}) with respect to the decomposition of $Q$ into the stiff ($Q_\stiff$) and soft ($Q_\soft$) components. We then recall a suitable version \cite{Ryzh_spec} of the Kre\u\i n formula for the resolvents of operators describing BVPs for the equation $Au=zu$ subject to boundary conditions of the form $\beta_0\Gamma_0u+\beta_1\Gamma_1u=0,$ where $\beta_0,$ $\beta_1$ are linear operators on the boundary. In the case of the operators $A_\varepsilon^{(\tau)},$ see (\ref{eq:transmissionBVP}), such resolvent expressions are an essential ingredient in the derivation (Section \ref{norm_resolvent_section}) of the corresponding norm-resolvent asymptotics  as $\varepsilon\to0.$
\begin{definition}[\cite{Ryzh_spec}]
The operator-valued function $M(z)$ defined on $\dom \Lambda$ for
$z\in \rho(A_0)$ (and in particular, for $z\in K_\sigma$) by the formula ({\it cf.} \eqref{eq:Sz})
\begin{equation}\label{defn:M-function}
  M(z) \phi = \Gamma_1 S_z \phi = \Gamma_1\bigl(1-z A_0^{-1}\bigr)^{-1}\Pi \phi
\end{equation}
is called the M-function of the problem \eqref{eq:transmissionBVP}.
\end{definition}

The next result of \cite{Ryzh_spec} summarises the properties of the $M$-function that we will need in what follows.

\begin{proposition}[\cite{Ryzh_spec}, Theorem 3.11]\label{prop:M}

1. The following representation holds:
\begin{equation}
\label{eq:M_representation}
  M(z)=\Lambda + z \Pi^* (1-z A_0^{-1})^{-1}\Pi, \quad z\in \rho(A_0).
\end{equation}

2. $M$ is an analytic function with values in the set of closed operators in ${\mathcal H},$ densely defined on the $z$-independent domain $\dom \Lambda$.

3. For $z,\zeta \in\rho (A_0)$ the operator $M(z) - M(\zeta)$ is bounded, and
$$
M(z) - M(\zeta) = (z - \zeta)S^*_{\bar z} S_\zeta.
$$
In particular, $\Im M(z) = (\Im z) S^{*}_{\bar z} S_{\bar z}$ and $(M(z))^*  = M(\bar z)$.

4. For $u_z \in \ker (A - zI) \cap \{ \dom A_0  \dotplus \Pi \dom \Lambda \}$ the following formula holds:
$$
M(z)\Gamma_0 u_z = \Gamma_1 u_z.
$$
\end{proposition}

Alongside $M(z),$ we define $M^\stiff(z)$ and $M^\soft(z)$ pertaining to stiff ($Q_\stiff$) and soft ($Q_\soft$) components of the composite by the formulae
\begin{equation}
\label{defn:M-stiffsoft}
M_{\stiff(\soft)}(z) \phi=\Gamma^{\stiff(\soft)}_1 S_z^{\stiff(\soft)}\phi = \Gamma_1^{\stiff(\soft)}\bigl(1-z (A_0^{\stiff(\soft)})^{-1}\bigr)^{-1}\Pi_{\stiff(\soft)} \phi,\qquad z\in\rho\bigl(A_0^{\stiff(\soft)}\bigr).
\end{equation}
A straightforward application of Proposition \ref{prop:M}, together with \eqref{eq:Sz}, \eqref{eq:Szstiffsoft}, yields the following statement.
\begin{lemma}\label{lemma:M_additivity}
The identity
\begin{equation}\label{eq:M_additivity}
  M(z)=M^\stiff(z)+M^\soft(z),\quad z\in \rho(A_0),
\end{equation}
holds.
\end{lemma}

The value of the above lemma is clear: in contrast to ${A}_\e^{(\tau)},$ which cannot be additively decomposed into ``independent" parts pertaining to the soft and stiff components, owing to the transmission interface conditions on the common boundary $\Gamma$ of the two, the $M$-function turns out to be additive ({\it cf.}, {\it e.g.}, \cite{CherKisSilva}, where this additivity was observed and exploited in an independent, but closely related, setting of scattering). In what follows we will observe that the resolvent $({A}_\e^{(\tau)}-z)^{-1}$ can be expressed in terms of $M(z)$ via a version of the celebrated Kre\u\i n formula, thus reducing the asymptotic analysis of the resolvent to that of the corresponding $M$-function.

Alongside the transmission problem \eqref{eq:transmissionBVP}, whose boundary conditions can now be (so far, formally) represented as $\Gamma_1 u=0,$ $u\in \dom A_0\dotplus\ran\Pi,$ in what follows we consider a wider class of problems,
 formally given by transmission conditions of the type
$$
\beta_0 \Gamma_0 u+\beta_1 \Gamma_1 u =0,\quad u\in \dom A=\dom A_0\dotplus\ran\Pi, 
$$
where $\beta_1$ is a bounded operator on ${\mathcal H}$ and $\beta_0$ is a linear operator defined on  $\dom \beta_0\supset \dom \Lambda$. In general, the operator $\beta_0\Gamma_0+\beta_1 \Gamma_1$ is not defined on the domain $\dom A$. This problem is taken care of by the following assumption, which we apply throughout:
$$
\beta_0+\beta_1 \Lambda, \text{\ defined on $\dom\Lambda,$ is closable in\ } \mathcal H.
$$
We remark that by Proposition \ref{prop:M} the operators $\beta_0+\beta_1 M(z)$ are closable for all $z\in \rho(A_0)$, and the domains of their closures coincide with $\dom \overline{\beta_0+\beta_1 \Lambda}$.

For any $f\in H$ and $\phi\in \dom \Lambda$, the equality
$$
(\beta_0\Gamma_0+\beta_1\Gamma_1)(A_0^{-1}f+\Pi \phi)=\beta_1 \Pi^* f + (\beta_0+\beta_1\Lambda)\phi
$$
shows that the operator $\beta_0\Gamma_0+\beta_1\Gamma_1$ is well defined on $A_0^{-1}H\dotplus \Pi \dom \Lambda\subset \dom A$. Denoting ${\ss}:=\overline{\beta_0+\beta_1\Lambda}$ with domain $\dom{\ss}\supset \dom \Lambda$, one checks \cite[Lemma 4.3]{Ryzh_spec} that $H_{\ss}:=A_0^{-1}H\dotplus \Pi\dom\ss$ is a Hilbert space with respect to the norm
$$
\|u\|^2_{\ss}:=\|f\|^2_H+\|\phi\|^2_{\mathcal H}+\|\ss\phi\|^2_{\mathcal H}, \quad
u=A_0^{-1}f+\Pi\phi.
$$
It is then proved \cite[Lemma 4.3]{Ryzh_spec} that $\beta_0\Gamma_0+\beta_1\Gamma_1$ extends to a bounded operator from $H_{\ss}$ to $\mathcal H$, for which we keep the notation $\beta_0\Gamma_0+\beta_1\Gamma_1$ for the sake of convenience.

In our analysis of the problems \eqref{eq:transmissionBVP} we shall make use of the following version of the Kre\u\i n formula, {\it cf.} \cite[Prop.\,2, Sec.\,2]{DM} and references therein.
\begin{proposition}[\cite{Ryzh_spec}, Theorem 5.5]
\label{prop:operator}
Let $z\in\rho(A_0)$ be such that the operator $\overline{\beta_0+\beta_1 M(z)}$ defined on $\dom{\ss}$ is boundedly invertible. Then\footnote{Note that the formula $S^*_{\bar z}=\Gamma_1(A_0-zI)^{-1},$ $z\in\rho(A_0),$ holds, see {\it e.g.} \cite{Ryzh_spec}.}
\begin{equation}
\label{eq:Krein_general}
R_{\beta_0,\beta_1}(z):=(A_0-z)^{-1}+S_z Q_{\beta_0,\beta_1}(z)S_{\bar z}^*, \text{ where } Q_{\beta_0,\beta_1}(z):=-\bigl(\overline{\beta_0+\beta_1 M(z)}\bigr)^{-1}\beta_1
\end{equation}
is the resolvent at the point $z$ of a closed densely defined operator $A_{\beta_0,\beta_1},$ with domain
$$
\dom A_{\beta_0,\beta_1}=\bigl\{u\in H_{\ss}|\ (\beta_0\Gamma_0+\beta_1\Gamma_1) u=0\bigr\}=\ker(\beta_0\Gamma_0+\beta_1\Gamma_1),
$$
and defined by the formula
\[
A_{\beta_0,\beta_1}u=Au,\qquad u\in \dom A_{\beta_0,\beta_1}.
\]
\end{proposition}


In particular, the resolvent of the (self-adjoint) operator of the transmission problem \eqref{eq:transmissionBVP}, which corresponds to the choice $\beta_0=0, \beta_1=I$, admits the following characterisation (``Kre\u\i n formula"):
\begin{equation}
\label{eq:Krein_transmission}
R_{0,I}(z)=(A_0-z)^{-1}-S_z M(z)^{-1}S_{\bar z}^*.
\end{equation}
In this case, one clearly has $H_{\ss}=A_0^{-1}H\dotplus \Pi \dom \Lambda$ and $\dom A_{0,I}=\{u\in H_{\ss}|\,\Gamma_1 u=0\}.$ 
The proof of the fact that $A_{0,I}={A}_\e^{(\tau)}$ follows from \cite[Remark 3.9]{Ryzh_spec}, which shows that 
$(A_{0,I}-z)^{-1}f$
is the (unique) weak solution of the corresponding BVP. 

We remark that the operators $\beta_0$ and $\beta_1$ above can be assumed $z$-dependent, as this does not affect the related proofs of \cite{Ryzh_spec}. In this case, however, 
the operator function $R_{\beta_0,\beta_1}(z)$ is shown to be the resolvent of a $z$-dependent operator family. In what follows, we will face the situation where $R_{\beta_0,\beta_1}(z)$ is a generalised resolvent in the sense of \cite{Naimark1940,Naimark1943};  the operator $A_{\beta_0,\beta_1}$ then acquires a dependence on $z$ in the description of its domain, {\it cf.} \cite{Strauss}.

\section{Norm-resolvent asymptotics}
\label{norm_resolvent_section}

In the present section we make use of the Kre\u\i n formula \eqref{eq:Krein_transmission} to obtain a norm-resolvent asymptotics of the family ${A}_\e^{(\tau)}$. In doing so we compute the asymptotics of $M(z)^{-1}$ based on a Schur-Frobenius type inversion formula, having first written $M(z)$ as a 
$2\times2$ operator matrix relative to a carefully chosen orthogonal decomposition of the Hilbert space $\mathcal H$. In our analysis of operator matrices, we rely on \cite{Tretter}.

\subsection{Diagonalisation of the $M$-operator on the stiff component}

Our overall strategy is based on the ideas of \cite{GrandePreuve}. The proof of the main result is, however, obtained in a different way, although the rationale (see also \cite{Physics}) is essentially the same. Towards the end of the section we outline a sketch of an alternative line of reasoning, which mimics the machinery of  \cite{GrandePreuve}. First, we consider the $M$-function $M^\stiff(z),$ and observe that
\begin{equation}\label{eq:Mstiff}
  M^\stiff(z)=\e^{-2} \widetilde M^\stiff(\e^2 z),\qquad \widetilde M^\stiff(z):=\widetilde \Gamma_1^\stiff \widetilde S^\stiff_z,
\end{equation}
where 
$\widetilde S_z^\stiff$ is the solution operator of the problem
\begin{equation}\label{eq:A0stiff_scaled_to_1}
\begin{cases}
-(\nabla+{\rm i}\tau)^2 u_\phi-z u_\phi=0,\qquad u_\phi\in \dom A_0^\stiff\dotplus \ran\Pi_\stiff,&\\[0.4em]
\Gamma_0^\stiff  u_\phi= \phi,&
\end{cases}
\end{equation}
and $\widetilde \Gamma_1^\stiff$ is defined similarly to \eqref{eq:Gamma1softstiff_dom}, \eqref{eq:Gamma1softstiff_action}, with $\Lambda^\stiff$ replaced by $\widetilde \Lambda^\stiff=\e^2 \Lambda^\stiff,$ {\it cf.} proof of Lemma \ref{lemma:selfadjointness_of_Lambda}, and $A_0^\stiff$ replaced by $\widetilde{A}_0^\stiff,$ 
the self-adjoint operator of the problem \eqref{eq:A0stiff_scaled_to_1} subject to Dirichlet boundary condition $\phi=0.$ It is easily seen that $\widetilde \Gamma_1^\stiff$ computes traces of the co-normal derivative 
\[
\dntau u:=-\left(\frac{\partial u}{\partial n}+{\rm i}u(\tau\cdot n)\right)\biggr|_\Gamma,
\] 
where $n$ is the external unit normal to $Q_\stiff.$ (Henceforth we often drop the subscripts $+$ and $-$ in the notation $n_+$ and $n_-$ whenever the choice is clear from the context. We reiterate that the normal is always chosen in the outward direction to the domain considered.)
The $M$-function $\widetilde M^\stiff(z)$ is therefore the classical $z$-dependent DN map pertaining to the magnetic Laplacian on $Q_\stiff$ (with periodic boundary conditions on $\partial Q$, if appropriate). Notice that it does not depend on the parameter $\e.$ The representation \eqref{eq:M_representation} applied to $\widetilde M^\stiff (z)$ together with \eqref{eq:Mstiff} implies that
\begin{align}
M^\stiff(z)= \Lambda^\stiff + z\Pi_\stiff^*\bigl(1-\e^2 z (\widetilde A_0^\stiff)^{-1}\bigr)^{-1}\Pi_\stiff&=\Lambda^\stiff + z\Pi_\stiff^*\Pi_\stiff+O(\e^2)\nonumber\\[0.3em]
&=\e^{-2}\widetilde\Lambda^\stiff + z\Pi_\stiff^*\Pi_\stiff+O(\e^2),
\label{eq:Mstiff_asymp}
\end{align}
with a uniform estimate on the remainder term.
The asymptotic expansion above follows from the obvious (due to the Rayleigh principle) estimate ({\it cf.} (\ref{A0_bound}))
\[
\bigl\|(\widetilde A_0^\stiff)^{-1}\bigr\|_{L^2(Q_\stiff)\to L^2(Q_\stiff)}\leq C.
\]
The second term of the asymptotic expansion (\ref{eq:Mstiff_asymp}) cannot be dropped since $\widetilde\Lambda^\stiff$ is singular at $\tau=0$ and, by continuity of DN maps (see \cite{Friedlander_old,Friedlander}), near-singular in the vicinity of $\tau=0$. In fact, in the setup of Model II it is singular at \emph{all} values of quasimomentum $\tau$. It is  expected that the term $O(\e^2)$ in (\ref{eq:Mstiff_asymp}) is asymptotically irrelevant from the point of view of the analysis of the resolvent $({A}_\e^{(\tau)}-z)^{-1},$ which we indeed show in what follows. 

The above argument suggests that we need to separate the singular and non-singular parts of $\widetilde \Lambda^\stiff.$ In order to do so, we recall that $\widetilde \Lambda^\stiff$ is a self-adjoint operator in ${\mathcal H},$ with domain $H^1(\Gamma)$. Clearly, its spectrum (``Steklov spectrum") consists of the so-called Steklov eigenvalues,
defined as values $\mu$ such that the problem 
\begin{equation}
\label{psi_eq}
\widetilde{\Lambda}^\stiff\psi=\mu\psi, \qquad\psi\in\dom\widetilde{\Lambda}^\stiff,
\end{equation}
or equivalently the problem
\begin{equation}
\label{eq:Steklov_stiff}
\begin{cases}
(\nabla+{\rm i}\tau)^2 u=0,\ \ u\in H^2(Q_\stiff),&\\[0.4em]
\dntau u=\mu u\ \ {\rm on}\ \Gamma,&
\end{cases}
\end{equation}
has a non-trivial solution. (Solutions to (\ref{psi_eq}) are the boundary values of the corresponding solutions to (\ref{eq:Steklov_stiff}).) 
If $\Gamma$ is 
Lipschitz, the Steklov spectrum is discrete, accumulates\footnote{Due to our choice of $\dntau,$ Dirichlet-to-Neumann maps defined in this paper are non-positive, as opposed to the standard definition, see {\it e.g.} \cite{Friedlander_old}.} to $-\infty,$ and obeys a Weyl-type asymptotics, see {\it e.g.} \cite{Polterovich} and references therein.

The least (by absolute value) Steklov eigenvalue $\mu_\tau$ clearly vanishes at $\tau=0$ (the corresponding solution 
of \eqref{eq:Steklov_stiff} is a constant). By the analysis of \cite[Lemma 7]{Friedlander}, it is quadratic in $\tau$ in the vicinity of $\tau=0$ in the case of Model I and 
 is identically zero for all $\tau$ in the case of Model II. The eigenvalue $\mu_\tau$ is simple, and we henceforth denote the associated normalised eigenfunction of $\widetilde \Lambda^\stiff$ (``Steklov eigenvector") by $\psi_\tau\in{\mathcal H}.$ Introduce the corresponding orthogonal  ${\mathcal H}$-projection $\P:=\langle \cdot, \psi_\tau\rangle_{\mathcal H}\psi_\tau$ and notice that it is a spectral projection relative to $\widetilde \Lambda^\stiff.$ 
 Next, we decompose the boundary space ${\mathcal H}$: 
\begin{equation}
{\mathcal H}=\P{\mathcal H}\oplus \Port{\mathcal H},
\label{space_decomp}
\end{equation}
where $\Port:=I-\P$ is the mutually orthogonal projection to $\P,$ which yields the following matrix representation for $\widetilde \Lambda^\stiff$ (and therefore a similar representation for $\Lambda^\stiff=\varepsilon^{-2}\widetilde{\Lambda}^\stiff$):
\begin{equation}
\label{diagonal_lambda}
\widetilde \Lambda^\stiff=\begin{pmatrix}
                        \mu_\tau  & 0 \\[0.4em]
                        0 & \widetilde\Lambda^\stiff_\bot
                      \end{pmatrix},
\end{equation}
where $\widetilde\Lambda^\stiff_\bot:=\Port \widetilde\Lambda^\stiff \Port$ is treated as a self-adjoint operator in $\Port {\mathcal H}$. Likewise, a matrix representation can be written for $M^\stiff$ and $\widetilde M^\stiff$, since they are obtained as bounded additive perturbations of 
$\Lambda^\stiff$ and $\widetilde\Lambda^\stiff$, respectively, see (\ref{eq:Mstiff_asymp}) or Proposition \ref{prop:M}.

\subsection{First asymptotic result for $A_\varepsilon^{(\tau)}$}
Our strategy is to write the operator $M(z)$
as a block-operator matrix relative to the decomposition (\ref{space_decomp}),
followed by an application of the Schur-Frobenius inversion formula, see \cite{Tretter} for the current state of the art in the area. The analysis is governed by the properties of 
$\Lambda,$ since by Proposition \ref{prop:M} the operator $M(z)$ is its bounded perturbation. Although $\P$ is not in general a spectral projection of $\Lambda$ (unless $\tau=0$), the fact that $\P {\mathcal H}$ is one-dimensional simplifies the analysis. 

Indeed, for all $\psi\in\dom \Lambda$ one has $\P\psi\in \dom \Lambda,$ and therefore $\Port \Lambda\P$ is well defined on $\dom \Lambda.$ Similarly, $\Port\psi=\psi-\P\psi\in \dom \Lambda,$ and $\P \Lambda\Port$ is also well-defined. Furthermore, by the self-adjointness of $\Lambda,$ one has
$$
\P \Lambda\Port\psi=\bigl\langle \Lambda\Port\psi,\psi_\tau\bigr\rangle \psi_\tau=\bigl\langle\Port\psi,\Lambda\psi_\tau\bigr\rangle \psi_\tau, \qquad \psi\in{\mathcal H},
$$
and therefore 
\[
\bigl\|\P \Lambda\Port\bigr\|_{{\mathcal H}\to{\mathcal H}}\leq\|\Lambda\psi_\tau\|_{\mathcal H}.
\] 
A similar calculation then shows that $\P \Lambda\Port$, $\Port \Lambda\P$ and $\P \Lambda\P$ are all extendable to bounded operators in $\mathcal H,$ and $\Port \Lambda\Port$ is an (unbounded) self-adjoint operator in $\Port \mathcal H$. It follows that for all $z\in\rho(A_0)$ the operator $M(z)$ admits a block-operator matrix representation with respect to the decomposition (\ref{space_decomp}), as follows:
\begin{equation}
M(z)=\begin{pmatrix}
       {\mathbb A} & {\mathbb B} \\[0.4em]
       {\mathbb E} & {\mathbb D}
     \end{pmatrix}, \qquad {\mathbb A}, {\mathbb B}, {\mathbb E} \text{ bounded.}
\label{PreSchur}
\end{equation}
For evaluating $M(z)^{-1}$ we use the Schur-Frobenius inversion formula \cite[Theorem 2.3.3]{Tretter} ({\it cf.} \cite{Fuerer})
\begin{equation}\label{eq:SchurF}
\overline{\begin{pmatrix}
       {\mathbb A} & {\mathbb B} \\
       {\mathbb E} & {\mathbb D}
     \end{pmatrix}}^{-1}=
\begin{pmatrix}
  {\mathbb A}^{-1}+\overline{{\mathbb A}^{-1}{\mathbb B}}(\overline{{\mathbb S}})^{-1}{\mathbb E}{\mathbb A}^{-1} & -\overline{{\mathbb A}^{-1}{\mathbb B}}(\overline{{\mathbb S}})^{-1} \\[0.4em]
  -(\overline{{\mathbb S}})^{-1}{\mathbb E}{\mathbb A}^{-1} & (\overline{{\mathbb S}})^{-1}
\end{pmatrix},\qquad
{\mathbb S}:={\mathbb D}-{\mathbb E} {\mathbb A}^{-1} {\mathbb B}.
\end{equation}
In applying this formula to (\ref{PreSchur}), we use the facts that $M(z)$ is closed
 and ${\mathbb A}^{-1}$ is bounded uniformly in $\tau$ for $z\in K_\sigma$ by Proposition \ref{prop:M} as the inverse of a (scalar) $R$-function \cite[Section 2.10]{Nussenzveig} (also known as Herglotz, or Nevanlinna, function).\footnote{\label{foot_page}By a suggestion of a referee, we provide an alternative rough estimate, which yields a uniform bound on ${\mathbb A}^{-1},$ independently of complex analysis and, in particular, analytic properties of Herglotz functions. Note that by Proposition \ref{prop:M},  one has $\Im M(z) = (\Im z) S^{*}_{\bar z} S_{\bar z}.$ 
 Moreover, since $S_z=(1-zA_0^{-1})^{-1}\Pi,$ one has 
 \[
 S^{*}_{\bar z} S_{\bar z}=\Pi^*(1-zA_0^{-1})^{-1}(1-{\bar z}A_0^{-1})^{-1}\Pi,
 \]
 and therefore, for some $\tau$-independent constants $c_1, c_2>0,$ 
 \[
\langle S^{*}_{\bar z} S_{\bar z}\P\phi, \P\phi\rangle_{\mathcal H}=\bigl\Vert(1-{\bar z}A_0^{-1})^{-1}\Pi\P\phi\bigr\Vert_H^2\ge c_1\Vert\Pi\P\phi\Vert_H^2\ge c_1\Vert\Pi_\stiff\P\phi\Vert_H^2\ge c_2\Vert\P\phi\Vert_{\mathcal H}^2\quad \forall\phi\in{\mathcal H}, \tau\in Q, z\in K_\sigma.
 \]
 where we have used the fact that the operator $A_0$ is bounded below by a positive constant independent of $\tau$ as well as the continuity of $\Vert\Pi_\stiff\psi_\tau\Vert_H$ with respect to $\tau,$ which in turn is a direct consequence of \cite[Proposition 2.2]{Friedlander_old}. The required bound for ${\mathbb A}^{-1}=(\P M(z)\P)^{-1}$ now follows immediately.}
 As for the Schur-Frobenius complement ${\mathbb S}$ of the block ${\mathbb A}$, it is closed since ${\mathbb D}$ is closed (see below) and ${\mathbb E}{\mathbb A}^{-1}{\mathbb B}$ is bounded.

We next show that ${\mathbb S}$  is boundedly invertible, with a uniformly small 
norm of the inverse. In doing so, we use the fact that 
by construction there is a bounded operator $\widetilde B$ on $\Port{\mathcal H}$ such that
$
{\mathbb D}={\mathfrak A}+{\mathfrak B}
+\widetilde B, 
$
where ${\mathfrak A}:=\Port \Lambda^\soft\Port,$ ${\mathfrak B}:=\Port \Lambda^\stiff \Port$ are self-adjoint operators in $\Port{\mathcal H}$ on the same domain. In addition, ${\mathfrak B}=\e^{-2}{\mathfrak B}_1$ 
with an $\e$-independent boundedly invertible ${\mathfrak B}_1,$
so that the uniform bound 
$$
\bigl\|{\mathfrak B}^{-1}\bigr\|_{\Port{\mathcal H}\to\Port{\mathcal H}}\leq C \e^2
$$
holds. 
By \cite[Chapter IV, Remark 1.5]{Kato} there exist $\alpha,\beta>0$ such that
$$
\|\mathfrak A u\|\leq \alpha \|\mathfrak B_1 u\|+\beta \|u\|\qquad \forall u\in \dom{\mathfrak A}=\dom{\mathfrak B}_1.
$$
Since $\mathfrak B_1$ is boundedly invertible, for each $v$ one can pick $u:=\mathfrak B_1^{-1}
v\in\mathfrak\dom\,{\mathfrak B}_1,$ so that
$$
\|\mathfrak A \mathfrak B_1^{-1} v\|\leq \alpha \|v\|+\beta\|\mathfrak B_1^{-1} v\|\leq C\|v\|,\qquad C>0.
$$
It follows that $\mathfrak A \mathfrak B_1^{-1}$ is bounded, and hence $\|\mathfrak A \mathfrak B^{-1}\|\leq C\e^2$. Therefore, 
\[
{\mathbb D}^{-1}=(\mathfrak A +\mathfrak B +\widetilde B)^{-1}=\mathfrak B^{-1}(I+\mathfrak A \mathfrak B^{-1}+\widetilde B \mathfrak B^{-1})^{-1},
\] 
and thus
$
\|{\mathbb D}^{-1}\|\leq C\e^2.
$
It remains to use the formula ${\mathbb S}^{-1}=(I-{\mathbb D}^{-1}{\mathbb E}{\mathbb A}^{-1}{\mathbb B})^{-1}{\mathbb D}^{-1}$, from which it follows immediately that ${\mathbb S}^{-1}=O(\e^2).$

Returning to \eqref{eq:SchurF}, we obtain
\begin{equation}
\label{eq:estimatefortheorem}
M(z)^{-1}=\begin{pmatrix}
       {\mathbb A} & {\mathbb B} \\[0.3em]
       {\mathbb E} & {\mathbb D}
     \end{pmatrix}^{-1}=
\begin{pmatrix}
  {\mathbb A}^{-1}& 0\\[0.3em]
  0 & 0
\end{pmatrix}
+O(\e^2),
\end{equation}
with a $\tau$-uniform estimate for the remainder term, where, we recall, ${\mathbb A}=\P M(z) \P$. Comparing the result of substituting (\ref{eq:estimatefortheorem}) into (\ref{eq:Krein_transmission}) with \eqref{eq:Krein_general}, 
where $\beta_0=\Port,$ $\beta_1=\P$, we arrive at the following statement.

\begin{theorem}
\label{thm:NRA}
There exists $C>0$ such that for all $z\in K_\sigma,$ $\tau\in Q'$, $\e<\e_0$ one has 
$$
\Bigl\|\bigl({A}_\e^{(\tau)}-z\bigr)^{-1}-\bigl(A_{\Port,\P}-z\bigr)^{-1}\Bigr\|_{H\to H}\leq C\e^2.
$$
\end{theorem}
\begin{proof}
For the resolvent $({A}_\e^{(\tau)}-z\bigr)^{-1}$ the formula (\ref{eq:Krein_transmission}) is applicable, in which for $M(z)^{-1}$ we use (\ref{eq:estimatefortheorem}). As for the resolvent 
$\bigl(A_{\Port,\P}-z\bigr)^{-1},$ Proposition \ref{prop:operator} with $\beta_0=\Port,$ $\beta_1=\P$  
is clearly applicable. 
Moreover, for this choice of $\beta_0,$ $\beta_1,$ the operator
$$
Q_{\Port, \P}(z)=-\bigl(\overline{\Port+\P M(z)}\bigr)^{-1}\P
$$
in (\ref{eq:Krein_general}) is easily computable ({\it e.g.}, by the Schur-Frobenius inversion formula of \cite{Tretter}, see \eqref{eq:SchurF})\footnote{We remark that $\Port+\P M(z)$ is triangular (${\mathbb A}=\P M(z)\P,$ ${\mathbb B}=\P M(z)\Port,$ ${\mathbb E}=0,$ ${\mathbb D}=I$ in \eqref{eq:SchurF}) with respect to the decomposition $\mathcal H=\P\mathcal H\oplus \Port \mathcal H$.}, yielding
\begin{equation}
Q_{\Port, \P}(z)=-\P\bigl(\P M(z) \P\bigr)^{-1}\P,
\label{Qadd}
\end{equation}
and the claim follows.
\end{proof}

It can be shown that the operator $A_{\Port,\P}$ in the above theorem is self-adjoint. We do not require this fact in the analysis to follow and therefore skip the proof of this statement. We note further that the statement of Theorem \ref{thm:NRA}  gives an answer to the problem of homogenisation we set out to analyse. Yet, the result in this form is rather inconvenient, as $(\P M(z)\P)^{-1}$ is not written in an explicit form; the operator $A_{\Port,\P},$ which is the norm-resolvent asymptotics of the family ${A}_\e^{(\tau)},$ is therefore not easily analysed. We next address this issue.

\subsection{Generalised resolvent on the soft component} 

First, we rewrite the norm-resolvent asymptotics obtained in Theorem \ref{thm:NRA} in a block-matrix form, this time, however, relative to the decomposition 
\[
H=P_\stiff H\oplus P_\soft H =L^2(Q_\stiff)\oplus L^2(Q_\soft).
\] 
This allows us to express the asymptotics of $({A}_\e^{(\tau)}-z)^{-1}$ in terms of the asymptotics of the generalised resolvent 
\begin{equation}
R_\e^{(\tau)}(z):=P_\soft ({A}_\e^{(\tau)}-z)^{-1}P_\soft,
\label{gen_res_star}
\end{equation}
which is derived from Theorem \ref{thm:NRA} as follows.

\begin{lemma}\label{lemma:gen_resolvent}
One has
\begin{equation}\label{eq:quasi-Krein1}
R_\e^{(\tau)}(z)=(A_0^\soft-z)^{-1}-
 S_z^\soft\bigl(M^\soft(z)-B^{(\tau)}(z)\bigr)^{-1}(S_{\bar z}^\soft)^*,\ \ z\in K_\sigma,
\end{equation}
where
$
B^{(\tau)}:=-M^\stiff.
$
\end{lemma}
\begin{remark}
The right-hand side of (\ref{eq:quasi-Krein1}) is well defined for all $z\in K_\sigma$ by the proof of Theorem \ref{thm:NRA}, or, in other words, due to the analytic properties of the operator functions $M^\soft$ and $B^{(\tau)}.$
\end{remark}

The \emph{proof}of Lemma \ref{lemma:gen_resolvent} follows immediately from \eqref{eq:Krein_transmission} and Lemma \ref{lemma:M_additivity}.

Comparing the statements of Lemma \ref{lemma:gen_resolvent} and Proposition \ref{prop:operator}, one realises that the generalised resolvent $R_\e^{(\tau)}$ is the solution operator of a spectral BVP, with the parameter $z$ appearing not only in the differential equation but also in the boundary condition. Namely, we have the following statement.

\begin{lemma}
The generalised resolvent $R_\e^{(\tau)}(z)$ is the solution operator of the following BVP on $L^2(Q_\soft)$:
\begin{align}
-&(\nabla + {\rm i}\tau)^2 u-zu=f, \quad f\in L^2(Q_\soft),\nonumber\\[0.2em]
&\Gamma_1^\soft u=B^{(\tau)}(z)\Gamma_0^\soft u.\label{bc_star}
\end{align}
The boundary condition (\ref{bc_star}) can be written in the more conventional form
$$
\dntau u = B^{(\tau)}(z)u\ \ {\rm on}\ \Gamma,\qquad \dntau u:=-\biggl(\frac{\partial u}{\partial n}+{\rm i}u(\tau \cdot n)\biggr)\biggr|_\Gamma,
$$
where $n$ is the outward normal to the boundary $\Gamma$ of $Q_\soft$. Equivalently,
$$
R_\e^{(\tau)}(z)=\bigl(A^\soft_{-B^{(\tau)}(z),I}-z\bigr)^{-1},
$$
where $A^\soft_{-B^{(\tau)}(z),I}$ for any fixed $z$ is the operator in $L^2(Q_\soft)$ defined by Proposition \ref{prop:operator} relative to the triple $(\mathcal H, \Pi_\soft, \Lambda^\soft)$. This operator is maximal anti-dissipative for $z\in\mathbb C_+$ and maximal dissipative for $z\in \mathbb C_-,$ see \cite{Strauss1968}.
\end{lemma}

\begin{proof}
Follows immediately from Proposition \ref{prop:operator} applied to the triple (in the sense of \cite{Ryzh_spec}) $(\mathcal H, \Pi_\soft, \Lambda^\soft)$ with the following choice of the operators parameterising boundary conditions: $\beta_1=I, \beta_0=-B^{(\tau)}=M^\stiff.$
\end{proof}

\begin{remark} 
The above lemma is a realisation of the corresponding general result of \cite{Strauss}.
\end{remark}

Theorem \ref{thm:NRA} now implies a uniform asymptotics for the generalised resolvents $R_\e^{(\tau)}$ as $\e\to 0$.

\begin{theorem}\label{thm:NRA_gen}
The operator family $R_\e^{(\tau)}(z),$ see (\ref{gen_res_star}), admits the following asymptotics in the operator-norm topology for $z\in K_\sigma$ and $\tau\in Q'$:
$$
R_\e^{(\tau)}(z)-R_\eff^{(\tau)}(z)=O(\e^2),
$$
where $R_\eff^{(\tau)}(z)$ is the solution operator for the following spectral BVP on the soft component $Q_\soft$:
\begin{equation}
\label{eq:BVPBz}
\begin{aligned}
-&(\nabla + {\rm i}\tau)^2 u-zu=f, \quad f\in L^2(Q_\soft),\\[0.2em]
&\beta_0(z)\Gamma_0^\soft u+ \beta_1\Gamma_1^\soft u=0,
\end{aligned}
\end{equation}
with $\beta_0(z)=\Port-\P B^{(\tau)}(z)\P$ and $\beta_1=\P$.

The boundary condition in (\ref{eq:BVPBz}) can be written in the more conventional form
$$
\Port u|_\Gamma =0, \quad \P \dntau u = \P B^{(\tau)}(z)\P u\bigl|_\Gamma.
$$
Equivalently,
$$
R_\e^{(\tau)}(z)-\Bigl(A^\soft_{\Port-\P B^{(\tau)}(z)\P,\P}-z\Bigr)^{-1}=O(\e^2),
$$
where $A^\soft_{\Port-\P B^{(\tau)}(z)\P,\P},$ for any fixed $z,$ is the operator in $L^2(Q_\soft)$ defined by Proposition \ref{prop:operator} relative to the triple $(\mathcal H, \Pi_\soft, \Lambda^\soft)$,  where the term ``triple" is understood in the sense of \cite{Ryzh_spec}. This operator is maximal anti-dissipative for $z\in\mathbb C_+$ and maximal dissipative for $z\in \mathbb C_-,$ see \cite{Strauss1968}.

\end{theorem}

\begin{proof}
On the one hand, by Theorem \ref{thm:NRA}, the resolvent $({A}_\e^{(\tau)}-z)^{-1}$ is $O(\e^2)$-close to 
$$
\bigl(A_{\Port,\P}-z\bigr)^{-1}=(A_0-z)^{-1}-S_z\bigl(\overline{\Port+\P M(z)}\bigr)^{-1}\P S_{\bar z}^*,
$$
and therefore
\begin{align}
R_\e^{(\tau)}(z)&=P_\soft (A_0-z)^{-1} P_\soft- P_\soft S_z\bigl(\overline{\Port+\P M(z)}\bigr)^{-1}\P S_{\bar z}^*P_\soft
+O(\e^2)
\nonumber\\[0.4em]
&=(A_0^\soft-z)^{-1}-S_z^\soft\bigl(\overline{\Port+\P M(z)}\bigr)^{-1}\P (S_{\bar z}^\soft)^*+O(\e^2)
\nonumber\\[0.4em]
&=(A_0^\soft-z)^{-1}-S_z^\soft \P\bigl(\P M(z)\P\bigr)^{-1}\P (S_{\bar z}^\soft)^*+O(\e^2)
\nonumber\\[0.4em]
&=(A_0^\soft-z)^{-1}-S_z^\soft \P\bigl(\P M^\soft(z)\P-\P B^{(\tau)}(z)\P\bigr)^{-1}\P (S_{\bar z}^\soft)^*+O(\e^2),
\label{eq:calc_gen_res}
\end{align}
where we have used \eqref{eq:Szstiffsoft}.

On the other hand, by Proposition \ref{prop:operator} the generalised resolvent of the BVP \eqref{eq:BVPBz} admits the form
\begin{align}
R_{\rm eff}^{(\tau)}(z)&=\bigl(A_0^\soft-z\bigr)^{-1}-S_z^\soft\bigl(\overline{\Port-\P B^{(\tau)}(z)\P+\P M^\soft(z)}\bigr)^{-1}\P (S_{\bar z}^\soft)^*\nonumber\\[0.4em]
&=\bigl(A_0^\soft-z\bigr)^{-1}-S_z^\soft \P\bigl(\P M^\soft(z)\P-\P B^{(\tau)}(z)\P\bigr)^{-1}\P (S_{\bar z}^\soft)^*,
\label{eq:BVPBz1}
\end{align}
by an application of the Schur-Frobenius inversion formula \eqref{eq:SchurF}. Comparing the right-hand sides of \eqref{eq:calc_gen_res} and \eqref{eq:BVPBz1} completes the proof.
\end{proof}

Theorem \ref{thm:NRA_gen} can be further clarified by the following construction. Consider the ``truncated'' boundary space\footnote{In what follows we consistently supply the (finite-dimensional)  ``truncated'' spaces and operators pertaining to them by the breve overscript.} 
$\breve {\mathcal H}:=\P \mathcal H$ (note, that in our setup $\breve{\mathcal H}$ is one-dimensional). Introduce the truncated $\tau$-harmonic lift on $\breve{\mathcal H}$ by $\breve{\Pi}_\soft:=\Pi_\soft|_{\breve{\mathcal H}}$ and the truncated operator $\breve{\Lambda}^\soft:=\P\Lambda^\soft\vert_{\breve {\mathcal H}}.$
We prove the following result.

\begin{theorem}
\label{thm:NRA_gen_truncated}

1. The formula
\begin{equation}\label{eq:gen_res_trunc}
R_\eff^{(\tau)}(z)=\bigl(A_0^\soft-z\bigr)^{-1}-\breve{S}_z^\soft\bigl(\breve{M}_\soft(z)-\P B^{(\tau)}(z)\P\bigr)^{-1} (\breve{S}_{\bar z}^\soft)^*
\end{equation}
holds, where $\breve{S}_z^\soft$ is the solution operator of the problem
$$
\begin{aligned}
-(\nabla+{\rm i}\tau)^2 u_\phi-z u_\phi&=0,\ \ u_\phi\in \dom A_0^\soft\dotplus \ran\breve{\Pi}_\soft,\\[0.1cm]
\Gamma_0^\soft u_\phi&= \phi, \quad \phi\in \breve{\mathcal H},
\end{aligned}
$$
and $\breve{M}^\soft$ is the $M$-operator defined in accordance with \eqref{defn:M-function}, \eqref{defn:M-stiffsoft} relative to the triple $(\breve{\mathcal H}, \breve{\Pi}_\soft, \breve{\Lambda}^\soft)$.

2. The ``effective'' resolvent $R_\eff^{(\tau)}(z)$ is represented as the generalised resolvent of the problem
\begin{equation}\label{eq:BVPBz_mod}
\begin{gathered}
-(\nabla + {\rm i}\tau)^2 u-zu =f, \quad f\in L^2(Q_\soft),\quad u\in \dom A_0^\soft\dotplus \ran \breve{\Pi}_\soft,\\[0.4em]
\P\dntau u = \P B^{(\tau)}(z)\P u\bigr|_\Gamma.
\end{gathered}
\end{equation}

3. The triple $(\breve{\mathcal H},\breve{\Gamma}_0^\soft, \breve{\Gamma}_1^\soft)$ is the classical boundary triple \cite{Gor,DM} for the operator $A_{\max}$ defined by the differential expression $-(\nabla + {\rm i}\tau)^2$ on the domain $\dom A_{\max}=\dom A_0^\soft\dotplus \ran\breve{\Pi}_\soft.$ Here $\breve{\Gamma}_0^\soft$ and $\breve{\Gamma}_1^\soft$ are defined on $\dom A_{\max}$ as the operator of the trace on the boundary $\Gamma$ and 
$\P\dntau,$ respectively.
\end{theorem}
\begin{corollary}
\label{MHerglotz}
The operator $\breve{M}_\soft(z)$ is the classical $M$-matrix\footnote{Given that $\breve{\mathcal H}$ is one-dimensional, one is tempted to refer to the $M$-matrix here as the Weyl-Titchmarsh $m$-coefficient.} of $A_{\max}$ relative to the boundary triple $(\breve{\mathcal H},\breve{\Gamma}_0^\soft, \breve{\Gamma}_1^\soft)$.
\end{corollary}

\begin{proof}[Proof of Theorem \ref{thm:NRA_gen_truncated}]
By the definition of $S_z^\soft$, one has $S_z^\soft=(1-z(A_0^\soft)^{-1})^{-1}\Pi_\soft$, and therefore $\breve{S}_z^\soft=S_z^\soft\P|_{\breve{\mathcal H}}$. Thus \eqref{eq:BVPBz} is equivalent to ({\it cf.} \eqref{eq:BVPBz1})
$$
R_\eff^{(\tau)}(z)=\bigl(A_0^\soft-z\bigr)^{-1}-\breve{S}_z^\soft\bigl(\P M^\soft(z)\P-\P B^{(\tau)}(z)\P\bigr)^{-1} (\breve{S}_{\bar z}^\soft)^*.
$$
Next, we check that $\breve{M}_\soft (z)=\P M^\soft(z)\P\vert_{\breve{\mathcal H}}$. Indeed, Proposition \ref{prop:M} implies
$$
M^\soft(z)=\Lambda^\soft + z \Pi^*_\soft\bigl(1-z (A_0^\soft)^{-1}\bigr)^{-1}\Pi_\soft,
$$
and therefore
$$
\P M^\soft(z)\P\vert_{\breve{\mathcal H}}=\breve{\Lambda}^\soft + z \breve{\Pi}^*_\soft\bigl(1-z (A_0^\soft)^{-1}\bigr)^{-1}\breve{\Pi}_\soft,
$$
from which the first claim follows.

We now turn our attention to the operator $\breve{\Gamma}_1^\soft$. It is defined on $(A_0^\soft)^{-1}L^2(Q_\soft)\dotplus \breve{\Pi}_\soft \breve{\mathcal H}$, where we have taken into account that $\breve{\Lambda}^\soft$ is bounded. Its action is set by
$$
\breve{\Gamma}_1^\soft: (A_0^\soft)^{-1}f+\breve{\Pi}_\soft \phi \mapsto \breve{\Pi}_\soft^* f + \breve{\Lambda}^\soft \phi.
$$
Using the definitions of $\breve{\Pi}_\soft$ and $\breve{\Lambda}^\soft$, one has
$$
\breve{\Gamma}_1^\soft\bigl((A_0^\soft)^{-1}f+\breve{\Pi}_\soft \phi\bigr)=\P \Pi_\soft^*f+\P\Lambda^\soft \P\phi,
$$
where we have taken into account that $\phi\in\breve{\mathcal H}$. Comparing the latter expression with the definition of $\Gamma_1^\soft$, we obtain
$$
\breve{\Gamma}_1^\soft=\P \Gamma_1^\soft\bigr|_{(A_0^\soft)^{-1}L^2(Q_\soft)\dotplus \ran \breve{\Pi}_\soft},
$$
as required. The triple property now follows from \eqref{eq:quasi_triple_property} and dimensional considerations.
\end{proof}

\subsection{Simplified asymptotics for $(A_\varepsilon^{(\tau)}-z)^{-1}$}

Equipped with Theorems \ref{thm:NRA_gen} and \ref{thm:NRA_gen_truncated}, we are now set to provide a convenient representation for the asymptotics of $({A}_\e^{(\tau)}-z)^{-1}$ obtained in Theorem \ref{thm:NRA}.

\begin{theorem}
\label{thm:main}
The resolvent $({A}_\e^{(\tau)}-z)^{-1}$ admits the following asymptotics in the uniform operator-norm topology:
$$
({A}_\e^{(\tau)}-z)^{-1}=\mathcal R^{(\tau)}_\eff(z)+ O(\e^2),
$$
where the operator $\mathcal R^{(\tau)}_\eff(z)$, which in general depends on $\e$, has the following representation relative to the decomposition $H=P_\soft H\oplus P_\stiff H=L^2(Q_\soft)\oplus L^2(Q_\stiff)$:
\begin{equation}\label{eq:NRA}
\mathcal R^{(\tau)}_\eff(z)=
\begin{pmatrix}
R^{(\tau)}_\eff(z)&\ \ \Bigl(\mathfrak K_{\bar z}^{(\tau)}\bigl[R_\eff^{(\tau)}(\bar z)-(A_0^{\soft}-\bar z)^{-1}\bigr]\Bigr)^*\Pi_\stiff^*\\[0.9em] \Pi_\stiff\mathfrak{K}^{(\tau)}_z \bigl[R_\eff^{(\tau)}(z)-(A_0^{\soft}-z)^{-1}\bigr] & \ \ \Pi_\stiff\mathfrak K_{z}^{(\tau)}\Bigl(\mathfrak K_{\bar z}^{(\tau)}\bigl[R_\eff^{(\tau)}(\bar z)-(A_0^{\soft}-\bar z)^{-1}\bigr]\Bigr)^*\Pi_\stiff^*
\end{pmatrix}.
\end{equation}
Here $\mathfrak{K}^{(\tau)}_z:=\Gamma_0^\soft|_{\mathfrak N_z}$, where $\mathfrak N_z:=\ran\bigl(S_z^\soft\P\bigr)$, $z\in \mathbb C_\pm$.
\end{theorem}

\begin{proof}
First, we note that since $\ran(S_z^\soft\P)$ is finite-dimensional (one-dimensional, to be precise), the operator $\mathfrak{K}^{(\tau)}_z$ is well defined as a bounded linear operator from $\mathfrak N_z$ to $\mathcal H,$ where $\mathfrak N_z$ is equipped with the standard norm of $L^2(Q_\soft).$ We proceed by representing the operator $(A_{\Port,\P}-z)^{-1},$ see Theorem \ref{thm:NRA}, in a block-operator matrix form relative to the orthogonal decomposition $H=L^2(Q_\soft)\oplus L^2(Q_\stiff)$. We compare the norm-resolvent asymptotics $(A_{\Port,\P}-z)^{-1},$ provided by Theorem \ref{thm:NRA}, with $R_\eff^{(\tau)}(z),$ which is $O(\varepsilon^2)$-close to $P_\soft (A_{\Port,\P}-z)^{-1} P_\soft,$ as established by Theorem \ref{thm:NRA_gen}, and write
\begin{align*}
P_\stiff\bigl(A_{\Port, \P} -z\bigr)^{-1}P_\soft &= -S^\stiff_z\P\bigl (\P M^\soft(z)\P-\P B^{(\tau)}(z)\P\bigr)^{-1}\P\bigl(S^\soft_{\bar z}\bigr)^*\\[0.4em]
&=-S^\stiff_z\Gamma_0^\soft S^\soft_z\P\bigl (\P M^\soft(z)\P-\P B^{(\tau)}(z)\P\bigr)^{-1}\P \bigl(S^\soft_{\bar z}\bigr)^*\\[0.55em]
&=S^\stiff_z \Gamma_0^\soft\bigl[R_\eff^{(\tau)}(z)-(A_0^{\soft}-z)^{-1}\bigr]
\\[0.55em]
&
=S^\stiff_z \mathfrak{K}^{(\tau)}_z\bigl[R_\eff^{(\tau)}(z)-(A_0^{\soft}-z)^{-1}\bigr].
\end{align*}
Here in the first equality we use Proposition \ref{prop:operator} along with (\ref{Qadd}), in second equality we use the fact that $\Gamma_0^\soft S^\soft_z=I,$ and in the third equality we use Proposition \ref{prop:operator} again, see also (\ref{eq:BVPBz1}). Invoking (\ref{SPi_stiff})
completes the proof in relation to the bottom-left element of \eqref{eq:NRA}.

Passing over to the top-right entry in \eqref{eq:NRA}, we write
\begin{align*}
P_\soft\bigl(A_{\Port,\P} -z\bigr)^{-1}P_\stiff&= -S^\soft_z\P\bigl (\P M^\soft(z)\P-\P B^{(\tau)}(z)\P\bigr)^{-1}\P(S^\stiff_{\bar z})^*\\[0.4em]
&=\Bigl(\mathfrak K_{\bar z}^{(\tau)}\bigl[R_\eff^{(\tau)}(\bar z)-(A_0^{\soft}-\bar z)^{-1}\bigl]\Bigr)^*(S^\stiff_{\bar z})^*\\[0.4em]
&=\Bigl(\mathfrak K_{\bar z}^{(\tau)}\bigl[R_\eff^{(\tau)}(\bar z)-(A_0^{\soft}-\bar z)^{-1}\bigr]\Bigr)^*\Pi_\stiff^*\bigl(1-z (A_0^{\stiff})^{-1}\bigr)^{-1},
\end{align*}
and the claim pertaining to the named entry follows by a virtually unchanged argument.

Finally, for the bottom-right entry in \eqref{eq:NRA} we have
$$
P_\stiff\bigl(A_{\Port, \P} -z\bigr)^{-1}P_\stiff =
\bigl(A_0^{\stiff}-z\bigr)^{-1}+S^\stiff_z \mathfrak K_{z}^{(\tau)}\Bigl(\mathfrak K_{\bar z}^{(\tau)}\bigl[R_\eff^{(\tau)}(\bar z)-\bigl(A_0^{\soft}-\bar z\bigr)^{-1}\bigr]\Bigr)^*\bigl(S^\stiff_{\bar z}\bigr)^*,
$$
which completes the proof.
\end{proof}

The representation for $R_{\rm eff}^{(\tau)}(z)$ given by Theorem \ref{thm:NRA_gen_truncated}
allows us to further simplify the leading-order term of the asymptotic expansion for $({A}_\e^{(\tau)}-z)^{-1}$. Indeed, consider the operator $\P B^{(\tau)}(z)\P$ in \eqref{eq:gen_res_trunc}; since $B^{(\tau)}=-M^\stiff$ by definition, we invoke \eqref{eq:Mstiff_asymp} to obtain
$$
\P B^{(\tau)}\P=-\P\Lambda^\stiff \P-z \P \Pi_\stiff^*\Pi_\stiff \P + O(\e^2) = -\breve{\Lambda}^\stiff -z \breve{\Pi}_\stiff^*\breve{\Pi}_\stiff +O(\e^2),
$$
with a uniform estimate for the remainder term. Here the truncated $\tau$-harmonic lift $\breve{\Pi}_\stiff$ is introduced as $\breve{\Pi}_\stiff:=\Pi_\stiff\vert_{\breve{\mathcal H}}$ relative to the same truncated boundary space as above, $\breve {\mathcal H}=\P \mathcal H$. The truncated DN map $\breve{\Lambda}^\stiff$ on $\breve{\mathcal H}$ is defined by $\breve{\Lambda}^\stiff:=
\P\Lambda^\stiff\vert_{\breve{\mathcal H}},$ and Theorem \ref{thm:NRA_gen_truncated} allows us to use \cite[Theorem 6.3]{GrandePreuve} directly. As a result, we obtain
$$
R_\eff^{(\tau)}(z)-R_\hom^{(\tau)}(z)=O(\e^2),
$$
with
\begin{equation}
\label{eq:Rhom}
R_\hom^{(\tau)}(z):= (A_0^\soft-z)^{-1}-\breve{S}_z^\soft\bigl(\breve{M}_\soft(z)+
\breve{\Lambda}^\stiff + z \breve{\Pi}_\stiff^*\breve{\Pi}_\stiff\bigr)^{-1} (\breve{S}_{\bar z}^\soft)^*.
\end{equation}
By a classical result of \cite{Strauss} (see also \cite{Naimark1940,Naimark1943}), the operator $R_\hom^{(\tau)}(z)$ 
is a generalised resolvent, so it defines a 
$z$-dependent family of closed densely defined operators in $L^2(Q_\soft),$ which are maximal anti-dissipative for $z\in \mathbb C_+$ and maximal dissipative for $z\in \mathbb C_-$. Thus, we have proven the following result.
\begin{theorem}
\label{thm:main_fin}
The resolvent $({A}_\e^{(\tau)}-z)^{-1}$ admits the following asymptotics in the uniform operator-norm topology:
$$
\bigl({A}_\e^{(\tau)}-z\bigr)^{-1}=\mathcal R^{(\tau)}_\hom(z) + O(\e^2),
$$
where the operator $\mathcal R^{(\tau)}_\hom(z)$, which is allowed to depend on $\e,$ has the following representation relative to the decomposition 
$H=P_\soft H\oplus P_\stiff H=L^2(Q_\soft)\oplus L^2(Q_\stiff)$:
\begin{equation}\label{eq:NRAbreve}
\mathcal R^{(\tau)}_\hom(z)=
\begin{pmatrix}
R^{(\tau)}_\hom(z)&\ \ \Bigl(\mathfrak K_{\bar z}^{(\tau)}\bigl[R_\hom^{(\tau)}(\bar z)-(A_0^{\soft}-\bar z)^{-1}\bigr]\Bigr)^*\breve{\Pi}_\stiff^*\\[0.9em] \breve{\Pi}_\stiff\mathfrak{K}^{(\tau)}_z \bigl[R_\hom^{(\tau)}(z)-(A_0^{\soft}-z)^{-1}\bigr] & \ \ \breve{\Pi}_\stiff\mathfrak K_{z}^{(\tau)}\Bigl(\mathfrak K_{\bar z}^{(\tau)}\bigl[R_\hom^{(\tau)}(\bar z)-(A_0^{\soft}-\bar z)^{-1}\bigr]\Bigr)^*\breve{\Pi}_\stiff^*
\end{pmatrix}.
\end{equation}
Here $\mathfrak{K}^{(\tau)}_z$ is as defined in Theorem \ref{thm:main} and the generalised resolvent $R_\hom^{(\tau)}(z)$ is defined by \eqref{eq:Rhom}.
\end{theorem}

The above theorem provides us with the simplest possible leading-order term of the asymptotic expansion for $(A_\varepsilon^{(\tau)}-z)^{-1}.$ However, it is not obvious whether this leading-order term $\mathcal R^{(\tau)}_\hom(z)$ is a resolvent of some self-adjoint operator in the space $L^2(Q_\soft)\oplus \breve{\Pi}_\stiff \breve{\mathcal H}\subset H.$ In the next section we prove that this is the case, but prior to doing so we discuss the relationship between our approach to proving Theorem \ref{thm:main_fin} and the approach of \cite{GrandePreuve}.

\subsection{Discussion of strategies for proving Theorem \ref{thm:main_fin}}

We recall that in \cite{GrandePreuve} we adopted an approach to an ODE problem closely related to the one studied here, whereby we derived a result of the same type as in Theorems \ref{thm:main} and \ref{thm:main_fin}, under the only prerequisite that a statement akin to Theorem \ref{thm:NRA_gen} holds. The question thus naturally arises, whether this strategy could also be applied in the PDE setting, or it is limited to the ODE setup. Here we present a sketch of the argument giving the former answer to this alternative.\footnote{This answer was obtained independently, based on the observation (see \cite[Theorem 6.3]{GrandePreuve}) that if two generalised resolvents of the type \eqref{eq:gen_res_trunc} are parameterised by close boundary operators $B_1^{(\tau)}(z)$ and $B_2^{(\tau)}(z)$, then these generalised resolvents are themselves close. (Recall that we use this argument in the proof of Theorem \ref{thm:main_fin}.)}

Indeed, assume that one has obtained the asymptotic behaviour of the generalised resolvent $R_\e^{(\tau)}(z)$, so that there exists a generalised resolvent $R_\eff^{(\tau)}(z)$ (which can still be $\e$-dependent) such that the difference between the two admits a uniform estimate by $O(\e^r),$ $r>0$ ({\it cf.} Theorem \ref{thm:NRA_gen}). Then one can still proceed as in the proof of Theorem \ref{thm:main}, see also \cite[Theorem 7.1]{GrandePreuve}, to establish the block-matrix representation for $({A}_\e^{(\tau)}-z)^{-1}$ in terms of $R_\e^{(\tau)}(z)$. In order to prove the asymptotic expansion for the former, one then needs to establish the boundedness of a suitably redefined operator $\mathfrak{K}_z^{(\tau)}$. For this one would need to prove the boundedness of the trace operator on the space $S_z \mathcal H$, which is obviously impossible, since the mentioned space necessarily contains non-smooth functions from $\ran \Pi_{\soft}$. However, if one redefines $\mathfrak{K}_z^{(\tau)}$ to be the trace operator restricted to $\mathfrak{R}_z:=\ran\bigl(R_\e^{(\tau)}(z)-(A_0^\soft -z)^{-1}\bigr),$ the situation can be partially salvaged. Indeed, either by the main {\it a priori} estimate of \cite{Schechter} or by a direct analysis based on the compactness properties of Neumann-to-Dirichlet maps, one then proves that $\Gamma_0^\soft|_{\mathfrak{R}_z}$ is a bounded operator. This is so since, on the one hand, $\mathfrak{R}_z\subset H^2(Q_\soft)$ and, on the other hand, $\mathfrak{R}_z\subset \ker (A^\soft-z),$ where, as for the operator $A$ in (\ref{eq:quasi_triple_property}), $A^\soft$ is the null extension of $A_0^\soft$ onto $\dom\Gamma_1^\soft,$ whence, by a standard ellipticity argument, the $L^2$ norm of functions in $\mathfrak{R}_z$ is equivalent to their $H^2$ norm (uniformly in $z\in K_\sigma$). The boundedness of the trace operator on $H^2$ is then invoked to complete the proof. 

In pursuing the strategy of \cite{GrandePreuve} in the PDE context, one would also have to establish that $\mathfrak{K}_z^{(\tau)}$ is a bounded operator on $\ran (R_\eff^{(\tau)}(z)-(A_0^\soft -z)^{-1})$ as well, in order to permit the estimate sought. For the problem we study here this range is finite-dimensional, but only so for the concrete choice of $R_\eff^{(\tau)}$ given by Theorem \ref{thm:NRA}, which constitutes precisely the part of the argument we are trying to avoid here.
Therefore, the condition that $\ran\bigl(R_\eff^{(\tau)}(z)-(A_0^\soft -z)^{-1}\bigr)$ is finite-dimensional must appear as an additional assumption. As a useful alternative, it suffices to assume that $\ran (R_\eff^{(\tau)}(z)-(A_0^\soft -z)^{-1})\subset H^2$. Either way, we see that the strategy of \cite{GrandePreuve} is indeed applicable in the PDE setup, albeit it requires an additional assumption\footnote{In fact, a version of precisely this additional assumption was imposed also in \cite{GrandePreuve}, although for a different reason.} on the effective operator $R_\eff^{(\tau)}.$ This is the reason why in the present paper we have chosen a different approach, despite the obvious appeal of a strategy whereby one only needs to establish an asymptotic expansion for the family of generalised resolvents pertaining to the soft component.

Having both above approaches in mind, one notices that the overall strategy in \cite{GrandePreuve} and the present paper can be seen as rooted in the analysis of the mentioned generalised resolvents. In a nutshell, one can understand this by the following abstract discourse.

\subsection{The nature of our strategy in abstract terms.}

Assuming, for the sake of argument, that $R_\e^{(\tau)}(z)$ has a limit, as $\varepsilon\to0,$ in the uniform operator topology for $z$ in a domain $D\subset\mathbb C$, and, further, 
that the resolvent $({A}_\e^{(\tau)}-z)^{-1}$ also admits such limit, one clearly has
\begin{equation}
P_\soft\bigl(A^{(\tau)}_{\text{eff}}-z\bigr)^{-1}\bigr|_{L^2(Q_\soft)}=R_0^{(\tau)}(z),\quad z\in D\subset \mathbb C,
\label{R_def}
\end{equation}
where $R^{(\tau)}_0$ and $A^{(\tau)}_{\text{eff}}$ are the limits introduced above. The powerful idea of simplifying the required analysis by passing to the resolvent ``sandwiched'' by orthogonal projections onto a carefully chosen subspace stems from the pioneering work of Lax and Phillips \cite{LP}, where the resulting sandwiched operator is shown to be the resolvent of a dissipative operator. This idea was later successfully extended to the case of generalised resolvents in \cite{DM}, as well as in \cite{AdamyanPavlov} with the scattering theory in mind.

The function  $R^{(\tau)}_0$ defined by (\ref{R_def}) is a generalised resolvent, whereas $A^{(\tau)}_{\text{eff}}$ is its out-of-space self-adjoint extension (or \emph{Strauss dilation} \cite{Strauss}). By a theorem of Neumark \cite{Naimark1940} ({\it cf.} \cite{Naimark1943}) this dilation is defined uniquely up to a unitary transformation of a special form, which leaves the subspace $H_\soft:=L^2(Q_\soft)=P_\soft H$ intact, provided that the minimality condition
$$
H=\bigvee_{\Im z\not =0}\bigl(A^{(\tau)}_{\text{eff}}-z\bigr)^{-1}H_\soft.
$$
holds. This can be reformulated along the following lines: one has minimality, provided that there are no eigenmodes in the effective media modelled by the operator $A^{(\tau)}_\varepsilon,$ and therefore in the medium modelled by the operator $A^{(\tau)}_{\text{eff}}$ as well, such that they ``never enter'' the soft component of the medium. A quick glance at the setup of critical-contrast composites helps one immediately convince oneself that this must be true. 
It then follows that the effective medium is completely determined, up to a unitary transformation, by $R_0^{(\tau)}(z).$

We must admit that the Neumark-Strauss general theory is not directly applicable in our setting. Part of the reason for this is that $R_\e^{(\tau)}(z)$ in general does not converge (we have shown it to admit an asymptotic expansion instead). Even in Model II, where one can obtain a limit proper, one still needs to prove that the resolvents $({A}_\e^{(\tau)}-z)^{-1}$ converge as well. Therefore, in our analysis we only use the general theory presented above as a guide. In fact, we manage to compute the required asymptotics of the resolvents independently, thus eliminating the non-uniqueness due to the mentioned unitary transformation.

\section{Effective model for the homogenisation limit}
\label{section:4}
\subsection{Self-adjointness of the asymptotics}\label{section:4.1}
This part of the paper draws upon our earlier results, see \cite{GrandePreuve} and references therein.

\begin{definition}\label{defn:gen_dilation}
Consider the Hilbert space $\mathfrak H=H_\soft\oplus \widehat H$, where $\widehat H $ is an auxiliary Hilbert space. Let $\widehat\Pi$ be a bounded and boundedly invertible operator\footnote{Note that by construction $\breve{\mathcal H}$ is one-dimensional (for more general models of scalar PDEs, finite-dimensional). The space $\widehat H$ is therefore also one-dimensional (finite-dimensional in general). This is, however, not necessarily so in the vector PDE setup, which will be treated elsewhere.} from $\breve{\mathcal H}$ to $\widehat H.$ We point out that the operator $\widehat{\Pi}$ may depend on $\tau,$ which we conceal in the notation, in line with the convention we adopt for $\Pi_{\stiff(\soft)}.$ It is assumed throughout that this dependence of $\widehat{\Pi}$ on $\tau$ is continuous. For each $\tau\in Q',$ we set
$$
\dom {\mathcal A}^{(\tau)}:=\Bigl\{(u,\widehat u)^\top\in H_\soft\oplus \widehat H: u\in \dom {A}_{\max}, \widehat u=\widehat \Pi\breve{\Gamma}_0^\soft u\Bigr\},
$$
where $A_{\max}$ is a closed operator in $H_\soft$ and $\breve{\Gamma}_0^\soft: \dom{A}_{\max}\to\breve{\mathcal H}$ is a surjective mapping 
defined in Theorem \ref{thm:NRA_gen_truncated}.
Clearly, $\dom{\mathcal A}^{(\tau)}$ thus defined is dense in $\mathfrak H$. We introduce a linear operator ${\mathcal A}^{(\tau)}$ on this domain by setting
\begin{equation}
\label{eq:strauss_op}
 {\mathcal A}^{(\tau)}\left(\begin{matrix}u\\[0.4em] \widehat u\end{matrix}\right):=\left(\begin{matrix}{A}_{\max} u\\[0.4em] -(\widehat\Pi^*)^{-1} \breve{\Gamma}_1^\soft u + {\mathcal B}^{(\tau)}\widehat u\end{matrix}\right),
\end{equation}
where ${\mathcal B}^{(\tau)}$ is assumed to be a bounded operator in $\widehat H$ and $\breve{\Gamma}_1^\soft: \dom {A}_{\max}\to\breve{\mathcal H}$ is a surjective mapping 
defined in Theorem \ref{thm:NRA_gen_truncated}. Recall that $(\breve{\mathcal H}, \breve{\Gamma}_0^\soft,\breve{\Gamma}_1^\soft)$ is a classical boundary triple for $A_{\max},$ see Theorem \ref{thm:NRA_gen_truncated}.
\end{definition}

\begin{remark}
The formula (\ref{eq:strauss_op}) is the PDE version of a formula in Definition 7.3 of \cite{GrandePreuve}.
\end{remark}

We have the following statement ({\it cf.} \cite{GrandePreuve}) for each $\tau\in Q'.$
\begin{lemma}\label{lemma:self-adjointness}
 The operator ${\mathcal A}^{(\tau)}$ is symmetric 
 if and only if ${\mathcal B}^{(\tau)}$ is self-adjoint in $\widehat H$.
\end{lemma}

\begin{proof}
On the one hand, for all $(u,\widehat u)^\top\in \dom{\mathcal A}^{(\tau)}$ and for $(v,\widehat v)^\top\in \dom{\mathcal A}^{(\tau)}$ one has
\begin{multline}
 \left \langle{\mathcal A}^{(\tau)}\binom{u}{\widehat u},\binom{v}{\widehat v}\right\rangle =\langle {A}_{\max} u,v\rangle -
  \bigl\langle (\widehat{\Pi}^*)^{-1} \breve{\Gamma}_1^\soft u,\widehat v\bigr\rangle + \bigl\langle {\mathcal B}^{(\tau)}\widehat\Pi \breve{\Gamma}_0^\soft u,\widehat v \bigr\rangle\\
  =\langle  u,{A}_{\max} v\rangle +\bigl\langle \breve{\Gamma}_1^\soft u, \breve{\Gamma}_0^\soft v \bigr\rangle
  -\bigl\langle \breve{\Gamma}_0^\soft u, \breve{\Gamma}_1^\soft v \bigr\rangle-
 \bigl\langle (\widehat{\Pi}^*)^{-1} \breve{\Gamma}_1^\soft u,\widehat\Pi \breve{\Gamma}_0^\soft v\bigr\rangle + \bigl\langle {\mathcal B}^{(\tau)}\widehat\Pi \breve{\Gamma}_0^\soft u,\widehat\Pi \breve{\Gamma}_0^\soft v\bigr\rangle\\[0.4em]
  =\bigl\langle  u,{A}_{\max} v\bigr\rangle -\bigl\langle \breve{\Gamma}_0^\soft u, \breve{\Gamma}_1^\soft v\bigr\rangle+ \bigl\langle \widehat\Pi \breve{\Gamma}_0^\soft u, \bigl({\mathcal B}^{(\tau)}\bigr)^*\widehat{\Pi} \breve{\Gamma}_0^\soft v\bigr\rangle,\label{one_star}
\end{multline}
where the second Green's identity has been used.

On the other hand,
\begin{equation}
  \left\langle \binom{u}{\widehat u}, {\mathcal A}^{(\tau)}\binom{v}{\widehat v}\right\rangle
  =\langle  u,{A}_{\max} v\rangle -
  \langle \breve{\Gamma}_0^\soft u, \breve{\Gamma}_1^\soft v\rangle+
 \bigl \langle  \widehat\Pi \breve{\Gamma}_0^\soft u,{\mathcal B}^{(\tau)}\widehat \Pi \breve{\Gamma}_0^\soft v\bigr\rangle,
 \label{two_stars}
\end{equation}
and the claim follows by comparing the right-hand sides of (\ref{one_star}) and (\ref{two_stars}).
\end{proof}

In fact, ${\mathcal A}^{(\tau)}$ is self-adjoint if and only if ${\mathcal B}^{(\tau)}$ is self-adjoint. This follows from the next theorem via an explicit construction of the resolvent $\bigl({\mathcal A}^{(\tau)}-z\bigr)^{-1}.$

\begin{theorem}\label{thm:dilation_resolvent}
Fix $\tau\in Q',$ and assume that ${\mathcal B}^{(\tau)}=\bigl({\mathcal B}^{(\tau)}\bigr)^*$. Then ${\mathcal A}^{(\tau)}$ is self-adjoint, and its resolvent $\bigl({\mathcal A}^{(\tau)}-z\big)^{-1}$ is defined at all $z\in \mathbb C_\pm$ by the following expression, relative to the space decomposition $\mathfrak H=H_\soft \oplus \widehat H$ (cf. \eqref{eq:NRAbreve}):
\begin{equation}\label{eq:NRA1cp}
\bigl({\mathcal A}^{(\tau)}-z\bigr)^{-1}=
\begin{pmatrix}
  R(z) & \Bigl({\widehat{\mathfrak{K}}_{\overline {z}}}\bigl[R(\bar z)-({A}_0^\soft-\bar z)^{-1}\bigr]\Bigr)^* \widehat{\Pi}^* \\[0.8em]
   \widehat\Pi \widehat{\mathfrak{K}}_z\bigl[R(z)-(A_{0}^{\soft}-z)^{-1}\bigr] & \widehat \Pi\widehat{\mathfrak{K}}_z\Bigl(\widehat{\mathfrak{K}}_{\overline{z}}\bigl[R(\bar z)-({A}_0^\soft-\bar z)^{-1}\bigr]\Bigr)^* \widehat{\Pi}^*
\end{pmatrix}.
\end{equation}
Here $R(z)$ is a generalised resolvent in $H_\soft$ defined as the solution operator of the BVP
\begin{align}
&A_{\max} u-z u=f,\quad u\in \dom{A_{\max}}, \quad f\in H_\soft,\nonumber\\[0.4em]
&\breve{\Gamma}_1^\soft u=\widehat\Pi^*\bigl({\mathcal B}^{(\tau)}-z\bigr)\widehat\Pi \breve{\Gamma}_0^\soft u,\label{R_bound_cond}
\end{align}
and the bounded rank-one operator $\widehat{\mathfrak K}_z$ is defined by $\widehat{\mathfrak K}_z:=\breve{\Gamma}_0^\soft |_{{\mathfrak N}_z}$, where, as before ({\it cf.} Theorem \ref{thm:main}), $\mathfrak N_z=\ran\bigl(S_z^\soft\P\bigr).$
\end{theorem}

\begin{proof}
We start by considering the problem
$$
\bigl({\mathcal A}^{(\tau)}-z\bigr) \left(\begin{matrix}u\\[0.35em]\widehat u\end{matrix}\right)=\left(\begin{matrix}f\\[0.35em]0\end{matrix}\right),\qquad \left(\begin{matrix}u\\[0.35em]\widehat u\end{matrix}\right)\in\dom{\mathcal A}^{(\tau)},
$$
which is rewritten as
\begin{align}
{A}_{\max} u-zu &=f,\nonumber\\[0.1em]
-(\widehat\Pi^*)^{-1}\breve{\Gamma}_1^\soft u + {\mathcal B}^{(\tau)}\widehat u-z \widehat u&=0.\label{second_eq}
\end{align}
Since $(u,\widehat u)^\top\in \dom{\mathcal A}^{(\tau)},$ it follows that $\widehat u=\widehat\Pi \breve{\Gamma}_0^\soft u.$ 
Therefore, the condition (\ref{second_eq}) admits the form (\ref{R_bound_cond}), providing the operator $R(z)$ is well defined,
and one has
$u=R(z)f$. We next prove that the operator $R(z)$ is a generalised resolvent, and is hence bounded, for $z\in{\mathbb C}\setminus{\mathbb R}.$ Indeed, this immediately follows from Proposition \ref{prop:operator}. For its applicability one checks that by Corollary \ref{MHerglotz} the operator $\breve{M}_\soft(z)$ is the $M$-function pertaining to the triple $(\breve{\mathcal H}, \breve{\Pi}_\soft, \breve{\Lambda}^\soft).$ Then, proceeding as in the footnote on 
p.\,\pageref{foot_page}, we obtain $\Im \breve{M}^\soft(z) = (\Im z)(\breve{S}^\soft_{\bar z})^*\breve{S}^\soft_{\bar z},$ where 
$\breve{S}_z^\soft$ and $\breve{M}^\soft(z)$ are as in Theorem \ref{thm:NRA_gen_truncated}, and therefore 
\begin{equation}
\begin{aligned}
\Bigl\vert\bigl\langle\Im\bigl(\breve{M}^\soft(z)-\widehat{\Pi}^*({\mathcal B}^{(\tau)}-z)\widehat{\Pi}\bigr)\phi, \phi\bigr\rangle_{\breve{\mathcal H}}\bigr\vert
&=\vert\Im z\vert\bigl\langle(\breve{S}^\soft_{\bar z})^*\breve{S}^\soft_{\bar z}\phi, \phi\bigr\rangle_{\breve{\mathcal H}}+\vert\Im z\vert\bigl\langle \widehat{\Pi}^{*}\widehat{\Pi}\phi, \phi\bigr\rangle_{\breve{\mathcal H}}\\[0.35em]
&\ge\vert\Im z\vert\bigl\Vert\widehat{\Pi}\phi\bigr\Vert_{\widehat{H}}\\[0.5em]
&\ge\vert\Im z\vert\bigl\Vert\widehat{\Pi}^{-1}\bigr\Vert_{\widehat{H}\to\breve{\mathcal H}}^{-1}
\Vert\phi\Vert_{\breve{\mathcal H}}\qquad\quad\forall\phi\in\breve{\mathcal H},\ \tau\in Q,\ z\in{\mathbb C}\setminus{\mathbb R}.
\end{aligned}
\label{add_est}
\end{equation}
It follows that the operator $\breve{M}_\soft(z)-\widehat{\Pi}^*({\mathcal B}^{(\tau)}-z)\widehat{\Pi}$ is boundedly invertible and the assumptions of Proposition \ref{prop:operator} are satisfied. Summarising the above, we obtain
$$
\widehat u=\widehat\Pi \breve{\Gamma}_0^\soft R(z) f=\widehat\Pi \widehat{\mathfrak K}_z\bigl[R(z)-(A_{0}^{\soft}-z)^{-1}\bigr]f,
$$
which completes the proof for the first column of the matrix representation \eqref{eq:NRA1cp}.

We proceed with establishing the second column.To this end, we rewrite the problem 
\begin{equation}\label{eq:solving_second_column}
\bigl({\mathcal A}^{(\tau)} - z\bigr) \left(\begin{matrix}u\\[0.35em]\widehat u\end{matrix}\right)=\left(\begin{matrix}0\\[0.35em]\widehat f\end{matrix}\right),\qquad \left(\begin{matrix}u\\[0.35em]\widehat u\end{matrix}\right)\in\dom{\mathcal A}^{(\tau)},
\end{equation}
as
\begin{align}
{A}_{\max} u-zu &=0,\nonumber\\[0.3em]
-(\widehat\Pi^*)^{-1}\breve{\Gamma}_1^\soft u + {\mathcal B}^{(\tau)}\widehat u-z \widehat u&=\widehat f.\label{second_equation}
\end{align}
Notice that (\ref{second_equation})
admits the form
$$
\breve{\Gamma}_1^\soft u=\widehat \Pi^*\bigl({\mathcal B}^{(\tau)} -z\bigr) \widehat \Pi \breve{\Gamma}_0^\soft u-\widehat \Pi^* \widehat f.
$$
Furthermore, pick a function $v_f\in\dom A_0^\soft$ that satisfies $\breve{\Gamma}_1^\soft v_f=\widehat \Pi^* \widehat f$. Such a choice is possible due to the surjectivity property of the boundary triple. We look for a solution to \eqref{eq:solving_second_column} such that its first component $u$ admits the form $u=v-v_f$. 
Using the fact that by construction $v_f\in \dom A_{0}^\soft,$ we obtain the formula $v=R(z)(A_{0}^\soft-z)v_f,$ and therefore
$
u=R(z)(A_{0}^\soft-z)v_f-v_f.
$
Letting $u_f:=(A_{0}^\soft-z)v_f$  
 and using Proposition \ref{prop:operator},
 we thus have
\begin{align*}
u=\bigl[R(z)-(A_{0}^\soft-z)^{-1}\bigr]u_f&=-\breve S^\soft_z\bigl(\breve{M}^\soft(z)-\widehat\Pi^*\bigl({\mathcal B}^{(\tau)}-z\bigr)\widehat\Pi
\bigr)^{-1}\breve{\Gamma}_1^\soft v_f\\[0.4em] 
&=-\breve{S}^\soft_z\bigl(\breve{M}^\soft(z)-\widehat\Pi^*\bigl({\mathcal B}^{(\tau)}-z\bigr)\widehat\Pi
\bigr)^{-1}\widehat\Pi^* \widehat f.
\end{align*}
Proceeding as in the proof of Theorem \ref{thm:main}, we rewrite the latter expression as follows:
$$
u=\Bigl(\widehat{\mathfrak{K}}_{\overline{z}}\bigl[R(\bar z)-({A}_0^\soft-\bar z)^{-1}\bigr]\Bigr)^* \widehat\Pi^* \widehat f,
$$
and thus complete the proof of the representation \eqref{eq:NRA1cp}. The fact that $\mathcal A$ is self-adjoint in $\mathfrak H$ now follows from Lemma \ref{lemma:self-adjointness}.
\end{proof}

The following corollary follows by a comparison of the assertions of Theorem \ref{thm:dilation_resolvent} and Theorem \ref{thm:main_fin}.

\begin{corollary}\label{cor:NRA} 
Consider the self-adjoint operator $\mathcal{A}^{(\tau)}_\hom$, introduced by Definition \ref{defn:gen_dilation} with $\widehat{H}=\ran\breve{\Pi}_\stiff,$ $\widehat\Pi
=\breve{\Pi}_\stiff,$ and the operator ${\mathcal B}^{(\tau)}$
chosen so that
$
-\breve \Lambda^\stiff=(\breve{\Pi}_\stiff)^*{\mathcal B}^{(\tau)}\breve{\Pi}_\stiff.
$
The null extension of its resolvent $(\mathcal{A}^{(\tau)}_\hom-z)^{-1}$ from the space $H_\soft\oplus\widehat{H}\equiv H_\soft\oplus\ran \breve{\Pi}_\stiff$ to the space $H$ coincides with the asymptotics $\mathcal R^{(\tau)}_\hom(z)$ of $({A}_\e^{(\tau)}-z)^{-1}$ provided by Theorem \ref{thm:main_fin}. 
\end{corollary}

\subsection{General homogenisation result}
Here we formulate the ultimate version of our homogenisation result, having collected all the necessary notation and definitions in one place for easy reference.

Compared to the general setup of Section \ref{section:4.1}, here we identify, by means of an obvious unitary transform, the auxiliary Hilbert space $\widehat H$ with $\mathbb C^1$. Then the result of Corollary \ref{cor:NRA} together with that of Theorem \ref{thm:main_fin} leads to the following explicit description of the homogenised resolvent for $({A}_\e^{(\tau)}-z)^{-1}$.

Let $\mu_\tau$ be the least (by absolute value) Steklov eigenvalue of the restriction of the BVP to the stiff component $Q_\stiff$ of the medium, see \eqref{eq:Steklov_stiff}. The multiplicity of $\mu_\tau$ is one, and the associated normalised Steklov eigenvector is denoted by $\psi_\tau\in H^1(\Gamma)$. The corresponding orthogonal projection is $\P:=\langle \cdot, \psi_\tau\rangle_{{\mathcal H}}\psi_\tau$ and its mutually orthogonal one is $\Port=I-\P$.

Recall that $\Lambda^\stiff$ is the DN map pertaining to the stiff component as defined by \eqref{eq:DN_stiff}--\eqref{conormal_stiff_later}. The 
$\tau$-harmonic lift operators $\Pi_{\stiff(\soft)}$ of the stiff (respectively, soft) component are defined in (\ref{eq:harmonic_lift1}).
The ``truncated" DN map and $\tau$-harmonic lift operators are $\breve{\Lambda}^\stiff=\P\Lambda^\stiff\vert_{\breve{\mathcal H}}$ and $\breve \Pi_{\stiff(\soft)}=\Pi_{\stiff(\soft)}\vert_{\breve{\mathcal H}}.$ Furthermore, we note, see (\ref{diagonal_lambda}), that $\breve{\Lambda}^\stiff$ is
 the operator of multiplication by $\e^{-2}\mu_\tau.$

The operator $A_{\max},$ see Theorem \ref{thm:NRA_gen_truncated}, is defined on the domain 
\[
\dom A_{\max}=\bigl(H^2(Q_\soft)\cap H^1_0(Q_\soft)\bigr)\dotplus{\mathcal L}\{\Pi_\soft \psi_\tau\}
\] 
as the null extension, by a one-dimensional subspace, of the Dirichlet magnetic Laplacian $A_{0}^\soft$ given by the differential expression $-(\nabla+{\rm i}\tau)^2$ in $L^2(Q_\soft).$  Let $H_\hom=L^2(Q_\soft)\oplus \mathbb C^1,$ and for each $\tau\in Q',$ consider the following self-adjoint operator $\mathcal A_{\rm hom}^{(\tau)}$ on the space $H_\hom.$ The domain $\dom \mathcal A_{\hom}^{(\tau)}$ is defined as
\begin{equation}\label{eq:domain_fin}
\dom \mathcal A_{\hom}^{(\tau)}=\bigl\{(u,\beta)^\top\in H_\hom:\ u\in \dom A_{\max},\, \beta = j_Q \Pi_\stiff(u\big|_\Gamma)\bigr\},
\end{equation}
where $u|_\Gamma\in \P{\mathcal H}$ is the trace of the function $u$ and $j_Q$ is the unitary operator from the space ${\mathcal L}\{\Pi_\stiff \psi_\tau\}$ (supplied with the standard norm of $L^2(Q_\stiff)$) to $\mathbb C^1$. On $\dom \mathcal A_{\hom}^{(\tau)}$ the action of the operator is set by
\begin{equation}\label{eq:operator_fin}
\mathcal A_{\hom}^{(\tau)}\binom{u}{\beta}=
\left(\begin{array}{c}\biggl(\dfrac{1}{\rm i}\nabla+\tau\biggr)^2 u\\[1.0em]
-(\Pi_{\tau}^*)^{-1}\bigl(\P\dntau u\big|_\Gamma+
\breve{\Lambda}^\stiff\Pi_{\tau}^{-1}\beta\bigr)
\end{array}\right),\qquad \Pi_{\tau}:=j_Q \Pi_\stiff\bigr\vert_{\P{\mathcal H}}.
\end{equation}
Here $\Pi_{\tau}$  is a boundedly invertible operator from $\P {\mathcal H}$ to $\mathbb C^1,$ and $\dntau u = -(\partial u/\partial n+{\rm i}\tau\cdot n u)|_\Gamma$ is the co-normal derivative on $Q_\soft,$ 
see \eqref{conormal_soft_later}. We remark that, due to the identity $\breve{\Lambda}^\stiff = \breve{\Gamma}_1^\stiff \breve {\Pi}_\stiff,$ the second component of the action of the operator ${\mathcal A}^{(\tau)}_{\rm hom}$ 
can be written as
$$
-(\Pi_{\tau}^*)^{-1}\bigl(\P\dntau u\big|_\Gamma+\breve{\Lambda}^\stiff\Pi_{\tau}^{-1}\beta\bigr)=
-(\Pi_{\tau}^*)^{-1}\P\bigl(\dntau u\big|_\Gamma+\varepsilon^{-2}\partial^{(\tau)}_{n} (j_Q^* \beta)\big|_\Gamma\bigr), 
$$
where, in the second term in brackets on the right-hand side, $\dntau$ is the co-normal derivative on $Q_\stiff$, see \eqref{conormal_stiff_later}.

The main result of the present work, which follows from Corollary \ref{cor:NRA}, is as follows.
\begin{theorem}\label{thm:general_homo_result}
The resolvent $({A}_\e^{(\tau)}-z)^{-1}$ admits the following estimate in the uniform operator norm topology:
\begin{equation}
\bigl({A}_\e^{(\tau)}-z\bigr)^{-1}-\Theta^*\bigl(\mathcal A_{\hom}^{(\tau)}-z\bigr)^{-1}\Theta=O(\e^2),
\label{main_est}
\end{equation}
where $\Theta$ is a partial isometry from $H$ onto $H_\hom:$ on the subspace $H_\soft$ it coincides with the identity, and each function from $L^2(Q_\stiff)=H\ominus H_\soft,$ represented as an orthogonal sum $c_\tau\Vert \Pi_\stiff \psi_\tau\Vert^{-1}\Pi_\stiff \psi_\tau\oplus \xi_\tau,$ $c_\tau\in{\mathbb C}^1,$ is mapped to $c_\tau$ unitarily.

The estimate (\ref{main_est}) is uniform in $\tau\in Q'$ and $z\in K_\sigma$.
\end{theorem}

\subsection{Models I and II}
\label{section:IandII}
The definition of the homogenised operator in the previous section essentially relies on the knowledge of the following objects: (i) the  normalised  first Steklov eigenvector $\psi_\tau$; (ii) the result of the calculation $\Pi_\tau \psi_\tau$; and (iii) the truncated DN map of the stiff component $\breve{\Lambda}^\stiff$ or, equivalently, $\e^{-2}\dntau \Psi_{\tau}$, where $\Psi_{\tau}$ is the $\tau$-harmonic lift to $Q_\stiff$ of the function $\psi_\tau$, or, equivalently yet again, $\e^{-2}\mu_\tau$ where $\mu_\tau$ is the least Steklov eigenvalue of the stiff component, see \eqref{eq:Steklov_stiff}.

Clearly, $\psi_\tau=\Psi_\tau|_\Gamma$, so one is only looking to determine $\mu_\tau$ and $\Psi_\tau$. The complexity of determining these objects in Models I and II is incomparable. Indeed, in Model II one immediately has $\mu_\tau$=0 and $\Psi_\tau=|\Gamma|^{-1/2}\exp(-i \tau\cdot x),$ $x\in Q_\stiff,$ for all $\tau\in[-\pi,\pi)^d.$  Thus, $\breve{\Lambda}^\stiff\equiv 0$, cancelling out the second term in the second component of the definition \eqref{eq:operator_fin}. Moreover, the rank-one operator $\Pi_\tau$ sends $\exp(-{\rm i}\tau \cdot x)|_{x\in\Gamma}$ to $|Q_\stiff|^{1/2}\in\mathbb C^1$, so that its action is given by
$$
\Pi_\tau=\frac {|Q_\stiff|^{1/2}}{|\Gamma|^{1/2}}\langle\cdot,\psi_\tau\rangle_{{\mathcal H}}e_1,
$$
where $e_1$ is the unit vector in $\mathbb C^1$. 
Thus,
$$
(\Pi_\tau^*)^{-1}=\frac {|\Gamma|^{1/2}}{|Q_\stiff|^{1/2}}\langle\cdot,\psi_\tau\rangle_{{\mathcal H}}e_1,
$$
and every object in \eqref{eq:domain_fin} and \eqref{eq:operator_fin} has been computed explicitly.

In contrast, for Model I one cannot explicitly compute either $\mu_\tau$ or $\Psi_\tau$ (in fact, even in the ODE setup of \cite{GrandePreuve} this computation yields highly non-trivial functions of $\tau$). The best one can hope for is getting hold of \emph{asymptotic expansions} for both quantities. Indeed, \cite{Friedlander} provides a result of this kind, proving that $\mu_\tau$ is quadratic in $\tau$ to the leading order and that $\Psi_\tau$ admits an expansion 
\[
\Psi_\tau=|\Gamma|^{-1/2}\bigl(1+\tau\cdot \Psi_\tau^{(1)}+O(|\tau|^2)\bigr).
\]
Therefore, one would ideally wish to have $\mu_\tau$, $\psi_\tau$ and $\Psi_\tau$ in \eqref{eq:domain_fin} and \eqref{eq:operator_fin} replaced by either the leading-order terms of their asymptotic expansions or finite sums of these. It is clear that in this situation one cannot hope for the error estimate of the same order $O(\e^2)$ as in the case of Model II, but the first question is whether the result is \emph{stable} under this asymptotic procedure. The answer to the posed question is not obvious and will be addressed in the next section.

\subsection{Stability considerations for Model I}\label{section:stability}
In the present paper we shall only treat the zero-order stability problem: we seek to replace $\psi_\tau$ and $\Psi_\tau$ in \eqref{eq:domain_fin} and \eqref{eq:operator_fin} by their zero-order terms in $\tau$, namely, $\psi_0=|\Gamma|^{-1/2} \mathbbm  1\vert_\Gamma$
and $\Psi_0=|\Gamma|^{-1/2} \mathbbm{1}|_{Q_\stiff},$ {\it i.e.} the $0$-harmonic lift of $\psi_0.$ Only the setup of Model I will be considered. For the analysis in this setting it suffices to use the results of \cite{Friedlander}, whereas in the more general cases (non-constant symbol $a^2(\cdot/\e)$ of the operator ${A}_\e^{(\tau)}$ on the stiff component, or higher-order stability problems), one would require an advanced asymptotic argument, to which the second part \cite{ChEK_future} of this paper will be devoted.

The main idea behind our construction below is 
 to prove that the resolvent $\bigl(A_{\Port,\P}-z\bigr)^{-1}$ 
 provided by Theorem \ref{thm:NRA} is asymptotically close to $(A_{\beta_0,\beta_1}-z)^{-1},$ where the definitions of $\beta_0$ and $\beta_1$ are based on the projections $\Po$ and $\Porto$ rather than $\P$ and $\Port$. An obvious guess would appear to be $\beta_0=\Porto$ and $\beta_1=\Po$, however we instead 
 arrive at the choice $\beta_0=\Porto +\Lambda_\Delta $ and $\beta_1=\Po$, where $\Lambda_\Delta $ is a Hermitian operator in $\Po \mathcal H$. We shall not prove this asymptotics for \emph{all} values of quasimomentum $\tau$ at once but first for $|\tau|\leq C \e^{2/3},$ and then the result for all $\tau\in Q'$ will follow from Theorem \ref{thm:main_fin} by using the triangle inequality, with an error estimate of order $O(\e^{2/3}),$ see Theorem \ref{thm:stability_NRA_fin}.

\begin{remark}
Asymptotics leading to better estimates of the remainder (all the way up to $O(\e^2)$) are available by resorting to higher order stability problems, see \cite{ChEK_future} for details.
\end{remark}

The starting point of our argument is the following statement.

\begin{lemma}
\label{lemma:PLP} The asymptotic formula
\begin{equation}
\label{eq:lambda_delta_estimate}
\bigl(\P-\Po\bigr)\Lambda^\stiff \Po = \Lambda_\Delta + O\bigl(|\tau|^3/\e^2\bigr),
\end{equation}
holds, where $\Lambda_\Delta$ is a Hermitian operator on $\mathcal H,$ given explicitly\footnote{We recall that the space $\Po\mathcal H$ is one-dimensional, so $\Lambda_\Delta$ is the multiplication by a real number when restricted to $\Po\mathcal H.$} by
\begin{equation}
\label{eq:lambda_delta}
\Lambda_\Delta:= \e^{-2}\bigl\langle \cdot, {\rm i}\Po (\tau\cdot n)(\tau\cdot\psi^{(1)})\bigr\rangle_{{\mathcal H}}\psi_0.
\end{equation}
Here $n$ is the external unit normal to $Q_\stiff$, and  $\tau\cdot\psi^{(1)},$ where $\psi^{(1)}\in L^2(\Gamma;\mathbb C^d),$ is the linear in $\tau$ term in the asymptotic expansion of $\psi_\tau$ (see e.g. \cite{Friedlander}):
\begin{equation}
\psi_\tau = \psi_0 + \tau\cdot \psi^{(1)}+O\bigl(\vert\tau\vert^2\bigr).
\label{taylor_exp}
\end{equation}
\end{lemma}

\begin{proof}
Consider the problem
$$
-(\nabla +{\rm i}\tau)^2 u =0, \quad u|_\Gamma = \psi_0, \quad u\in L^2(Q_\stiff).
$$
Writing $u=\Psi_0+w,$ one has 
$$
-(\nabla +{\rm i}\tau)^2w=(\nabla +{\rm i}\tau)^2 \Psi_0=-|\tau|^2 \Psi_0, \quad w|_\Gamma=0,
$$
hence $w=-|\tau|^2 \e^{-2} (A_0^\stiff)^{-1}\Psi_0,$ and 
$$
\widetilde \Lambda^\stiff \psi_0= \dntau u\bigr\vert_\Gamma=- (\nabla u + {\rm i }\tau u)\cdot n\bigr|_\Gamma =-(\nabla \Psi_0 +{\rm  i}\tau \Psi_0)\cdot n\bigr|_\Gamma -|\tau|^2 \Gamma_1^\stiff(A_0^\stiff)^{-1}\Psi_0.
$$
Furthermore, using $\Pi_\stiff^*=\Gamma_1^\stiff (A_0^\stiff)^{-1}$ and
$
-(\nabla \Psi_0 + {\rm i}\tau \Psi_0)\cdot n|_\Gamma=-{\rm i}(\tau\cdot n)\psi_0,
$
yields
\begin{equation}
\widetilde \Lambda^\stiff \psi_0= -{\rm i}(\tau\cdot n)\psi_0 -|\tau|^2 \Pi_\stiff^* \Psi_0= -{\rm i}(\tau\cdot n)\psi_0 -|\tau|^2  \Pi_\stiff^* \Pi_{\stiff,0}\psi_0,
\label{lambda_action}
\end{equation}
where $\Pi_{\stiff,0}$ is the $0$-harmonic lift,  
which maps, in particular, $\psi_0$ to $\Psi_0$.

Next, we compute the leading-order term with respect to $\tau$ for $\bigl(\P-\Po\bigr) \Lambda^\stiff \Po,$ see the left-hand side of (\ref{eq:lambda_delta_estimate}). Note that since the second term in (\ref{lambda_action}) is of order $O(|\tau|^2),$ the corresponding contribution to this leading-order term
is given by
$$
\bigl(\P-\Po\bigr)\e^{-2}|\tau|^2  \Pi_\stiff^* \Pi_{\stiff,0}\psi_0= O(|\tau|^3/\e^2).
$$
Here we use the fact that $\P-\Po=O(|\tau|),$ which, in turn, follows from 
$\P=\langle\cdot, \psi_\tau\rangle\psi_\tau$ and $\psi_\tau=\psi_0+O(|\tau|),$ see (\ref{taylor_exp}).

For the contribution of the first term of (\ref{lambda_action}),
notice first that
\begin{equation}
\Po\bigl(-{\rm i}(\tau\cdot n)\psi_0\bigr)= -{\rm i}\bigl\langle(\tau\cdot n)\psi_0,\psi_0\bigr\rangle_{{\mathcal H}}\psi_0=0,
\label{P0_id}
\end{equation}
since $\psi_0$ is constant and the integral of the normal $n$ over $\Gamma$ vanishes. 
Second, using (\ref{taylor_exp}) yields
\begin{equation}
\P\bigl(-{\rm i}(\tau\cdot n)\psi_0\bigr)=\bigl\langle -{\rm i}(\tau\cdot n)\psi_0,\psi_0\bigr\rangle_{{\mathcal H}}\psi_\tau +\bigl\langle -{\rm i}(\tau\cdot n)\psi_0,\tau\cdot\psi^{(1)}\bigr\rangle_{{\mathcal H}}\psi_\tau+O(|\tau|^3).
\label{Ptau}
\end{equation}
The first term in (\ref{Ptau}) is zero, and 
$$
\bigl\langle -{\rm i}(\tau\cdot n)\psi_0,\tau\cdot\psi^{(1)}\bigr\rangle_{{\mathcal H}}\psi_\tau=\bigl\langle -{\rm i}(\tau\cdot n)\psi_0,\tau\cdot\psi^{(1)}\bigr\rangle_{{\mathcal H}}\psi_0 + O(|\tau|^3),
$$
which immediately implies \eqref{eq:lambda_delta_estimate}.

It remains to verify that $\Lambda_\Delta$ is Hermitian. This follows from the fact that for the $0$-harmonic lift $\Psi_\tau^{(1)}$ of $\tau\cdot\psi^{(1)}$ to $Q_\stiff$ one has
({\it cf.} \cite{Friedlander}, \cite{Suslina_dyrki}, \cite{Cher_Serena})
$$
n\cdot\nabla\Psi_\tau^{(1)}\bigr\vert_\Gamma
=-{\rm i}|\Gamma|^{-1/2}(\tau\cdot n),
$$
and hence
$$
\bigl\langle -{\rm i}(\tau\cdot n)\psi_0,\tau\cdot\psi^{(1)}\bigr\rangle_{{\mathcal H}}=\bigl\langle n\cdot \nabla\Psi_\tau^{(1)}, \Psi_\tau^{(1)}\bigr\rangle_{{\mathcal H}}
=\bigl\|\nabla \Psi_\tau^{(1)}\bigr\|^2_{L^2(Q_\stiff)},
$$
by the first Green's identity.
\end{proof}



\begin{remark} 
\label{Remark4_9}
The formulae (\ref{lambda_action}) and (\ref{P0_id}) imply
$$
\Po \Lambda^\stiff \Po=-\Po|\tau|^2\varepsilon^{-2}\Pi_\stiff^*\Pi_{\stiff, 0}\Po=-|\tau|^2\varepsilon^{-2}\breve{\Pi}_{\stiff,0}^*\breve{\Pi}_{\stiff,0}+O(|\tau|^3/\varepsilon^2),
$$
with $\breve{\Pi}_{\stiff,0}:={\Pi}_{\stiff,0}\Po,$ where we use the fact that $\Pi_{\rm stiff}=\Pi_{{\rm stiff}, 0}+O(|\tau|).$ Therefore,  $\Po \Lambda^\stiff \Po$ is a negative Hermitian operator on $\Po\mathcal H$ for sufficiently small $\tau.$ 
\end{remark}

\begin{lemma}
\label{lemma:PLP_addon} 
The following asymptotic formula holds:
$$
\P\Lambda^\stiff \P= \Po \Lambda^\stiff \Po + \Lambda_\Delta + O(|\tau|^3/\e^2).
$$  
\end{lemma}

\begin{proof}
Notice that by Lemma \ref{lemma:PLP} and the property $\P\Lambda^\stiff=\Lambda^\stiff\P$ one has
\begin{align*}
\P\Lambda^\stiff \P - \Po \Lambda^\stiff \Po&=\bigl(\P-\Po\bigr)\Lambda^\stiff \Po + \P\Lambda^\stiff\bigl(\P-\Po\bigr)\\[0.3em]
&=\Lambda_\Delta+O\bigl(|\tau|^3/\e^2\bigr)+\Lambda^\stiff\P\bigl(\P-\Po\bigr)\\[0.3em]
&=\Lambda_\Delta+O\bigl(|\tau|^3/\e^2\bigr)+\mu_\tau\e^{-2}\bigl(\P-\Po\bigr)=\Lambda_\Delta+O\bigl(|\tau|^3/\e^2\bigr)
\end{align*}
where in the last equality we use the fact that $\mu_\tau$ is quadratic in $\tau$ to the leading order (see \cite{Friedlander}). 
\end{proof}

\begin{lemma}
\label{lemma:PLP_mu}
Consider the operator $\Po \Lambda^\stiff \Po + \Lambda_\Delta $ on $\Po \mathcal H$. For sufficiently small $\tau,$ it is a negative Hermitian operator,
and the leading order in $\tau$ of its eigenvalue is a negative-definite quadratic form $\e^{-2}\mu_*\tau\cdot\tau$ on $Q'\times Q'.$ Moreover, one has
$
\mu_\tau-\mu_*\tau\cdot\tau= O(|\tau|^3).
$
\end{lemma}
\begin{proof}
By Lemma \ref{lemma:PLP_addon}, we have
\begin{align*}
\bigl(\Po \Lambda^\stiff \Po + \Lambda_\Delta \bigr) \psi_0&=\P\Lambda^\stiff\P \psi_0 + O\bigl(|\tau|^3/\e^2\bigr)
\\[0.4em]
&=\e^{-2}\mu_\tau \P \psi_0 + O\bigl(|\tau|^3/\e^2\bigr)=\e^{-2}\mu_\tau\psi_0+\e^{-2}\mu_\tau(\P\psi_0-\psi_0)+O\bigl(|\tau|^3/\e^2\bigr),
\end{align*}
which yields the required asymptotics. The negative-definiteness now follows from \cite{Friedlander}, where it is proved that the quadratic in $\tau$ leading-order term of the Taylor series for $\mu_\tau$ is negative definite\footnote{Note that under our definition the DN maps are the negatives of the ones in \cite{Friedlander}. Thus, the positive-definiteness of \cite[Lemma 7]{Friedlander} yields negative-definiteness of $\mu_\tau$.}.
\end{proof}

We have thus constructed an operator in $\Po \mathcal H$ that has the same (to the leading order) eigenvalue as the operator $\P\Lambda^\stiff\P$ in the ``shifted'' subspace $\P \mathcal H$ and is therefore a good approximation of the latter. 

\begin{theorem}
\label{thm:stability_NRA}
For $z\in K_\sigma, $ one has
$$
(A_{\beta_0,\beta_1}-z)^{-1}-(A_{\beta_0',\beta_1'}-z)^{-1} = O\bigl(|\tau|\bigr),
$$
where $\beta_0=\Port,$ $\beta_1=\P$ and $\beta_0'=\Porto+\Lambda_\Delta,$ $\beta_1'=\Po$.
\end{theorem}

\begin{proof}
For both resolvents we use the representation (\ref{eq:Krein_general}) provided by Proposition \ref{prop:operator}. As the first terms are identical, the only difference is between the functions 
$$
Q(z)=-(\overline{\beta_0+\beta_1 M(z)})^{-1}\beta_1 \quad \text{and} \quad Q'(z)=-(\overline{\beta_0'+\beta_1' M(z)})^{-1}\beta_1',
$$
which we compare next. The expression $Q'(z)-Q(z),$ within the choice for $\beta_j,$ $\beta'_j,$ $j=0,1,$ as in the formulation of the theorem, is given by 
\begin{align}
&\bigl(\overline{\Port+\P M(z)}\bigr)^{-1}\P-\bigl(\overline{\Porto+\Po M(z)+\Lambda_\Delta}\bigr)^{-1}\Po\nonumber\\[0.4em]
&=\bigl(\overline{\Port +\P M(z)}\bigr)^{-1}\bigl(\P-\Po\bigr)+\left[\bigl(\overline{\Port +\P M(z)}\bigr)^{-1}-\bigl(\overline{\Porto +\Po M(z)+\Lambda_\Delta}\bigr)^{-1}\right]\Po\nonumber\\[0.4em]
&=\bigl(\overline{\Port +\P M(z)}\bigr)^{-1}\bigl(\P-\Po\bigr)\nonumber\\[0.4em]
&+ \bigl(\overline{\Port +\P M(z)}\bigr)^{-1}\left[\bigl(\P-\Po\bigr) - \bigl(\P-\Po\bigr) M(z) +\Lambda_\Delta\right]\bigl(\overline{\Porto +\Po M(z)+\Lambda_\Delta}\bigr)^{-1}\Po\nonumber\\[0.4em]
&=O(|\tau|)+\bigl(\overline{\Port +\P M(z)}\bigr)^{-1}\left[O(|\tau|) -\bigl(\P-\Po\bigr) M(z) +\Lambda_\Delta\right]\Po\bigl(\overline{\Porto +\Po M(z)+\Lambda_\Delta}\bigr)^{-1}\Po,\label{star_product}
\end{align}
where in the second equality we use the second Hilbert identity, and in the third one the estimate $\P-\Po=O(|\tau|),$ see the proof of Lemma \ref{lemma:PLP}, as well as the fact that by the Schur-Frobenius inversion formula (\ref{eq:SchurF}), utilising the triangular form of $\overline{\Porto +\Po M(z)+\Lambda_\Delta},$ one has
\begin{align}
\bigl(\overline{\Porto +\Po M(z)+\Lambda_\Delta}\bigr)^{-1}\Po&=\begin{pmatrix}
\bigl(\Po M(z) \Po+\Lambda_\Delta\bigr)^{-1}&0\label{inv_exp}\\[0.4em]
0&0
\end{pmatrix}\\[0.9em]
&=\Po\bigl(\overline{\Porto +\Po M(z)+\Lambda_\Delta}\bigr)^{-1}\Po.\nonumber
\end{align}

We estimate the three terms of the product in (\ref{star_product}) one by one. Note first that $\bigl(\overline{\Port +\P M(z)}\bigr)^{-1}$ is bounded uniformly in $\e$ and $\tau,$ which is seen by writing, see (\ref{eq:SchurF}),
\begin{equation}
\label{triangle_invert}
\bigl(\overline{\Port +\P M(z)}\bigr)^{-1}=\begin{pmatrix}
\bigl(\P M(z) \P\bigr)^{-1}&-\bigl(\P M(z) \P\bigr)^{-1} \P M(z)\Port\\[0.4em]
0&\Port
\end{pmatrix},
\end{equation}
and using the fact that $\bigl(\P M(z) \P\bigr)^{-1}$ is uniformly bounded, see the argument following (\ref{eq:SchurF}).

Furthermore, one has 
\begin{equation}
\label{4_12_analysis}
\begin{aligned}
-(\P-\Po) M(z)\Po+\Lambda_\Delta&=-(\P-\Po)\bigl(M^\stiff(z)+M^\soft(z)\bigr)\Po+\Lambda_\Delta\\[0.35em]
&=-(\P-\Po)\bigl(\Lambda^\soft+z\Pi_\soft^*\bigl(1-z(A_0^\soft)^{-1}\bigr)^{-1}\Pi_\soft\bigr)\Po\\[0.35em]
&\ \ \ -(\P-\Po)\bigl(\Lambda^\stiff + z\Pi_\stiff^*\Pi_\stiff+O(\e^2)\bigr)\Po+\Lambda_\Delta\\[0.4em]
&=O(|\tau|^2)+O(|\tau|)+O(|\tau|^3/\e^2)+O(|\tau|),
\end{aligned}
\end{equation}
where we use (\ref{eq:M_representation}), (\ref{eq:Mstiff_asymp}), and the analysis 
is immediately reduced to Lemma \ref{lemma:PLP}. The lemma is applied directly to $(\P-\Po)\Lambda^\stiff\Po,$ which results in the $O(|\tau|^3/\e^2)$ term in (\ref{4_12_analysis}). As for the term $(\P-\Po)\Lambda^\soft\Po,$ we proceed as the proof of Lemma \ref{lemma:PLP}, obtaining in place of (\ref{lambda_action})
\begin{equation}
\Lambda^\soft\psi_0= -{\rm i}(\tau\cdot n)\psi_0 -|\tau|^2  \Pi_\soft^* \Pi_{\soft,0}\psi_0,
\label{kher'}
\end{equation}
where $ \Pi_{\soft,0}$ is the $0$-harmonic lift and $\psi_0$ is the same as above. 
Applying the operator $\P-\Po$ to the expression (\ref{kher'}) yields a 
$O(|\tau|^2)$ estimate in (\ref{4_12_analysis}).


Finally, we estimate the operator in (\ref{inv_exp}). On the one hand, for a constant $\tilde{c}>0,$ which we choose below,
it is bounded uniformly when $|\tau|\leq \tilde{c}\e.$ Indeed, notice first that due to (\ref{eq:Mstiff_asymp}) one has 
\[
\Po M(z) \Po+\Lambda_\Delta=\Po M^\soft(z)\Po+\Po\Lambda^\stiff\Po+\Lambda_\Delta+z\Po\Pi_\stiff^*\Pi_\stiff\Po+O(\varepsilon^2).
\]
Furthermore, 
using the fact that $\Im M^\soft(z) = (\Im z)(S^\soft_{\bar z})^*S^\soft_{\bar z},$ we obtain ({\it cf.} (\ref{add_est}))
\begin{equation*}
\begin{aligned}
\Bigr\vert\bigl\langle\Im\bigl(\Po M^\soft(z)\Po&+\Po\Lambda^\stiff\Po+\Lambda_\Delta+z\Po\Pi_\stiff^*\Pi_\stiff\Po\bigr)\phi, \phi\bigr\rangle_{\mathcal H}\Bigr\vert\\[0.4em]
&=\vert\Im z\vert\bigl\langle(S^\soft_{\bar z})^*S^\soft_{\bar z}\Po\phi, \Po\phi\bigr\rangle_{\mathcal H}+\vert\Im z\vert\bigl\langle\Pi_\stiff^*\Pi_\stiff\Po\phi, \Po\phi\bigr\rangle_{\mathcal H}\\[0.5em]
&\ge\vert\Im z\vert\bigl\Vert\Pi_\stiff\Po\phi\bigr\Vert_{H}\\[0.5em]
&\ge \vert\Im z\vert\Bigl\Vert\bigl(\Pi_\stiff\bigr\vert_{\Po{\mathcal H}}\bigr)^{-1}\Bigr\Vert^{-1}\Vert\phi\Vert_{\mathcal H},\qquad\quad\forall\phi\in{\mathcal H},\ \tau\in Q',\ z\in{\mathbb C}\setminus{\mathbb R}.
\end{aligned}
\end{equation*}
The above estimate and the continuity of the operator $\Pi_\stiff$ with respect to $\tau$ (see {\it e.g.} \cite[Proposition 2.2]{Friedlander_old}), and hence of the of the rank-one operator $\Pi_\stiff\vert_{\Po{\mathcal H}},$ yield the claim for $|\tau|\leq \tilde{c}\e.$

 On the other hand, in the case $|\tau|\geq \tilde{c}\e$ we argue, similarly to (\ref{4_12_analysis}), that $\Po M^\stiff(z)\Po+\Lambda_\Delta$ is approximated with $\Po\Lambda^\stiff\Po+\Lambda_\Delta,$ which in turn is estimated by Lemma \ref{lemma:PLP_mu}, while $\Po M^\soft(z)\Po$ is uniformly bounded. As a result, for $|\tau|>\tilde{c}\varepsilon,$ we have 
\begin{equation}
\bigl(\Po M(z)\Po+\Lambda_\Delta\bigr)^{-1}=\bigl(\Po \Lambda^\stiff \Po + \Lambda_\Delta + O(1) + \Po M^\soft(z)\Po\bigr)^{-1} = \frac{\varepsilon^2}{|\tau|^2}\biggl(\mu_*+\frac{O(1)}{{\tilde{c}}^2}\biggr)^{-1},
\label{second_region}
\end{equation}
where $O(1)$-terms are uniformly bounded in the operator norm independently of $\varepsilon, \tau.$ Choosing the constant $\tilde{c}$ to be sufficiently large yields an $O(\e^2/|\tau|^2)$ estimate
for such $\varepsilon,$ $\tau.$



Putting together the above bounds, we obtain an estimate $O(|\tau|^3/\e^2)$ for $|\tau|\leq \e$ and $O(|\tau|^3/\e^2)\cdot O(\e^2/|\tau|^2)$ otherwise, which results in an overall estimate of order $O(|\tau|).$
\end{proof}

We apply Theorem \ref{thm:stability_NRA} to the case $|\tau|\leq \e^{2/3}$. The remaining values of the quasimomentum $\tau$ are considered on the basis of the result of Theorem \ref{thm:main_fin}. Indeed, by Lemmata \ref{lemma:PLP_addon} and \ref{lemma:PLP_mu}, it follows from the formula (\ref{eq:Rhom}) by an essentially unchanged argument to the one of (\ref{second_region}), that
\[
R_\hom^{(\tau)}(z)-(A_0^\soft-z)^{-1}= O(\e^2/|\tau|^2).
\]
Then by Theorem \ref{thm:main_fin} one immediately obtains 
$$
({A}_\e^{(\tau)}-z)^{-1}-(A_0^\soft-z)^{-1} = O\bigl(\e^2/|\tau|^2\bigr),\quad \quad |\tau|\geq \e^{2/3},
$$
so the estimate of the remainder simplifies to $O(\e^{2/3})$ for $|\tau|\geq \e^{2/3}.$

We will now obtain the same estimate for the resolvent $(A_{\beta_0',\beta_1'}-z)^{-1}$ of Theorem \ref{thm:stability_NRA}. Taking into account (\ref{inv_exp}), (\ref{second_region}), using Proposition \ref{prop:operator}, and bearing in mind that $(A_0^\soft-z)^{-1}=O(\varepsilon^2).$
we show that 
$$
({A}_{\beta_0',\beta_1'}-z)^{-1}-(A_0^\soft-z)^{-1} = O\bigl(\e^2/|\tau|^2\bigr),\quad \quad |\tau|\geq \e^{2/3}.
$$
The triangle inequality together with the results of Section \ref{section:4.1} then yields the following result.

\begin{theorem}
\label{thm:stability_NRA_fin}
For $z\in K_\sigma $ and uniformly in $\tau\in Q'$, one has
$$
\bigl({A}_\e^{(\tau)}-z\bigr)^{-1}-\bigl(A_{\Porto+\Lambda_\Delta, \Po}-z\bigr)^{-1} = O\bigl(\e^{2/3}\bigr),
$$
where the operator $A_{\Porto+\Lambda_\Delta, \Po}$ is self-adjoint on $H_\soft\oplus \Pi_{\stiff,0}\mathcal H$. 
\end{theorem}

The latter theorem allows us to achieve the aim of this section. Indeed, it ensures that Theorem \ref{thm:general_homo_result} continues to hold for the operator of Model I with the following replacements: (i) the operator $\Pi_\tau$ is replaced by $\Pi_0:=j_Q \Pi_{\stiff,0}\Po$; (ii) the projection $\P$ is replaced by $\Po$; and (iii) the operator $\breve \Lambda^\stiff$ is replaced by $\breve \Lambda_0^\stiff:=\Po \Lambda^\stiff \Po + \Lambda_\Delta $ of Lemma \ref{lemma:PLP_addon}. Following these amendments, the error estimate $O(\e^2)$ in (\ref{main_est}) roughens up to $O(\e^{2/3}).$

\begin{remark}
The result of Theorem \ref{thm:stability_NRA_fin} does not seem to be order-sharp, owing to the fact that in proving it we have resorted to the triangle inequality. Yet the bottom-right element of the matrix in (\ref{triangle_invert}) seems to prevent the possibility of obtaining a better estimate than that claimed by Theorem \ref{thm:stability_NRA}. The question of whether such an estimate can be obtained remains open, to the best of our knowledge. We hope to return to this matter in \cite{ChEK_future}.
\end{remark}

\subsection{Relationship with the Birman-Suslina spectral germ}\label{section:spectral_germ}

Here we discuss the relationship between our results 
and the spectral germ, which lies at the centre of the spectral approach to homogenisation \cite{BirmanSuslina}, \cite{BirmanSuslina_corr} in the moderate-contrast case. In particular, we refer the reader to the recent paper \cite{Suslina_dyrki} ({\it cf.} \cite{Cher_Serena} for a close development using a different approach), where the setup of moderate-contrast homogenisation in perforated media is treated in full detail. We shall consider a special case of \cite{Suslina_dyrki}, which is our Model I (possibly with non-constant weight $a^2(\cdot/\e)$) with the soft component replaced by voids, supplied with Neumann boundary conditions on $\Gamma$.

The following result of \cite{Suslina_dyrki} then holds: there exists an operator $\mathfrak S$ in $L^2(Q_\stiff)$ given by
\begin{equation}
\mathfrak S= \mathfrak{q}_\tau\bigl\langle\cdot, |Q_\stiff|^{-1/2}\mathbbm 1\bigr\rangle_{L^2(Q_\stiff)}|Q_\stiff|^{-1/2}\mathbbm 1\vert_{Q_\stiff},
\label{BS_germ}
\end{equation}
called the \emph{spectral germ}, where $\mathfrak{q}_\tau$ is a positive definite quadratic form on $Q'\times Q'$ (which we henceforth also refer to as the spectral germ) that governs the asymptotic behaviour of the problem under consideration in the norm-resolvent sense:
$$
\bigl\|(A^{(\tau)}_\e+1)^{-1}- (\e^{-2}\mathfrak S+1)^{-1}\bigr\|_{L^2(Q_\stiff)\to L^2(Q_\stiff)}\leq C\e,
$$
where ${A}_\e^{(\tau)}$ is the Gelfand transform of the original operator.

We compare $\mathfrak{q}_\tau$ with $\mu_*\tau\cdot\tau$ of Lemma \ref{lemma:PLP_mu}, which is a negative definite quadratic form in $\tau$. It can be shown that these forms are related as follows:
\begin{equation}
\mu_*\tau\cdot\tau= -\frac{\vert Q_\stiff\vert}{|\Gamma|}\mathfrak{q}_\tau.
\label{mu_q_rel}
\end{equation}
The proof of this fact in the case of a general smooth weight $a^2(\cdot/\e)$ will appear in the second part \cite{ChEK_future} of the present paper.

The value of this result is that if the Birman-Suslina spectral germ for the perforated medium (with the soft component deleted) is known via the analysis of \cite{Suslina_dyrki} (see also references therein), then the matrix $\mu_*$ is also known, thus making all the objects appearing in the formulation of Theorem \ref{thm:general_homo_result} explicit. Since the expressions for $\Po$ and $\Pi_0$ in this case are straightforward, the norm-resolvent asymptotics of the operator family ${A}_\e^{(\tau)}$ of Model I is essentially reduced to the same question of controlling the spectral germ as in the case of moderate-contrast perforated medium. From this point of view, our result can be seen as a natural, albeit non-trivial, generalisation of the spectral approach of Birman and Suslina. One significant difference, however, is that in our setup only the $O(\e^{2/3})$ error bound is obtained, compared to $O(\e)$ in the perforated case. Including additional corrector terms, however, allows one to improve the order of convergence, to which the second part  \cite{ChEK_future} of our work is devoted.

\section{Time-dispersive formulations}
\label{time_disp_section}
In this section we continue the study of Models I and II, for which we have constructed the resolvent asymptotics, in view to obtain equivalent time-dispersive formulations on the stiff component of the medium. In order to achieve this, we first introduce the orthogonal  projection  ${\mathfrak P}$  of $H_{\rm hom}$ onto $H_{\rm hom}\ominus H_{\rm soft},$ the latter space being ${\mathbb C}^1$ in both models under consideration. Following this, we determine the corresponding Schur-Frobenius complement ${\mathfrak P}(\mathcal A_{\hom}^{(\tau)}-z)^{-1}{\mathfrak P},$ see \cite[p.\,416]{Fuerer} and the state of the art in \cite{Tretter}.

\subsection{Model I. Asymptotically equivalent setup}
Applying Theorem \ref{thm:general_homo_result} together with the results of Section \ref{section:stability} and in particular Theorem \ref{thm:stability_NRA_fin}, we obtain the following explicit description of the leading-order asymptotics
of the homogenised operator $\mathcal{A}_\hom^{(\tau)}$ in (\ref{eq:operator_fin}) for all $\tau\in Q'$ ({\it cf.} the result of \cite[Example (1)]{GrandePreuve}). In view of the fact that we remain within the equivalence class of operators that are $O(\varepsilon)$-close to each other in the norm-resolvent sense, we keep the same notation $\mathcal{A}_\hom^{(\tau)}$ for this leading-order asymptotics. 

Let $H_\hom=L^2(Q_\soft)\oplus \mathbb C^1,$ and set 
\begin{equation}\label{eq:domain_fin_modI}
\dom \mathcal A_{\hom}^{(\tau)}=\Bigl\{(u,\beta)^\top\in H_\hom:\ u\in H^2(Q_\soft),  u|_\Gamma=\bigl\langle u|_\Gamma,\psi_0\bigr\rangle_{{\mathcal H}}\psi_0 \text{ and }  \beta = \kappa\bigl\langle u|_\Gamma, \psi_0\bigr\rangle_{{\mathcal H}}\Bigr\},
\end{equation}
where $u|_\Gamma\in L^2(\Gamma)$ is the trace of the function $u$, $\kappa:=|Q_\stiff|^{1/2}/|\Gamma|^{1/2}$ and $\psi_0= |\Gamma|^{-1/2} \mathbbm 1$. On $\dom \mathcal A_{\hom}^{(\tau)}$ the action of the operator is set by
\begin{equation}\label{eq:operator_fin_modI}
\mathcal A_{\hom}^{(\tau)}\binom{u}{\beta}=
\left(\begin{array}{c}\biggl(\dfrac{1}{\rm i}\nabla+\tau\biggr)^2 u\\[1.0em]
- \kappa^{-1}\bigl\langle \dntau u|_\Gamma, \psi_0\bigr\rangle_{{\mathcal H}}
- \kappa^{-2}\e^{-2}(\mu_*\tau\cdot\tau)\beta
\end{array}\right),
\end{equation}
where $\dntau u = -(\partial u/\partial n+{\rm i}\tau\cdot n u)|_\Gamma$ is the co-normal boundary derivative for $Q_\soft$ and $\mu_*\tau\cdot\tau$ is the negative definite quadratic form in $Q'\times Q'$ provided by Lemma \ref{lemma:PLP_mu}.

The operator $\mathcal A_\hom^{(\tau)}$ is thus the leading-order asymptotics as $\varepsilon\to0,$ in the norm-resolvent 
sense, for the family ${A}_\e^{(\tau)}$; the corresponding error term is estimated as $O(\e^{2/3})$ in the $H\to H$ operator norm (where $H=L^2(Q)$). We remark that in this case the homogenised operator $\mathcal A_\hom^{(\tau)}$ still depends on $\e$. This is to be expected, as so does (under the chosen scaling) the homogenised operator $(\e^{-2}\mathfrak S+1)^{-1}$ of \cite{BirmanSuslina}, see (\ref{BS_germ}) above. Nevertheless, the setup of moderate-contrast homogenisation of \cite{BirmanSuslina} does allow to get rid of the factor $\e^{-2,}$ by employing a non-unitary transform (which is precisely what is done in \cite{BirmanSuslina} and follow-up papers). In the case of critical contrast, it proves impossible to ``hide'' the parameter $\e$ in this way.

\subsection{Model II. Asymptotically equivalent setup}
Applying Theorem \ref{thm:general_homo_result} together with the results of Section \ref{section:IandII}, we obtain the following explicit description of the homogenised operator $\mathcal{A}_\hom^{(\tau)}$ for all $\tau\in Q'$ ({\it cf.} the result of \cite[Example (0)]{GrandePreuve} and also of \cite{CherKis}; note the form that the Datta-Das Sarma \cite{Datta} boundary conditions of the mentioned works admit in the PDE setup). Let $H_\hom=L^2(Q_\soft)\oplus \mathbb C^1,$ and set ({\it cf.} (\ref{eq:domain_fin_modI}))
\begin{equation}\label{eq:domain_fin_modII}
\dom \mathcal A_{\hom}^{(\tau)}=\Bigl\{(u,\beta)^\top\in H_\hom:\ u\in H^2(Q_\soft),  u|_\Gamma=\bigl\langle u|_\Gamma,\psi_\tau\bigr\rangle_{{\mathcal H}}\psi_\tau \text{ and }  \beta = \kappa\bigl\langle u|_\Gamma, \psi_\tau\bigr\rangle_{{\mathcal H}}\Bigr\},
\end{equation}
where $u|_\Gamma,$ $\kappa$ are as in (\ref{eq:domain_fin_modI})
and $\psi_\tau= |\Gamma|^{-1/2}\exp(-{\rm i}\tau\cdot x)|_\Gamma.$ 
The action of the operator is set by
\begin{equation}\label{eq:operator_fin_modII}
\mathcal A_{\hom}^{(\tau)}\binom{u}{\beta}=
\left(\begin{array}{c}\biggl(\dfrac{1}{\rm i}\nabla+\tau\biggr)^2 u\\[1.0em]
-\kappa^{-1}\bigl\langle \dntau u|_\Gamma, \psi_\tau\bigr\rangle_{{\mathcal H}}
\end{array}\right),\qquad \binom{u}{\beta}\in \dom \mathcal A_{\hom}^{(\tau)},
\end{equation}
where $\dntau u = -(\partial u/\partial n+{\rm i}\tau\cdot n u)|_\Gamma$ is the co-normal derivative for $Q_\soft$. We remark that in this case the homogenised operator $\mathcal A_\hom^{(\tau)}$ does not depend on $\e$, and therefore the asymptotics claimed by Theorem \ref{thm:general_homo_result} is in fact the limit proper; the error is estimated as $O(\e^2)$.

\subsection{Asymptotic dispersion relations}
Due to similarities in the formulae, we  shall consider Models I and II simultaneously. To this end, for $\tau\in Q'$ we set
\begin{equation}
\Gamma_\tau\left(\begin{matrix}u\\[0.2em] \beta\end{matrix}\right)=
-\bigl\langle\dntau u|_\Gamma,\psi\bigr\rangle_{{\mathcal H}}-\kappa^{-1}\e^{-2}(\mu_*\tau\cdot\tau)\beta,
\label{Gamma}
\end{equation}
where in Models I and II we have $\psi=\psi_0$ and $\psi=\psi_\tau,$ $\mu_*=0,$ respectively.

The problem of calculating ${\mathfrak P}(\mathcal A_{\hom}^{(\tau)}-z)^{-1}{\mathfrak P}$
consists in determining $\beta$ that solves
\begin{equation}
\label{diff_part}
\begin{aligned}
-&(\nabla+{\rm i}\tau)^2u-zu=0,\\[0.3em]
&\left(\begin{matrix}u\\[0.2em] \beta\end{matrix}\right)\in\dom \mathcal A_{\hom}^{(\tau)},\quad \frac{1}{\kappa}\,\Gamma_\tau\left(\begin{matrix}u\\[0.2em] \beta\end{matrix}\right)-z\beta=\delta,
\end{aligned}
\end{equation}
where $z$ is such that the resolvent $(\mathcal A_{\hom}^{(\tau)}-z)^{-1}$ exists (which is the case, in particular, for $z\in K_\sigma$).
In order to express $u$ from (\ref{diff_part})
we represent it as a sum of two functions: one of them is a solution to the related inhomogeneous Dirichlet problem, while the other takes care of the boundary condition. More precisely, consider the solution $v$ to the problem
\begin{equation*}
-(\nabla+{\rm i}\tau)^2v=0,\qquad\qquad
v|_\Gamma=\psi,
\end{equation*}
{\it i.e.}
\begin{equation}
\label{function_v}
v = \Pi_\soft \psi.
\end{equation}
The function $\widetilde{u}:=u-\beta\kappa^{-1}v$
satisfies
\begin{equation*}
-(\nabla+{\rm i}\tau)^2\widetilde{u}-z\widetilde{u}=z\beta\kappa^{-1}v,\qquad 
 \widetilde{u}|_\Gamma=0,
\end{equation*}
or, equivalently, one has
\begin{equation*}
\widetilde{u}=z\beta\kappa^{-1}(A_{0}^\soft-z)^{-1}v,
\end{equation*}
where $A_{0}^\soft$ is, as above, the Dirichlet operator in $L^2(Q_\soft)$ associated with the differential expression
$
-(\nabla+{\rm i}\tau)^2.
$
We now write the ``boundary''  part of the system (\ref{diff_part})
as
\begin{equation}
K(\tau, z)\beta-z\beta=\delta,
\label{K_eq}
\end{equation}
where
\begin{equation}
K(\tau, z):=\dfrac{1}{\kappa^2}\left\{z\Gamma_\tau\left(\begin{matrix}
(A_{0}^\soft-z)^{-1}v\\[0.3em] 0\end{matrix}\right)+
\Gamma_\tau\left(\begin{matrix}v\\[0.2em] \kappa\end{matrix}\right)\right\}.
\label{K_expr}
\end{equation}
Thus ${\mathfrak P}(\mathcal A_{\hom}^{(\tau)}-z)^{-1}{\mathfrak P}$ is the operator of multiplication in ${\mathbb C }^1$ by $(K(\tau, z)-z)^{-1}.$



The formula (\ref{K_expr}) shows, in particular, that the dispersion function $K$ is singular only at eigenvalues of the Dirichlet Laplacian on the soft component. It allows to compute $K$ in terms of the spectral decomposition of $A_{0}^\soft,$ {\it cf.} \cite{Zhikov2000}. In order to see this, we represent the action of the resolvent $(A_{0}^\soft-z)^{-1}$ as a series in terms of the normalised eigenfunctions $\{\varphi_j^{(\tau)}\}_{j=1}^\infty$ of $A_0^\soft,$ which yields
\begin{equation}
K(\tau, z)=\dfrac{1}{\kappa^2}\biggl\{z\sum_{j=1}^\infty\dfrac{\bigl\langle v, \varphi_j^{(\tau)}\bigr\rangle_{L^2(Q_\soft)}}{\lambda_j^{(\tau)}-z}\Gamma_\tau\left(\begin{matrix}
\varphi_j^{(\tau)}\\[0.3em] 0\end{matrix}\right)+
\Gamma_\tau\left(\begin{matrix}v\\[0.2em] \kappa\end{matrix}\right)\biggr\}.
\label{K_general1}
\end{equation}
where $\lambda_j^{(\tau)},$ $j=1,2,3,\dots,$ are the corresponding eigenvalues and $v$ is defined in \eqref{function_v}. In each case we consider the problem (\ref{diff_part}),
where operator $\Gamma_\tau$ depends on the specific example at hand.

\subsubsection*{Model I}
Here the eigenvalues of the ($\tau$-dependent) Dirichlet operators $A_0^\soft$ are independent of $\tau,$ and we write $\lambda_j^{(\tau)}=\lambda_j,$ $\tau\in Q',$ where the corresponding eigenfunctions satisfy $\varphi_j^{(\tau)}(x)=\varphi_j^{(0)}(x)\exp(-{\rm i}\tau\cdot x),$ $x\in Q_\soft.$ By the second Green's identity we calculate
\begin{align*}
\Gamma_\tau\binom{\varphi_j^{(\tau)}}{0}&=-\bigl\langle \dntau \varphi_j^{(\tau)}|_\Gamma,\psi_0\bigr\rangle_{{\mathcal H}}=
\bigl({|\tau|^2-\lambda_j}\bigr)\bigl\langle\varphi_j^{(\tau)},\widetilde{\Psi}_0\bigr\rangle_{L^2(Q_\soft)},\\[0.5em]
\Gamma_\tau\binom{v}{\kappa}&={|\tau|^2}\bigl\langle v,\widetilde{\Psi}_0\bigr\rangle_{L^2(Q_\soft)}-\e^{-2}\mu_*\tau\cdot\tau,\qquad j=1,2,\dots,
\end{align*}
where 
\begin{equation}
\label{hourglass}
\widetilde{\Psi}_0:=|\Gamma|^{-1/2}\mathbbm 1|_{Q_\soft},
\end{equation}
so that $\widetilde{\Psi}_0|_{\Gamma}=\psi_0$. Next, we find $v=\widetilde{\Psi}_0-|\tau|^2 (A_0^\soft)^{-1}\widetilde{\Psi}_0$, which allows us to determine
\begin{align*}
\langle v,\varphi_j^{(\tau)}\rangle_{L^2(Q_\soft)}&=\frac{\lambda_j-|\tau|^2}{\lambda_j}\bigl\langle \widetilde{\Psi}_0,\varphi_j^{(\tau)}\bigr\rangle_{L^2(Q_\soft)},\quad\ \ \ j=1,2,\dots,\\[0.5em] 
\bigl\langle v,\widetilde{\Psi}_0\bigr\rangle_{L^2(Q_\soft)}&=\bigl\|\widetilde{\Psi}_0\bigr\|_{L^2(Q_\soft)}^2 -|\tau|^2 \sum_{j=1}^\infty\frac{\Bigl|\bigl\langle \widetilde{\Psi}_0,\varphi_j^{(\tau)}\bigr\rangle_{L^2(Q_\soft)}\Bigr|^2}{\lambda_j}=
\sum_{j=1}^\infty\frac{\lambda_j-|\tau|^2}{\lambda_j}\Bigl|\bigl\langle \widetilde{\Psi}_0,\varphi_j^{(\tau)}\bigr\rangle_{L^2(Q_\soft)}\Bigr|^2.
\end{align*}
Thus, ultimately one has
\begin{align}
K_{\rm I}(\tau, z)&=\dfrac{|\Gamma|}{|Q_\stiff|}\biggl\{\bigl(|\tau|^2-z\bigr)\sum_{j=1}^\infty
\frac{\lambda_j-|\tau|^2}{\lambda_j-z}
\Bigl|\bigl\langle \widetilde{\Psi}_0,\varphi_j^{(\tau)}\bigr\rangle_{L^2(Q_\soft)}\Bigr|^2 
-\e^{-2}\mu_*\tau\cdot\tau\biggr\}\label{former_form}\\[0.4em]
&=\dfrac{|\Gamma|}{|Q_\stiff|}\Bigl\{\bigl(|\tau|^2-z\bigr)\Bigl\langle\bigl(A^\soft_0-|\tau|^2\bigr)(A_0^\soft-z)^{-1}\widetilde{\Psi}_0,\widetilde{\Psi}_0\Bigr\rangle_{L^2(Q_\soft)}-\e^{-2}\mu_*\tau\cdot\tau\Bigr\},
\label{eq:K_modI}
\end{align}
which is treated, for each $\tau\in Q',$ as a meromorphic function in $z.$


\begin{remark}\label{rem:nemtsy}
Formula (\ref{former_form}) reveals the meaning of the additional assumption,\footnote{This assumption is known to be generically true \cite{Uhlenbeck}.} imposed both in \cite{HempelLienau_2000} and \cite{Friedlander}, that $\mathbbm 1|_{Q_\soft}$ has to be non-orthogonal to all eigenfunctions of the Dirichlet Laplacian pertaining to the soft component of the medium. We also remark that the limiting spectrum of our Model I, following from Theorem \ref{thm:general_homo_result}, coincides with the set of $z$ such that $K_{\rm I}(\tau,z)=z$ for some $\tau\in Q'$, thus allowing to use \eqref{eq:K_modI} to derive the results of \cite{HempelLienau_2000,Friedlander}.
\end{remark}

\begin{remark} An alternative representation of $K_{\rm I}(\tau, z)$ is readily available. Denoting $\widetilde{\Psi}_0^{(\tau)}:=\Pi_\soft \psi_0\equiv v,$ {\it cf.} (\ref{function_v}), one 
has, by $A_0^\soft\varphi_j^{(\tau)}=\lambda_j \varphi_j^{(\tau)}$ and the second Green's identity,
\begin{equation}
\Gamma_\tau \binom{\varphi_j^{(\tau)}}{0}=-\bigl\langle \dntau \varphi_j^{(\tau)}|_\Gamma, \psi_0\bigr\rangle_{{\mathcal H}}= -\lambda_j \bigl\langle \varphi_j^{(\tau)}, \widetilde{\Psi}_0^{(\tau)}\bigr\rangle_{L^2(Q_\soft)},
\label{Gamma_tau}
\end{equation}
and hence, proceeding as in the proof of Lemma \ref{lemma:PLP}, {\it cf.} in particular the formulae (\ref{lambda_action}), (\ref{P0_id}), see also Remark \ref{Remark4_9},
we obtain
\begin{align}
\Gamma_\tau \binom{\widetilde{\Psi}_0^{(\tau)}}{\kappa}&=-\bigl\langle\Lambda^\soft\psi_0,\psi_0\bigr\rangle_{{\mathcal H}}-\e^{-2}\mu_*\tau\cdot\tau\label{triangle}\\[0.4em]
&=-\bigl\langle \Po \Lambda^\soft\psi_0,\psi_0\bigr\rangle_{{\mathcal H}}-\e^{-2}\mu_*\tau\cdot\tau\nonumber\\[0.4em]
&=|\tau|^2 \frac{|Q_\soft|}{|\Gamma|} -\e^{-2}\mu_*\tau\cdot\tau+O\bigl(|\tau|^3\bigr).\nonumber
\end{align}
Thus, one has
\begin{align}
K_{\rm I}(\tau, z)&=\dfrac{|\Gamma|}{|Q_\stiff|}\biggl\{-z\sum_{j=1}^\infty\frac{\lambda_j}{\lambda_j-z}
\Bigl|\bigl\langle \widetilde{\Psi}_0^{(\tau)},\varphi_j^{(\tau)}\bigr\rangle_{L^2(Q_\soft)}\Bigr|^2
+|\tau|^2 \frac{|Q_\soft|}{|\Gamma|}-\e^{-2}\mu_*\tau\cdot\tau+O\bigl(|\tau|^3\bigr)\biggr\}
\nonumber\\[0.4em]
&=\dfrac{|\Gamma|}{|Q_\stiff|}\Bigl\{-z\bigl\langle A_0^\soft(A_0^\soft-z)^{-1}\widetilde{\Psi}_0^{(\tau)},\widetilde{\Psi}_0^{(\tau)} \bigr\rangle_{L^2(Q_\soft)}
+|\tau|^2 \frac{|Q_\soft|}{|\Gamma|}-\e^{-2}\mu_*\tau\cdot\tau+O\bigl(|\tau|^3\bigr)\Bigr\},
\label{eq:K_modI_alt}
\end{align}
where $\widetilde{\Psi}_0^{(\tau)}=|\Gamma|^{-1/2}\Pi_\soft \mathbbm 1\vert_\Gamma.$
\end{remark}

\begin{remark} 
In the proof of Lemma \ref{sin_label} we also use the representation
\begin{equation}
K_{\rm I}(\tau, z)=-\dfrac{|\Gamma|}{|Q_\stiff|}\Bigl\{z\bigl\langle A_0^\soft(A_0^\soft-z)^{-1}\widetilde{\Psi}_0^{(\tau)},\widetilde{\Psi}_0^{(\tau)} \bigr\rangle_{L^2(Q_\soft)}
+\bigl\langle\Lambda^\soft\psi_0,\psi_0\bigr\rangle_{{\mathcal H}}+\e^{-2}\mu_*\tau\cdot\tau\Bigr\},
\label{double_triangle}
\end{equation}
which follows from (\ref{K_general1}), (\ref{Gamma_tau}), (\ref{triangle}).
\end{remark}

\subsubsection*{Model II}
Here we have 
\begin{equation}
\label{psi_II}
\psi=\psi_\tau=|\Gamma|^{-1/2}\exp(-{\rm i}\tau\cdot x)\bigr|_{x\in\Gamma},
\end{equation}
and we set $\widetilde{\Psi}_\tau:=\Pi_\soft \psi_\tau\equiv v,$ {\it cf.} (\ref{function_v}). One has, by the second Green's identity,
$$
\Gamma_\tau \binom{\varphi_j^{(\tau)}}{0}=-\bigl\langle \dntau \varphi_j^{(\tau)}\big|_\Gamma, \psi_\tau\bigr\rangle_{{\mathcal H}}=
-\lambda_j^{(\tau)}\bigl\langle \varphi_j^{(\tau)}, \widetilde{\Psi}_\tau\bigr\rangle_{L^2(Q_\soft)}
$$
and
$$
\Gamma_\tau \binom{v}{\kappa}=-\bigl\langle \dntau \widetilde{\Psi}_\tau\big|_\Gamma, \psi_\tau\bigr\rangle_{{\mathcal H}}=
-\bigl\langle \P \Lambda^\soft\psi_\tau, \psi_\tau\bigr\rangle_{{\mathcal H}}.
$$
Therefore, one has
\begin{align*}
K_{\rm II}(\tau, z)&=-\dfrac{|\Gamma|}{|Q_\stiff|}\biggl\{z\sum_{j=1}^\infty\frac{\lambda_j^{(\tau)}}{\lambda_j^{(\tau)}-z}
\Bigl|\bigl\langle \widetilde{\Psi}_\tau,\varphi_j^{(\tau)}\bigr\rangle_{L^2(Q_\soft)}\Bigr|^2
+\bigl\langle \Lambda^\soft \psi_\tau,\psi_\tau\bigr\rangle_{{\mathcal H}}\biggr\}\nonumber\\[0.5em]
&=-\dfrac{|\Gamma|}{|Q_\stiff|}\Bigl\{z\bigl\langle A_0^\soft\bigl(A_0^\soft-z\bigr)^{-1}\widetilde{\Psi}_\tau,\widetilde{\Psi}_\tau\bigr\rangle_{L^2(Q_\soft)}
+\bigl\langle \Lambda^\soft \psi_\tau,\psi_\tau\bigr\rangle_{{\mathcal H}}\Bigr\}.
\end{align*}

For each $\tau\in Q',$ 
 $z\in K_\sigma,$ consider the solution $V_\tau\equiv(A_0^\soft-z\bigr)^{-1}\widetilde{\Psi}_\tau$ to the Dirichlet problem
\begin{equation}
-(\nabla +{\rm i}\tau)^2V_\tau=zV_\tau+\widetilde{\Psi}_\tau,\qquad V_\tau\bigr\vert_{\Gamma}=0.
\label{Vtau_problem}
\end{equation}
Rearranging the terms and using Green's formula, we obtain 
\begin{align*}
z\bigl\langle A_0^\soft\bigl(A_0^\soft&-z\bigr)^{-1}\widetilde{\Psi}_\tau,\widetilde{\Psi}_\tau\bigr\rangle_{L^2(Q_\soft)}
+\bigl\langle \Lambda^\soft \psi_\tau,\psi_\tau\bigr\rangle_{{\mathcal H}}\\[0.3em]
&=z\bigl\langle\widetilde{\Psi}_\tau,\widetilde{\Psi}_\tau\bigr\rangle_{L^2(Q_\soft)}+z^2\bigl\langle(A_0^\soft-z\bigr)^{-1}\widetilde{\Psi}_\tau,\widetilde{\Psi}_\tau\bigr\rangle_{L^2(Q_\soft)}
+\bigl\langle \Lambda^\soft \psi_\tau,\psi_\tau\bigr\rangle_{{\mathcal H}}\\[0.3em]
&=z\bigl\langle\widetilde{\Psi}_\tau,\widetilde{\Psi}_\tau\bigr\rangle_{L^2(Q_\soft)}+z\bigl\langle zV_\tau,\widetilde{\Psi}_\tau\bigr\rangle_{L^2(Q_\soft)}
+\bigl\langle \Lambda^\soft \psi_\tau,\psi_\tau\bigr\rangle_{{\mathcal H}}\\[0.35em]
&=z\bigl\langle -(\nabla +{\rm i}\tau)^2V_\tau,\widetilde{\Psi}_\tau\bigr\rangle_{L^2(Q_\soft)}
+\bigl\langle \Lambda^\soft \psi_\tau,\psi_\tau\bigr\rangle_{{\mathcal H}}\\[0.35em]
&=-\bigl\langle (\nabla +{\rm i}\tau)(zV_\tau+\widetilde{\Psi}_\tau)\cdot n,\psi_\tau\bigr\rangle_{\mathcal H},
\end{align*}
where, in accordance with our convention, $n$ is the normal to $\Gamma$ pointing into $Q_\stiff.$
Expressing $V_\tau$ in terms of $\widetilde{\Psi}_\tau$ from (\ref{Vtau_problem}) and using the definition of the $M$-matrix on the soft component $M_{\rm soft},$ we obtain 
\begin{equation}
K_{\rm II}(\tau, z)=-\dfrac{|\Gamma|}{|Q_\stiff|}\bigl\langle M_{\rm soft}(z)\psi_\tau,\psi_\tau\bigr\rangle_{{\mathcal H}}.
\label{KII}
\end{equation}
The formula ({\ref{KII}}) yields, in particular, an asymptotic description for small $\tau$ (which will be discussed in \cite{ChEK_future}), showing that in the case of Model II the integrated density of states admits a limit proper as $\varepsilon\to0,$ as opposed to the case of Model I, see \cite{Friedlander} for details.
We thus 
and arrive at the following statement.
\begin{lemma}
\label{complement_final}
 In both Model I and Model II, the action of the Schur-Frobenius complement ${\mathfrak P}(\mathcal A_{\hom}^{(\tau)}-z)^{-1}{\mathfrak P}$
is represented as the operator of multiplication by $(K(\tau, z)-z)^{-1},$ where the dispersion function $K$ is given by the following formulae:
\begin{itemize}
  \vskip 0.15cm
  \item For Model I, $K=K_{\rm I}(\tau, z)$ defined by \eqref{eq:K_modI} or equivalently by \eqref{eq:K_modI_alt};
  \vskip 0.15cm
  \item For Model II, $K=K_{\rm II}(\tau, z)$ defined by \eqref{KII}.
\end{itemize}
Here $\tau\in Q',$ the functions $\varphi_j^{(\tau)},$ $j=1,2,\dots,$ are the normalised eigenfunctions of the self-adjoint operator $A_0^\soft$ defined on $Q_\soft$ by the differential expression $-(\nabla+{\rm i}\tau)^2$ subject to Dirichlet boundary conditions on $\Gamma,$ and $\lambda_j^{(\tau)},$ $j=1,2,\dots,$ are the corresponding eigenvalues.
\end{lemma}

\subsection{Effective macroscopic problems in $\mathbb R^d$ with frequency dispersion}
\label{eff_macro}
Here we shall interpret the Schur-Frobenius complements constructed in the previous section as a result of applying the Gelfand transform, see Section \ref{Gelfand_section}, to a homogenised medium in $\mathbb R^d$. 

\subsubsection{Preparatory material} 
\label{prep_mat}
In order to make our result more useful for applications, we seek to obtain the whole-space setting as the homogenised medium, see \cite{Cher_Serena} and references therein, contrary to the approach of \cite{Suslina_dyrki} where the homogenised medium is represented as a perforated one, albeit described by an operator with constant (homogenised) symbol away from the Neumann perforations.

To this end, we put $t=\e^{-1}\tau$ (which corresponds to the inverse unitary rescaling to the one of Section 2), then unitarily immerse the $L^2$ space of functions of $t$  into the $L^2$ space of functions of $t$ and $x$ corresponding to the stiff component of the original medium, by the formula
\begin{equation}
\beta(t)\mapsto\beta(t)\e^{-d/2}\frac{1}{\sqrt{|Q_\stiff|}}{\mathbbm 1}(x),\qquad x\in \varepsilon Q_{\rm stiff},
\label{first_embedding}
\end{equation}
and write the effective problem (\ref{K_eq})
in the form
\begin{equation}
K(\varepsilon t, z)\beta(t)\e^{-d/2}\frac{1}{\sqrt{|Q_\stiff|}}{\mathbbm 1}(x)-z\beta(t)\e^{-d/2}\frac{1}{\sqrt{|Q_\stiff|}}{\mathbbm 1}(x)=\delta(t)\e^{-d/2}\frac{1}{\sqrt{|Q_\stiff|}}{\mathbbm 1}(x),\qquad t\in \e^{-1}Q'.
\label{proj_eq}
\end{equation}
The solution operator for (\ref{proj_eq}), namely
$$
\delta(t)\e^{-d/2}\frac{1}{\sqrt{|Q_\stiff|}}{\mathbbm 1}(x)\mapsto \beta(t)\e^{-d/2}\frac{1}{\sqrt{|Q_\stiff|}}{\mathbbm 1}(x)\ \ {\rm such\ that\ (\ref{proj_eq})\  holds,}
$$
is the composition of a projection operator in $L^2\bigl(\e Q_\stiff\times \e^{-1}Q'\bigr)$ onto constants in $x$ and multiplication by the function $\bigl(K(\varepsilon t, z)-z\bigr)^{-1},$ as follows:
\begin{equation}
\bigl(K(\varepsilon t, z)-z\bigr)^{-1}\biggl\langle\cdot, \e^{-d/2}\frac{1}{\sqrt{|Q_\stiff|}}{\mathbbm 1}\biggr\rangle_{L^2(\varepsilon Q_\stiff)}\e^{-d/2}\frac{1}{\sqrt{|Q_\stiff|}}{\mathbbm 1}(x),\qquad x\in\varepsilon Q_\stiff,
\label{sol_proj_m}
\end{equation}
for all $z$ such that $K(\varepsilon t, z)-z$ is invertible, in particular, for $z\in K_\sigma.$ We map the operator of \eqref{sol_proj_m} unitarily to the rank-one operator in $L^2\bigl(\e Q\times\varepsilon^{-1}Q')$ defined as
\begin{equation}
\bigl(K(\varepsilon t, z)-z\bigr)^{-1}\bigl\langle\cdot, \e^{-d/2}{\mathbbm 1}\bigr\rangle_{L^2(\varepsilon Q)}\e^{-d/2}{\mathbbm 1}(x),\qquad x\in \varepsilon Q,\ \ t\in \e^{-1}Q'.
\label{sol_proj}
\end{equation}

Note that we have thus used a sequence of two unitary embeddings, where the first one, given by (\ref{first_embedding}), essentially maps the one-dimensional space 
${\mathbb C}$ back to $L^2(Q_\stiff),$ whereas the second one ``eliminates" the inclusions and results in a homogeneous medium. For this reason we present these embeddings separately, instead of combining them into a single unitary operator.

The sought representation in ${\mathbb R}^d$ of the Schur-Frobenius complement of Lemma \ref{complement_final} is obtained by sandwiching the operator (\ref{sol_proj}) with the 
Gelfand transform $G$ defined in \eqref{scaled_Gelfand}
and its inverse $G^*$ given by (\ref{inverse_scaled_Gelfand}),
so that the overall operator is given by
$$
G^*\Bigl\{\bigl(K(\varepsilon t, z)-z\bigr)^{-1}\bigl\langle G
\,\cdot, \e^{-d/2}{\mathbbm 1}\bigr\rangle_{L^2(\varepsilon Q)}\e^{-d/2}{\mathbbm 1}(x)\Bigr\}.
$$
This results in the mapping $F\mapsto\Upsilon^\varepsilon_zF,$ where
\begin{align}
(\Upsilon^\varepsilon_z F)(x):=&\left(\frac{\varepsilon}{2\pi}\right)^{d/2}\int_{\e^{-1}Q'}\bigl(K(\varepsilon t, z)-z\bigr)^{-1}\bigl\langle 
GF, \e^{-d/2}{\mathbbm 1}\bigr\rangle_{L^2(\varepsilon Q)}(t)\e^{-d/2}{\mathbbm 1}(x)\exp({\rm i}t\cdot x)dt\nonumber\\[0.6em]
=&(2\pi)^{-d/2}\int_{\e^{-1}Q'}\bigl(K(\varepsilon t, z)-z\bigr)^{-1}\widehat{F}(t)\exp({\rm i}t\cdot x)dt,\qquad x\in{\mathbb R}^d,\quad \varepsilon>0,
\label{Uform}
\end{align}
which yields the required effective problem on ${\mathbb R^d}.$ In deriving (\ref{Uform}) we use the fact that the inner product of the Gelfand transform of a function and the identity coincides
with its Fourier transform. Here
$$
\widehat{F}(t)=({2\pi})^{-d/2}\int_{{\mathbb R}^d}F(x)\exp(-{\rm i}x\cdot t)dx,\qquad t\in{\mathbb R^d},
$$
is the Fourier transform of $F.$
Combining Theorems \ref{thm:general_homo_result} and \ref{thm:stability_NRA_fin}, we arrive at the following statement.
\begin{theorem}
\label{thm:pseudodifferential}
The direct integral of Schur-Frobenius complements
\begin{equation}
\oplus\int_{Q'}P_\stiff\bigl(A_\varepsilon^{(\tau)}-z\bigr)^{-1}\Bigr\vert_{L^2(Q_\stiff)}d\tau
\label{dir_int}
\end{equation}
is $O(\varepsilon^r)$-close in the uniform operator-norm topology, uniformly in $z\in K_\sigma,$ to an operator unitary equivalent\footnote{Here the unitary operator corresponding to the equivalence does not depend on $\varepsilon,$ see Corollary \ref{cor:pseudodifferential}.} to the (pseudo-differential) operator $\Upsilon^\varepsilon_z.$ 
Here $r=2/3$, $K=K_{\rm I}$ for Model I and $r=2$, $K=K_{\rm II}$ for Model II.
\end{theorem}
The direct integral (\ref{dir_int}) is the composition of the original resolvent family $(A_\varepsilon-z)^{-1}$ applied  to functions supported by the stiff component in $\mathbb R^d$ and the orthogonal projection onto the same stiff component. Applying a version of the scaled transform to the restriction of functions in $L^2({\mathbb R}^d)$ to the stiff component of the original $\varepsilon$ medium, we obtain the following result for the operator family $A_\varepsilon.$ 
\begin{corollary}
\label{cor:pseudodifferential}
Denote by $\Omega_\stiff^\varepsilon:=\cup_{n\in{\mathbb Z}^d}\varepsilon(Q_\stiff+n)$ the stiff component of the original composite, see (\ref{eq:generic_hom}), (\ref{weight}), and by $P_\stiff^\varepsilon$ the orthogonal projection of $L^2({\mathbb R}^d)$ onto $L^2(\Omega_\stiff^\varepsilon).$ The operator $P_\stiff^\varepsilon(A_\varepsilon-z)^{-1}\vert_{L^2(\Omega_\stiff^\varepsilon)}$ is  $O(\varepsilon^r)$-close in the norm-resolvent sense to the operator $\Theta^\e_{\rm hom}\Upsilon^\varepsilon_z(\Theta^\e_{\rm hom})^*,$ where $\Upsilon^\varepsilon_z$ is given by (\ref{Uform}), $\Theta^\e_{\rm hom}$ is a unitary operator\footnote{An explicit description of the operator $\Theta^\varepsilon_{\rm hom}$is somewhat ugly and depends on whether one deals with Model I or Model II. In both cases it is based on a unitary operator $V_\tau$ which is 
block-diagonal with respect to the decomposition of $L^2(Q_{\rm stiff})$ into ${\rm span}\{\eta_\tau\}$ and its orthogonal complement.
  In the Model I $\eta_\tau=|Q_{\rm stiff}|^{-1}{\mathbbm 1}_{Q_{\rm stiff}},$ and in Model II $\eta_\tau=\Vert\Pi_{\rm stiff}\psi_\tau\Vert^{-1}\Pi_{\rm stiff}\psi_\tau,$ {\it cf.} Theorem \ref{thm:general_homo_result}.  The operator $V_\tau$  sends $\eta_\tau$ to $|Q_{\rm stiff}|^{-1}{\mathbbm 1}_{Q_{\rm stiff}}$ and acts as an arbitrary unitary operator on its orthogonal complement.  The operator $\Theta^\varepsilon_{\rm hom}$ is then given by $\Theta^\varepsilon_{\rm hom}=G^{-1}\Phi_\varepsilon^*V_\tau \widetilde{\Phi}_\varepsilon G_{\rm stiff},$ where $G$ is the Gelfand transform of Section \ref{Gelfand_section}, $\Phi_\varepsilon$ is the unitary rescaling defined in Section \ref{Gelfand_section},
  $\widetilde{\Phi}_\varepsilon$ is a version of this unitary rescaling obtained by a restriction to $L^2(\varepsilon Q_{\rm stiff}),$
 and $G_{\rm stiff}$ is an appropriate  version of the Gelfand transform on $\Omega^\varepsilon_{\rm stiff}.$} from $L^2({\mathbb R}^d)$ to $L^2(\Omega_\stiff^\varepsilon),$ and the exponent $r$ is as in Theorem \ref{thm:pseudodifferential}. 

Note that if one applies a unitary rescaling to the original medium, so that $\Omega_\stiff^\varepsilon$ is transformed to $\Omega_\stiff^1:=\cup_{n\in{\mathbb Z}^d}(Q_\stiff+n),$ {\it cf.} the setups of \cite{HempelLienau_2000, Friedlander}, then the family $\Theta^\e_{\rm hom}$ is mapped to an $\varepsilon$-independent unitary operator.
\end{corollary}

We will now discuss the implications of the above result for the effective time-dispersive medium in each of the two models.

\subsubsection{Model I}
Our first result in this direction is the following lemma. Define $\widetilde {K}_{\rm I}(\tau,z)$ by the formula ({\it cf.} (\ref{eq:K_modI}))
\begin{equation}\label{eq:Ktilde}
\widetilde {K}_{\rm I}(\tau, z)=-\dfrac{|\Gamma|}{|Q_\stiff|}\biggl\{z\sum_{j=1}^\infty
\frac{\lambda_j}{\lambda_j-z}
\Bigl|\bigl\langle\widetilde{\Psi}_0,\varphi_j^{(0)}\bigr\rangle_{L^2(Q_\soft)}\Bigr|^2+\e^{-2}\mu_*\tau\cdot\tau\biggr\},\qquad \tau\in Q',
\end{equation}
where ({\it cf.} (\ref{hourglass})) $\widetilde{\Psi}_0=|\Gamma|^{-1/2}\mathbbm 1|_{Q_\soft}.$ The following estimate allows us to replace $K_{\rm I}$ by $\widetilde{K}_{\rm I}$ in the expression for the operator asymptotics we are seeking for Model I.


\begin{lemma} 
\label{sin_label}
For all $F\in L^2({\mathbb R}^d),$ set ({\it cf.} (\ref{Uform}))
\begin{equation}
(\widetilde{\Upsilon}^\varepsilon_zF)(x):=(2\pi)^{-d/2}\int_{\varepsilon^{-1}Q'}\bigl(\widetilde{K}_{\rm I}(\varepsilon t, z)-z\bigr)^{-1}\widehat{F}(t)\exp({\rm i}t\cdot x)dt,\qquad x\in{\mathbb R}^d,\quad \varepsilon>0.
\label{Ups_00}
\end{equation}
Then following estimate holds:
\begin{equation}
\bigl\Vert\Upsilon^\varepsilon_z-\widetilde{\Upsilon}^\varepsilon_z\bigr\Vert_{L^2({\mathbb R^d})\to L^2({\mathbb R^d})}=O(\varepsilon), \quad \varepsilon\to0,\quad z\in K_\sigma.
\label{ups_est}
\end{equation}
\end{lemma}
\begin{proof} 

For each $z\in{\mathbb C}\setminus{\mathbb R},$ consider the function 
\[
{\mathfrak D}(z, \tau):=\left\{\begin{array}{ll}\bigl(K_{\rm I}(\tau, z)-z\bigr)^{-1}-\bigl(\widetilde{K}_{\rm I}(\tau,z)-z\bigr)^{-1},\quad\tau\in Q',\\[0.2cm]
0,\quad\tau\in{\mathbb R}^d\setminus Q'.\end{array}\right.
\] 
Using the expressions (\ref{eq:K_modI_alt}), (\ref{double_triangle}), (\ref{eq:Ktilde}), we write
\begin{equation}
{\mathfrak D}(z, \tau)=\left\{\begin{array}{ll}O\bigl(|\tau|\bigr)\Bigl({\mathfrak m}(z,\tau)+\bigl\langle\Lambda^{\rm soft}\psi_0, \psi_0\bigr\rangle_{\mathcal H}+\varepsilon^{-2}\mu_*\tau\cdot\tau\Bigr)^{-1}\bigl({\mathfrak m}(z,0)+\varepsilon^{-2}\mu_*\tau\cdot\tau\bigr)^{-1},\quad\tau\in Q',\\[0.7em]0,\quad\tau\in{\mathbb R}^d\setminus Q',\end{array}\right.
\label{K_difference}
\end{equation}
where 
\begin{align*}
{\mathfrak m}(z, \tau)&:=z\biggl\{\bigl\langle A^\soft_0(A_0^\soft-z)^{-1}\widetilde{\Psi}_0^{(\tau)},\widetilde{\Psi}_0^{(\tau)}\bigr\rangle_{L^2(Q_\soft)}+\dfrac{|Q_\stiff|}{|\Gamma|}\biggr\},\quad z\in{\mathbb C}\setminus{\mathbb R},\quad \tau\in Q'.
\end{align*}
We shall now bound 
the right-hand side of
 (\ref{K_difference}) in terms of $\varepsilon$ and express $\Upsilon^\varepsilon_z-\widetilde{\Upsilon}^\varepsilon_z$ in terms of ${\mathfrak D},$ which will yield the required estimate (\ref{ups_est}).

Proceeding with this plan, notice that  ${\mathfrak m}(\cdot, \tau)$ is an 
$R$-function for each $\tau\in Q',$ and $|\Im{\mathfrak m}(z, \tau)|/|\Im z|$ is estimated away from zero uniformly in  $z\in {\mathbb C}\setminus{\mathbb R},$ $\tau\in Q'.$ Indeed, for each $\tau\in Q'$ denote by ${\mathfrak E}_\tau$ the spectral measure of the operator $A_0^\soft$ and by $\nu_\tau$ the scalar measure on ${\mathbb R}$ defined on Borel sets $B$ by the formula 
\[
\nu_\tau(B)=\bigl\langle{\mathfrak E}_\tau(B)\widetilde{\Psi}_0^{(\tau)}, \widetilde{\Psi}_0^{(\tau)}\bigr\rangle_{L^2(Q_\soft)}.
\] 
Assume first that $\Im z>0.$ 
Then by the spectral theorem (see {\it e.g.} \cite{Birman_Solomjak}) one has
\begin{align}
\Im{\mathfrak m}(z,\tau)-\frac{|Q_\stiff|}{|\Gamma|}\Im z
&=\Im\biggl(z\bigl\langle A^\soft_0(A_0^\soft-z)^{-1}\widetilde{\Psi}_0^{(\tau)},\widetilde{\Psi}_0^{(\tau)}\bigr\rangle_{L^2(Q_\soft)}\biggr)\nonumber\\[0.5em]
&
=\Im\biggl(z\int_{\mathbb R}\frac{\xi}{\xi-z}d\nu_\tau(\xi)\biggr)
=\Im z\int_{\mathbb R}\frac{\xi^2}{|\xi-z|^2}d\nu_\tau(\xi)>0.
\label{Im_est}
\end{align}
Similarly, in the case when $\Im z<0$ we obtain 
\begin{equation*}
\Im{\mathfrak m}(z,\tau)<\frac{|Q_\stiff|}{|\Gamma|}\Im z,
\label{Im_est_neg}
\end{equation*}
and combining this with (\ref{Im_est}) 
yields 
\begin{equation}
\bigl|\Im{\mathfrak m}(z,\tau)\bigr|>\frac{|Q_\stiff|}{|\Gamma|}|\Im z|,\quad z\in{\mathbb C}\setminus{\mathbb R},\quad \tau\in Q',
\label{Im_est_final}
\end{equation}
as claimed.

Next, bearing in mind that  $\bigl\langle\Lambda^{\rm soft}\psi_0, \psi_0\bigr\rangle_{\mathcal H}$ and $\varepsilon^{-2}\mu_*\tau\cdot\tau$ are real-valued for all $z,\tau,$ we estimate ${\mathfrak D}$ 
as follows, for some $C>0:$
\begin{align*}
\bigl\vert{\mathfrak D}(z, \tau)\bigr\vert&\le C|\tau|\bigl|\Im\,{\mathfrak m}(z,\tau)\bigr|^{-1}\bigl\vert{\mathfrak m}(z,0)+\varepsilon^{-2}\mu_*\tau\cdot\tau\bigr\vert^{-1}\\[0.4em]
&\le C|\tau|\bigl|\Im\,{\mathfrak m}(z,\tau)\bigr|^{-1}\Bigl\{\bigl(\Im{\mathfrak m}(z,0)\bigr)^2+\bigl(\Re {\mathfrak m}(z,0)+\varepsilon^{-2}\mu_*\tau\cdot\tau\bigr)^2\Bigr\}^{-1/2},\\
&\le C\alpha|\tau|\Bigl\{\alpha^{-2}+\bigl(\Re {\mathfrak m}(z,0)+\varepsilon^{-2}\mu_*\tau\cdot\tau\bigr)^2\Bigr\}^{-1/2}=:\widetilde{\chi}(\varepsilon, z, \tau),
\qquad z\in {\mathbb C}\setminus{\mathbb R},\quad\tau\in Q'.
\end{align*}
where ({\it cf.} (\ref{Im_est_final}))
$\alpha:=|\Gamma|/(|Q_\stiff||\Im z|).$
Setting additionally $\widetilde{\chi}(\varepsilon, z, \tau)=0$ whenever $z\in{\mathbb C}\setminus{\mathbb R},$ $\tau\in{\mathbb R}^d\setminus Q',$ we infer that 
\begin{equation}
\bigl\vert{\mathfrak D}(\tau)\bigr\vert\le 
\widetilde{\chi}(\varepsilon, z, \tau)\qquad \forall  z\in {\mathbb C}\setminus{\mathbb R},\ \ \tau\in{\mathbb R}^d.
\label{Dest}
\end{equation}


We proceed by using the fact that 
\begin{equation}
\sup_{z\in K_\sigma}\sup_{\tau\in{\mathbb R}^d}\widetilde{\chi}(\varepsilon, z, \tau)=O(\varepsilon).
\label{chi_sup}
\end{equation} 
Indeed, first, following the same argument as for (\ref{Im_est}), we write
\[
\Re{\mathfrak m}(z,0)=\Re z\biggl(\int_{\mathbb R}\frac{\xi^2}{|\xi-z|^2}d\nu_0(\xi)\biggr)-|z|^2\int_{\mathbb R}\frac{\xi}{|\xi-z|^2}d\nu_0(\xi)+\frac{|Q_\stiff|}{|\Gamma|}\Re z.
\]
Furthermore, by direct calculation we infer
\[
\sup_{\xi\in{\mathbb R}_+}\frac{\xi^2}{|\xi-z|^2}=\left\{\begin{array}{ll}\dfrac{|z|^2}{(\Im z)^2},\quad \Re z>0,\\[0.9em]
1,\quad \Re z\le0.
\end{array}\right.
\]
Hence, using the fact that ${\rm supp}(\nu_0)\subset{\mathbb R}_+,$ we obtain
\begin{equation*}
\bigl\vert\Re{\mathfrak m}(z,0)\bigr\vert
\le\vert\Re z\vert\biggl(\dfrac{|z|^2}{(\Im z)^2}+\frac{|Q_\stiff|}{|\Gamma|}\biggr),
\end{equation*}
which is bounded on $K_\sigma.$ Finally, for each $z$ consider the constant
\[
c_*=c_*(z):=\max_{|t|=1}
\sqrt{2\frac{\sup_{z\in K_\sigma}\bigl\vert\Re{\mathfrak m}(z, 0)\bigr\vert}{|\mu_*t\cdot t|}},
\] 
so that for all $t\in{\mathbb R}^d$ with $|t|\ge c_*$ one has $\bigl|\Re{\mathfrak m}(z, 0)\bigr|\le|\mu_*t\cdot t|/2$ for all $z\in K_\sigma.$
Then for $z\in K_\sigma$ and $\tau\in Q'$ such that $|\tau|\ge (c_*+1)\varepsilon,$ respectively $|\tau|\le(c_*+1)\varepsilon,$  one has the bound
\begin{align*}
\widetilde{\chi}(\varepsilon, z, \tau)
&\le\frac{C\alpha|\tau|}{\bigl|\Re {\mathfrak m}(z,0)+\varepsilon^{-2}\mu_*\tau\cdot\tau\bigr|}\le\frac{C\alpha|\tau|}
{\bigl|\varepsilon^{-2}\mu_*\tau\cdot\tau\bigr|-\bigl|\Re {\mathfrak m}(z,0)\bigr|}\\[0.4em]
&\le\frac{2C\alpha|\tau|}{\varepsilon^{-2}|\mu_*\tau\cdot\tau|}\le\frac{2C\alpha\varepsilon^2}{\Vert\mu_*^{-1}\Vert^{-1}|\tau|}\le \frac{2C\alpha\Vert\mu_*^{-1}\Vert\varepsilon}{c_*+1},
\end{align*}
respectively 
\[
\widetilde{\chi}(\varepsilon, z, \tau)
\le C\alpha^2|\tau|\le C\alpha^2(c_*+1)\varepsilon.
\]
Finally, for $z\in K_\sigma,$ $\tau\in{\mathbb R}^d\setminus Q',$ we have $\widetilde{\chi}(\varepsilon, z, \tau)=0,$ which concludes the proof of (\ref{chi_sup}).

For all $F\in L^2({\mathbb R}^d),$ using the Parseval identity, the bound (\ref{Dest}), and then the Parseval identity once again, we obtain 
\begin{equation*}
\bigl\Vert(\Upsilon^\varepsilon_z-\widetilde{\Upsilon}^\varepsilon_z)F\bigr\Vert_{L^2({\mathbb R}^d)}=\bigl\Vert{\mathfrak D}(z, \varepsilon\cdot)\widehat{F}\bigr\Vert_{L^2({\mathbb R}^d)}\le\Vert\widehat{F}\Vert_{L^2({\mathbb R}^d)}\sup_{z\in K_\sigma}\sup_{\tau\in{\mathbb R}^d}\widetilde{\chi}(\varepsilon, z, \tau)=\Vert F\Vert_{L^2({\mathbb R}^d)}\sup_{z\in K_\sigma}\sup_{\tau\in{\mathbb R}^d}\widetilde{\chi}(\varepsilon, z, \tau).
\end{equation*}
Combining this with (\ref{chi_sup}) yields
$
\bigl\Vert(\Upsilon^\varepsilon_z-\widetilde{\Upsilon}^\varepsilon_z)F\bigr\Vert_{L^2({\mathbb R}^d)}
\le\widetilde{C}\varepsilon\Vert F\Vert_{L^2({\mathbb R}^d)},$ where 
the constant
$\widetilde{C}>0$ is independent of $F,$  
as claimed.
\end{proof}

In the next lemma we extend the integration in $\tau$ in the formula (\ref{Ups_00}) to the whole space ${\mathbb R}^d,$ thus eliminating the dependence on $\varepsilon$ in the asymptotics.
\begin{lemma}\label{10.2}
One has the estimate
$$
\bigl\Vert\widetilde{\Upsilon}^\varepsilon_z-\Upsilon^0_z\bigr\Vert_{L^2({\mathbb R^d})\to L^2({\mathbb R^d})}=O(\varepsilon^2),\quad\varepsilon\to0,
$$
where ({\it cf.} (\ref{Uform})), for all $F\in L^2({\mathbb R}^d),$
\begin{equation}
(\Upsilon^0_zF)(x):=(2\pi)^{-d/2}\int_{\mathbb R^d}\bigl(\widetilde{K}_{\rm I}(\varepsilon t, z)-z\bigr)^{-1}\widehat{F}(t)\exp({\rm i}t\cdot x)dt,\qquad x\in{\mathbb R}^d,\quad \varepsilon>0,
\label{Ups_0}
\end{equation}
with $\widetilde{K}_{\rm I}(\tau, z)$ defined by (\ref{eq:Ktilde}) for all values of $\tau\in{\mathbb R}^d.$ 
\end{lemma}
\begin{proof}
We follow the strategy of the proof of Lemma \ref{sin_label}. Notice that for  $\varepsilon>0,$ $z\in K_\sigma,$ and $x\in{\mathbb R}^d$ one has
\begin{equation}
\begin{aligned}
(\widetilde{\Upsilon}^\varepsilon_zF)(x)-(\Upsilon^0_zF)(x)&=-(2\pi)^{-d/2}\int_{{\mathbb R}^d\setminus{\varepsilon^{-1}Q'}}\bigl(\widetilde{K}_{\rm I}(\varepsilon t, z)-z\bigr)^{-1}\widehat{F}(t)\exp({\rm i}t\cdot x)dt\\[0.4em]
&=-(2\pi)^{-d/2}\int_{{\mathbb R}^d}\chi(t)\widehat{F}(t)\exp({\rm i}t\cdot x)dt,
\end{aligned}
\end{equation}
where 
\begin{equation*}
\chi(t):=\left\{\begin{array}{ll}\bigl(\widetilde{K}_{\rm I}(\varepsilon t, z)-z\bigr)^{-1}=\dfrac{|Q_{\rm stiff}|}{|\Gamma|}\Bigl(\mu_*t\cdot t
+O\bigl(|z|\bigr)\Bigr)^{-1},\qquad t\in{\mathbb R}^d\setminus{\varepsilon^{-1}Q'},\\[0.6em]
0,\qquad  t\in\varepsilon^{-1}Q'.\end{array}\right.
\end{equation*}
Proceeding as in the proof of Lemma \ref{sin_label} via the Parceval identity, we conclude that
\begin{equation}
\bigl\Vert(\widetilde{\Upsilon}^\varepsilon_z-\Upsilon^0_z)F\bigr\Vert_{L^2({\mathbb R}^d)}=\Vert\chi\widehat{F}\Vert_{L^2({\mathbb R}^d)}\le\sup_{t\in{\mathbb R}^d}\bigl|\chi(t)\bigr|\Vert\widehat{F}\Vert_{L^2({\mathbb R}^d)}=\sup_{t\in{\mathbb R}^d}\bigl|\chi(t)\bigr|\Vert F\Vert_{L^2({\mathbb R}^d)}.
\end{equation}
Therefore, the following bound holds:
\begin{equation}
\begin{aligned}
\bigl\Vert(\widetilde{\Upsilon}^\varepsilon_z-\Upsilon^0_z)F\bigr\Vert_{L^2({\mathbb R}^d)}&\le\dfrac{|Q_{\rm stiff}|}{|\Gamma|}\sup_{t\in{\mathbb R}^d\setminus{\varepsilon^{-1}Q'}}\Bigl\vert\mu_*t\cdot t+O\bigl(|z|\bigr)\Bigr\vert^{-1}\Vert F\Vert_{L^2({\mathbb R}^d)}
\\[0.5em]
&\le\dfrac{|Q_{\rm stiff}|}{\pi^2\Vert\mu_*^{-1}\Vert^{-1}|\Gamma|}\varepsilon^2\bigl(1+O(\varepsilon^2)\bigr)\Vert F\Vert_{L^2({\mathbb R}^d)},
\end{aligned}
\end{equation}
and the claim follows.
\end{proof}

Finally, introduce the ``Zhikov function" ({\it cf.} \cite{Zhikov2000})
\begin{align}
\mathfrak B(z)&:=z+z\dfrac{|\Gamma|}{|Q_\stiff|}\sum_{j=1}^\infty\frac{\lambda_j}{\lambda_j-z}\Bigl|\bigl\langle\widetilde{\Psi}_0,\varphi_j^{(0)}\bigr\rangle_{L^2(Q_\soft)}\Bigr|^2=z+\frac{z}{|Q_\stiff|}\sum_{j=1}^\infty\frac{\lambda_j}{\lambda_j-z}\biggl|\int_{Q_\soft}\varphi_j^{(0)}\biggr|^2\nonumber\\[0.4em]
&=\frac{z}{|Q_\stiff|}\biggl(1+z\sum_{j=1}^\infty(\lambda_j-z)^{-1}\Biggl|\int_{Q_\soft}\varphi_j^{(0)}\biggr|^2\biggr),\label{functionB}
\end{align}
so that\footnote{The formula (\ref{eq:Zhikov_function}) can be viewed as representing ``frequency conversion'': for each $t\in Q',$ the set of values $z$ for which $\widetilde {K}_{\rm I}(\e t,z)=z$ coincides with the set of solutions to $\mathfrak B (z)=-\vert Q_\stiff\vert^{-1}|\Gamma|\mu_*t\cdot t.$ Notice also that $\widetilde {K}_{\rm I}(\e t,z)$ does not actually depend on $\varepsilon.$}
\begin{equation}\label{eq:Zhikov_function}
\widetilde {K}_{\rm I}(\e t,z)=z-\vert Q_\stiff\vert^{-1}|\Gamma|\mu_*t\cdot t-\mathfrak B (z),
\end{equation}
and consider the operator $\mathfrak A^\hom:=-\nabla A^\hom\cdot \nabla$ in $L^2({\mathbb R}^d)$ with a constant symbol $A^\hom$ such that ({\it cf.} (\ref{BS_germ}), (\ref{mu_q_rel}))
\begin{equation}
\label{eq:hom_symbol}
A^\hom \xi\cdot \xi=-\vert Q_\stiff\vert^{-1}|\Gamma|\mu_*\xi\cdot\xi={\mathfrak q}_\xi,\qquad \xi\in{\mathbb R}^d,
\end{equation}
and hence
\[
\widetilde {K}_{\rm I}(\e t,z)=z+A^\hom t\cdot t-\mathfrak B (z),\qquad t\in{\mathbb R}^d.
\]
Therefore, for all 
$F\in L^2({\mathbb R}^d),$ one has 
\begin{align*}
\bigl(\bigl(\mathfrak A^\hom - \mathfrak B(z)\bigr)^{-1}F\bigr)(x)&=(2\pi)^{-d/2}\int_{\mathbb R^d}\bigl(A^\hom t\cdot t-{\mathfrak B}(z)\bigr)^{-1}\widehat{F}(t)\exp({\rm i}t\cdot x)dt
\\[0.5em]
&=(2\pi)^{-d/2}\int_{\mathbb R^d}\bigl(\widetilde{K}_{\rm I}(\varepsilon t, z)-z\bigr)^{-1}\widehat{F}(t)\exp({\rm i}t\cdot x)dt,\qquad x\in{\mathbb R}^d.
\end{align*}
Taking into account Corollary \ref{cor:pseudodifferential}, the following result has thus been proved.
\begin{theorem}
\label{thm:ModelI_final_result}
In the case of Model I, one has the following estimate:
\begin{equation}
\label{eq:modelI_fin}
\Bigl\Vert P^\varepsilon_{\rm stiff}\bigl(A_\varepsilon-z)^{-1}\bigr\vert_{L^2(\Omega^\varepsilon_\stiff)}-\Theta^\e_\hom\bigl(\mathfrak A^\hom - \mathfrak B(z)\bigr)^{-1}(\Theta^\e_\hom)^*\Bigr\|_{L^2(\Omega_\stiff^\varepsilon)\to L^2(\Omega_\stiff^\varepsilon)}\leq C\e^{2/3}.
\end{equation}
Here $\Omega_\stiff^\varepsilon$ is the stiff component of the original composite, see (\ref{eq:generic_hom})--(\ref{weight}), $P^\varepsilon_\stiff$ is the orthogonal projection of $L^2({\mathbb R}^d)$ onto $L^2(\Omega^\varepsilon_\stiff),$ $\mathfrak A^\hom$ is an elliptic operator in $L^2(\mathbb R^d)$ with the constant symbol \eqref{eq:hom_symbol}, the function $\mathfrak B(z)$ is defined by \eqref{functionB}, and $\Theta^\e_\hom$ is a unitary operator from
$L^2({\mathbb R}^d)$ to $L^2(\Omega_\stiff^\varepsilon).$

\end{theorem}

\begin{remark}
Note that an estimate of this type can be obtained by Theorem \ref{thm:general_homo_result} with a better error than $O(\e^{2/3})$. In this case, however, the operator $\mathfrak A^\hom$ may no longer be shown to be elliptic, see, e.g., \cite{Physics} where it is shown that spatial dispersion ({\it i.e.} spatial non-locality) may appear in the description of $\mathfrak A^\hom$.
\end{remark}

\subsubsection{Model II}
In the case of Model II, there appears to be no obvious way to proceed beyond the statement of Theorem \ref{thm:pseudodifferential} in the analysis of $({A}_\e^{(\tau)}-z)^{-1}$ sandwiched onto the stiff component of the medium. In particular cases this can be accommodated, however, see, {\it e.g.}, Examples (0) and (2) of \cite{GrandePreuve}, where a related ODE setup leads to a finite-difference frequency-dispersive homogenised operator on the stiff component. This owes to the explicit dependence of $K_{\rm II}(\tau,z)$ on $\tau$ in models stemming from ODEs. (In a nutshell, the function $K_{\rm II}$ is shown to depend explicitly on $\cos \tau$ only.) 

We would like to also point out a relation between the approach and the results of the present paper pertaining to Model II and the setup of \cite{Capdeboscq2007} (see also references therein), which is, however, beyond the scope of the present paper and shall be scrutinised elsewhere.

\section{Concluding remarks}
\label{concluding_remarks}

Here we indicate the applicability of our argument to the more general case, where the symbol $a^2(\cdot/\e)$ of ${A}_\e$ is no longer constant on the stiff component $Q_\stiff$ of the medium. It is of course questionable whether this general setup is appealing enough from the point of view of applications in materials science to vindicate its in-depth study, since it requires to select a rather exotic medium with material parameters depending on the spatial variable on the stiff component, while at the same time the major influence on the results is surely exerted by the presence of the soft component, as evident from in particular the results of Section \ref{time_disp_section} above.

The material of Sections \ref{norm_resolvent_section} and \ref{section:4}, up to and including Theorem \ref{thm:general_homo_result}, is formulated in an abstract way and requires no amendments, provided that the DN maps introduced in Section \ref{setup_section} are redefined accordingly. In particular, the co-normal derivative $\dntau$ pertaining to the stiff component of the medium $Q_\stiff$ is defined as $\dntau u=-a^2(\nabla u+{\rm i}\tau u)\cdot n|_{\Gamma}.$

The analysis of Section \ref{section:stability} stays intact after an obvious modification to the definition of the operator $\Lambda^\stiff$ based on the above re-definition of the operator $\dntau,$ and provided the asymptotic analysis of the DN map $\Lambda^\stiff$ (and in particular of its least eigenvalue) is carried out. We shall address this in the second part of the paper \cite{ChEK_future}. In particular, Theorem \ref{thm:stability_NRA_fin} continues to hold, pending the necessary modification of the one-dimensional operator $\Lambda_\Delta.$

Finally, the analysis of Section \ref{time_disp_section} continues to hold with no modifications.  

\section*{Acknowledgements}
KDC and YYE are grateful for the financial support of
the Engineering and Physical Sciences Research Council: Grant EP/L018802/2 ``Mathematical foundations of metamaterials: homogenisation, dissipation and operator theory''. AVK has been partially supported by the  RFBR grant 16-01-00443-a and the Russian Federation Government megagrant 14.Y26.31.0013. YYE has also been supported by the RFBR grant 19-01-00657-a. 

We express deep gratitude to Professor Sergey Naboko for fruitful discussions. We also thank the referees for a very thorough review of the manuscript and a number of comments that helped improving it.

\end{document}